\documentclass[a4paper]{amsart}

\usepackage[utf8]{inputenc}
\usepackage[T1]{fontenc}
\usepackage{lmodern, enumerate}
\usepackage{amssymb,amsxtra,amscd,amsmath}
\usepackage{amsthm}
\usepackage{amsfonts}
\usepackage{bbm}
\usepackage[all]{xy}
\usepackage{xcolor}
\usepackage{nicefrac,mathtools}
\usepackage{microtype}
\usepackage[normalem]{ulem}
\usepackage[pdftitle={Deformation via coactions},
 pdfauthor={Alcides Buss, Siegfried Echterhoff},
 pdfsubject={Mathematics}
]{hyperref}
\usepackage[lite]{amsrefs}

\newcommand*{\arxiv}[1]{\href{http://www.arxiv.org/abs/#1}{arXiv: #1}}
\usepackage{comment}

\makeatletter
\@namedef{subjclassname@2020}{%
  \textup{2020} Mathematics Subject Classification}
\makeatother

\numberwithin{equation}{section}
\theoremstyle{plain}
\newtheorem{theorem}[equation]{Theorem}
\newtheorem{lemma}[equation]{Lemma}
\newtheorem{proposition}[equation]{Proposition}
\newtheorem{corollary}[equation]{Corollary}
\theoremstyle{definition}
\newtheorem{definition}[equation]{Definition}
\theoremstyle{remark}
\newtheorem{remark}[equation]{Remark}
\newtheorem{example}[equation]{Example}

\DeclareMathOperator{\Aut}{Aut}
\DeclareMathOperator{\cspn}{\overline{span}}

\DeclareMathOperator{\id}{\mathrm{id}}

\newcommand*{\Prim}{\mathrm{Prim}}

\newcommand*{\SL}{\mathrm{SL}} 
\newcommand*{\PSL}{\mathrm{PSL}} 
 
\newcommand*{\SO}{\mathrm{SO}} 
\newcommand*{\SU}{\mathrm{SU}} 
\newcommand*{\tg}{\mathrm{tg}}

\newcommand*{\nb}{\nobreakdash}
\newcommand*{\Star}{\(^*\)\nobreakdash-}
\newcommand*{\dd}{\,\mathrm d} 
\newcommand*{\R}{\mathbb R}
\newcommand*{\T}{\mathbb T}
\newcommand*{\Z}{\mathbb Z}
\newcommand*{\N}{\mathbb N}
\newcommand*{\C}{\mathbb C}

\renewcommand*{\L}{\mathcal L}
\newcommand*{\K}{\mathcal K}
\renewcommand*{\H}{\mathcal H}

\newcommand*{\F}{\mathcal F} 
\newcommand*{\cont}{C}
\newcommand*{\contz}{\cont_0}
\newcommand*{\contc}{\cont_c}
\newcommand*{\M}{\mathcal M}

\newcommand{\brg}{\mathrm{Br}(G)}

\newcommand{\tildeV}{\widetilde{V}}
\newcommand{\tig}{\tilde{g}}

\newcommand*{\Ad}{\textup{Ad}}

\newcommand*{\U}{\mathcal U}
\newcommand*{\E}{\mathcal E}
\newcommand*{\X}{\mathcal X}
\newcommand{\rt}{\mathrm{rt}}

\newcommand{\Br}{\operatorname{Br}}

\newcommand*{\defeq}{\mathrel{\vcentcolon=}}
\newcommand*{\congto}{\xrightarrow\sim}

\newcommand*{\braket}[2]{\langle#1\!\mid\!#2\rangle}

\newcommand*{\sbe}{\subseteq} 

\newcommand*{\cstar}{\texorpdfstring{$C^*$\nobreakdash-\hspace{0pt}}{*-}}
\newcommand*{\into}{\hookrightarrow}
\newcommand*{\onto}{\twoheadrightarrow}
\newcommand*{\red}{r}
\renewcommand*{\max}{\mathrm{max}}

\newcommand*{\dual}[1]{\widehat{#1}}
\newcommand*{\dualG}{\widehat{G}}

\newcommand{\om}{\omega}
\newcommand{\Om}{\Omega}
\newcommand{\PU}{\mathcal PU}
\newcommand{\Twist}{\mathrm{Twist}}

\newcommand*{\ev}{\mathrm{ev}}

\newcommand{\lk}{\langle}
\newcommand{\rk}{\rangle}
\newcommand{\csp}{\overline{\operatorname{span}}}

\renewcommand{\top}{\mathrm{top}}

\newcommand{\car}{\curvearrowright}

\newcommand{\A}{\mathcal A}
\newcommand{\B}{\mathcal B}
\newcommand{\G}{\mathcal G}
\newcommand{\hatG}{\widehat{G}}

\newcommand\Hom{\operatorname{Hom}}

\begin{document}

\title[A new approach to deformation of C*-algebras]{A new approach to deformation of C*-algebras via coactions}

\author{Alcides Buss}
\email{alcides.buss@ufsc.br}
\address{Departamento de Matem\'atica\\
 Universidade Federal de Santa Catarina\\
 88.040-900 Florian\'opolis-SC\\
 Brazil}

\author{Siegfried Echterhoff}
\email{echters@uni-muenster.de}
\address{Mathematisches Institut\\
Universit\"at M\"un\-ster\\
 Einsteinstr.\ 62\\
 48149 M\"unster\\
 Germany}

\begin{abstract}
We revisit the procedure of deformation of \cstar{}algebras via coactions of locally compact groups and extend the methods to cover deformations for maximal, reduced, and  exotic coactions for a given group $G$ and circle-valued Borel $2$-cocycles on $G$. In the special case of reduced (or normal) coactions our deformation method substantially differs from -- but turns out to be equivalent to -- the ones used by previous authors, specially those given by Bhowmick, Neshveyev, and Sangha in \cite{BNS}.

Our approach yields all expected results, like a good behaviour of deformations under nuclearity, continuity of fields of C*-algebras and $K$-theory invariance under mild conditions.
\end{abstract}



\thanks{This work was funded by: the Deutsche Forschungsgemeinschaft (DFG, German Research Foundation) Project-ID 427320536 SFB 1442 and under Germany's Excellence Strategy EXC 2044  390685587, Mathematics M\"{u}nster: Dynamics, Geometry, Structure; and CNPq/CAPES/Humboldt - Brazil.}
\maketitle

\tableofcontents

\section{Introduction}
In his pioneering work \cites{Rieffel:Deformation, Rief-K}   Rieffel introduced the concept of a deformation of a \cstar{}algebra $A$  equipped 
with an action $\alpha:\R^d\car A$ of the vector group $\R^d$ by some real skew-symmetric $d\times d$-matrix $\Theta\in \mathbb{M}_d(\R)$ by directly constructing a deformed multiplication using oscillatory integrals. A different approach was later given by Kasprzak  in \cite{Kasprzak:Rieffel} by describing deformations 
of a \cstar{}algebra $A$ equipped with an action $\alpha:\widehat{G}\car A$ of the Pontrjiagin dual $\widehat{G}$ of a locally compact abelian group $G$ via  continuous $2$-cocycles $\omega\in Z^2(G,\T)$ (here we switched the roles of $G$ and $\widehat{G}$ to be consistent with what follows below). The idea is easy to explain: We consider the crossed-product \cstar{}algebra $B:=A\rtimes_{\alpha}\widehat{G}$ together with the dual action $\beta:=\widehat{\alpha}:G\car B$. Identifying  $C^*(\widehat{G})$ with $C_0(G)$ via Fourier transform, the canonical representation 
$\iota_{\widehat{G}}:\widehat{G}\to U\M(B)$ integrates to give a nondegenerate \Star{}homomorphism $\phi:C_0(G)\to \M(B)$.
Now, a continuous cocycle  $\omega:G\times G\to \T$ determines a map $g\mapsto \omega(\cdot,g)\in C(G,\T) 
=U\M(C_0(G))$. Composing this with $\phi$ we obtain a map $U^\omega:G\to U\M(B); g\mapsto U^\omega_g=\phi(\overline{\omega}(\cdot,g))$. Using the cocycle condition on $\omega$, one can check that $U^\omega$ is an `$\om$-twisted' $1$-cocycle for $\beta$, so that 
$$\beta^\omega:G\car B,\quad \beta^\omega_g=\Ad U^{\omega}_g\circ\beta_g$$
becomes a new  action of $G$ on $B$ for which  $\phi:C_0(G)\to \M(B)$ is $\rt-\beta^\omega$ equivariant, where $\rt_g(f)(h)=f(hg)$ denotes the right translation action of $G$ on $C_0(G)$. It follows then from Landstad duality for actions of abelian groups (e.g., see \cite{Landstad:Duality}) that there exists a 
unique (up to isomorphism) \cstar{}algebra $A^\omega$ together with an action $\alpha^\omega:G\car A^\omega$ such that 
$$(B, \beta^\omega,\phi)\cong  (A^\omega\rtimes_{\alpha^\omega}\widehat{G}, \widehat{\alpha^\omega},\phi^\om).$$
 If $\omega$ is trivial, then 
we recover the original system $(A,G,\alpha)$.
It is shown in \cite{BNS} that in case of $G=\R^n$ this reduces to the original Rieffel deformation. 

Kasprzak's approach to \cstar{}deformations has been extended in various directions to general locally compact groups (\cites{Kasprzak1,  BNS}) and  to locally compact quantum groups (\cite{NT}). Indeed, in \cite{NT} Neshveyev and Tuset provide a general approach to deform locally compact quantum groups $\mathbb{G}$ and their (reduced) coactions using quantum versions of measurable $2$-cocycles as certain (unitary) elements of the von Neumann algebra $\Omega\in L^\infty(\mathbb{G})\bar\otimes L^\infty(\mathbb{G})$ associated with $\mathbb{G}$. This general approach is compatible with deformation of locally compact group coactions from \cites{Kasprzak1, BNS}, and therefore also with ours. Deformation via $2$-cocycles is also related to the theory of Galois objects for locally compact quantum groups, as developed by De Commer \cite{DCo}, see also previous works by Baaj and Crespo \cite{BCr} and De Rijdt and Vander Vennet \cite{DRVV}. 

In this paper we only deal with locally compact groups and their (co)actions. For a  non-abelian locally compact group $G$, the analogue of an action of the (nonexistent) dual group 
 $\widehat{G}$  on a \cstar{}algebra $A$ is given by a coaction $\delta:A\to \M(A\otimes C^*(G))$. The crossed product 
 $B=A\rtimes_{\delta}\widehat{G}$ then comes with a canonical nondegenerate inclusion $\phi:C_0(G)\to \M(B)$ and a dual action 
 $\beta:=\widehat{\delta}:G\car B$. Similar as in Kasprzak's approach, given a 
 {\em continuous} cocycle $\omega:G\times G\to \T$ on $G$ we may construct the deformed 
 action  $\beta^\omega:G\car B$ and use Landstad duality to obtain the deformed coaction $\delta^\omega:A^\omega\to \M(A^\omega\otimes C^*(G))$ such that 
 $$(B, \beta^\omega,\phi)\cong  (A^\omega\rtimes_{\delta^\omega}\widehat{G}, \widehat{\delta^\omega}, \phi^\om).$$
This approach has been followed by Bhowmick, Neshveyev, and Sangha in \cite{BNS} who also extended the method to cover 
  deformations by possibly non-continuous Borel cocycles. But \cite{BNS} only covers the case of 
{\em normal} (or reduced) coactions, i.e., coactions for which the composition $(\id_A\otimes\lambda)\circ \delta$ is faithful on $A$,
where $\lambda:C^*(G)\to C_r^*(G)$ denotes the (integrated form of the)  left regular representation of $G$. To see that this leads to severe restrictions, let us look at the case of dual coactions $\delta=\widehat{\beta}$ on maximal crossed products
$A:=B\rtimes_{\beta,\max}G$.
The natural candidate for the deformed algebra $A^\omega$ in this situation should be the crossed product 
$B\rtimes_{(\beta,\iota^\om),\max}G$ twisted by $\om$. But the method of \cite{BNS} only applies to dual coactions of reduced crossed products and hence 
does not cover this obvious situation. The reason for this lies in the fact that the main computations are 
performed in a faithful representation of the 
reduced double crossed product $A\rtimes_{\delta}\widehat{G}\rtimes_{\widehat{\delta},r}G$ inside $\mathcal M\big(A\otimes \K\otimes \K\big)$ with
$\K:=\K(L^2(G))$, 
which restricts to a faithful representation of $A$ if and only if  $\delta$ is reduced (or {\em normal} in the notation introduced by Quigg in \cite{Quigg}). More generally, as we shall see below, a general coaction $(A,\delta)$ of $G$ corresponds to some exotic crossed product $\rtimes_\mu$ which lies between the maximal and reduced crossed products in a certain sense (see \cites{BGW, BEW,BEW2} for a general treatment), so a complete picture should cover such situations as well: in particular for a (``nice'') exotic crossed product $B\rtimes_{\beta,\mu}G$ carrying a dual coaction, we would like to have the twisted exotic crossed product $B\rtimes_{(\beta,\iota^\omega),\mu}G$ as its deformed algebra, endowed with the associated dual coaction.

The main goal of this paper is to overcome the above problems and obtain some of the expected deformation results by using the theory of generalized fixed-point algebras 
and Landstad duality as introduced by the authors in \cite{Buss-Echterhoff:Exotic_GFPA}.
Even if we stick to normal coactions, our approach leads to a new picture of the deformed coactions.

Our general deformation procedure starts with a coaction  $(A,\delta)$ of $G$ which we assume to satisfy Katayama duality with respect to some exotic crossed product $\rtimes_\mu$, i.e., there is a 
canonical isomorphism 
$$A\rtimes_{\delta}\widehat{G}\rtimes_{\widehat\delta,\mu}G\cong A\otimes \K(L^2(G)).$$
We then say that $(A,\delta)$ is a $\mu$-coaction.
As explained above, the  crossed product $B:=A\rtimes_\delta\widehat{G}$ then comes with a dual action $\beta:=\widehat{\delta}$ together with a structure map $\phi:C_0(G)\to \M(B)$ which is $\rt-\beta$ equivariant for the right translation action $\rt:G\car C_0(G)$. Such triple 
$(B,\beta,\phi)$ is called a {\em weak $G\rtimes G$-algebra}.
Exotic Landstad duality allows us to recover the coaction $(A,\delta)$ from the weak $G\rtimes G$-algebra $(B,\beta,\phi)$ and the crossed product $B\rtimes_\mu G$.  Now any `deformation' $\beta^\om$ of the action $\beta$ to a new action  such that $\phi:C_0(G)\to \M(B)$ remains to be $\rt-\beta^\om$ equivariant (e.g., via a continuous Borel cocycle $\om$ as above) gives a new 
weak $G\rtimes G$-algbra structure $(B,\beta^\om,\phi)$ such that Landstad duality provides a deformed $\mu$-coaction $(A_\mu^\om,\delta_\mu^\om)$ of the original $\mu$-coaction $(A,\delta)$. 

The deformation procedure by Borel $2$-cocycles $\omega$ on $G$ is  more complicated. 
Our approach is better explained in terms of group extensions: given a central extension of locally compact groups 
$$\sigma=(\T\stackrel{\iota}\into  G_\sigma \stackrel{q}\onto G)$$
of $G$ by the circle group $\T$,  we look at its associated complex line bundle and take its $\contz$-sections, which we write as $\contz(G_\sigma,\iota)$. This is an imprimitivity $\contz(G)-\contz(G)$ bimodule, and it implements an equivalence between the ordinary (right) translation $G$-action on the left $\contz(G)$-module structure and a ``twisted'' translation $G$-action on the right $\contz(G)$-module structure. Every such extension corresponds to a $G$-action on the \cstar{}algebra $\K$ of compact operators on a Hilbert space. This action is unique up to Morita equivalence and determines an isomorphism from the group $\Twist(G)$ of isomorphism classes $[\sigma]$ of extensions $\sigma$ 
of $G$ as above, and the Brauer group $\Br(G)$ of Morita equivalence classes $[\alpha]$ of actions of $\alpha:G\car \K(\H)$. 
Indeed, there  are canonical group isomorphisms $\Twist(G)\cong H^2(G,\T)\cong \Br(G)$, and 
any of these can serve as deformation data for a $G$-coaction or, equivalently, a weak $G\rtimes G$-algebra $(B,\beta,\phi)$. 

Given an element $[\sigma]\in \Twist(G)$ represented by an extension $\sigma=(\T\into G_\sigma\onto G)$ 
and  a weak $G\rtimes G$-algebra $(B,\beta,\phi)$, we consider the Hilbert $B$-module
$$\L(G_\sigma, B):=\contz(G_\sigma,\iota)\otimes_\phi B.$$
We then prove that the \cstar{}algebra $B_\sigma:=\K(\L(G_\sigma, B))$ carries a canonical action $\beta_\sigma:G\car B_\sigma$ and a structure map
$\phi_\sigma:C_0(G)\to \M(B_\sigma)$ such that the triple $(B_\sigma,\beta_\sigma,\phi_\sigma)$ 
 becomes a weak $G\rtimes G$-algebra. So we can apply (exotic) Landstad duality to obtain the deformed $\mu$-coaction $(A^\sigma_\mu,\delta^\sigma_\mu)$ if we start with the $\mu$-coaction $(A,\delta)$ and the weak $G\rtimes G$-algebra $(B,\beta,\phi)=(A\rtimes_{\delta}\dualG, \widehat{\delta}, j_{\contz(G)})$.

If $\omega$ is a continuous $2$-cocycle and $(A, \delta)$ satisfies Katayama duality for the reduced 
crossed product, then we show that the above deformation procedure covers all previous deformation approaches of 
\cites{Kasprzak:Rieffel, BNS} via cocycles. More precisely, we prove that $(A^\omega_\mu,\delta^\omega_\mu)$ is 
canonically isomorphic to $(A^\sigma_\mu,\delta^\sigma_\mu)$ if $\sigma=(\T\into G_\om\onto G)$ is the extension 
corresponding to $\om$ via the isomorphism $H^2(G,\T)\cong \Twist(G)$ mentioned above.
The same is true for the Bhowmick-Neshveyev-Sangha deformation by Borel cocycles, but the proof turned out to be very long and technical, so we postpone the details to some later publication. 
Different from \cite{BNS}, our approach does not use any separability assumptions on $A$ or $G$.
Using our approach to deformation, we prove some of the expected results. In particular we prove

\begin{theorem}[Nuclearity]
Suppose $(A,\delta)$ is a coaction  such that the dual action $\beta=\dual\delta$ on $B=A\rtimes_\delta\dualG$ is amenable (e.g.,~if $G$ is amenable). Then  $(A,\delta)$ is normal and, for every $\sigma\in \Twist(G)$, the deformed \cstar{}algebra $A^\sigma$ is nuclear if and only if $A$ is nuclear.
\end{theorem}

By a continuous family of twists $x\mapsto \sigma_x=(\T\into G_{\sigma_x}\onto G)$ over a locally compact space $X$, 
we shall understand a central extension 
$$\Sigma=(X\times\T\into \G\onto X\times G)$$
of the group bundle $X\times G$ over $X\times \T$ in the sense of groupoids (see the discussions in Section~\ref{sec:twists} below).

\begin{theorem}[Continuity]
Suppose that $X$ is a locally compact Hausdorff space and $(\sigma_x)_{x\in X}$ is a continuous family of twists over $X$. Then for every exact 
duality crossed-product functor $\rtimes_\mu$, the family of deformed coactions $(A^{\sigma_x}_\mu,\delta^{\sigma_x}_\mu)$ is upper semicontinuous in the sense that they can be realized as  fibres of a $G$-coaction on a $\contz(X)$-algebra $(A^\Sigma_\mu,\delta^\Sigma_\mu)$. 

If $G$ is exact, the bundle of normal deformed coactions $(A_{\red}^{\sigma_x},\delta_{\red}^{\sigma_x})$ corresponding to the reduced crossed-product functor $\rtimes_{\red}$ is continuous. 
\end{theorem}

As we shall explain below, a {\em duality} crossed-product functor is a functor which assigns to each action $\beta:G\car B$ a crossed product $B\rtimes_{\beta,\mu}G$ which admits a dual coaction $\widehat{\beta}_\mu$.
The following theorem requires $\rtimes_\mu$ to be a {\em correspondence crossed product functor} as in  Definition~\ref{def-correspondence} below. These cover in particular the maximal and reduced crossed products as well as the smallest exact Morita compatible crossed-product functor $\rtimes_{\epsilon}$ which appears in a new formulation of the Baum-Connes conjecture of \cite{BGW}. Note that every correspondence functor is also a duality functor.

\begin{theorem}[$K$-theory invariance]
Suppose that $A$ is separable and $\sigma_0$ and $\sigma_1$ are homotopic twists. If $G$ is second countable and satisfies 
the Baum-Connes conjecture with coeficients, then $$K_*(A^{\sigma_0}_r)\cong K_*(A^{\sigma_1}_r).$$
If, in addition, $G$ is $K$-amenable, then 
$K_*(A^{\sigma_0}_\mu)\cong K_*(A^{\sigma_0}_r)\cong K_*(A^{\sigma_1}_r)\cong K_*(A^{\sigma_1}_\mu)$ 
for every correspondence crossed-product functor $\rtimes_\mu$. 
\end{theorem}

Besides deformation via twists (or cocycles), we also consider a deformation procedure inspired by the deformation of Fell bundles
introduced by Abadie and Exel in \cite{Abadie-Exel:Deformation}. We shall recall the basic ingredients for this type of deformation, and will explain how it fits in our general approach to deformation of coactions. In the forthcoming paper 
\cite{BE:Fellbundles} we shall study more deeply the deformation by dual coactions  $(A,\delta)=(C_\mu^*(\A), \delta_\mu)$ of (exotic) cross-sectional algebras of Fell bundles $p:\A\to G$ over $G$.
In particular, we shall see that the above described deformations can all be obtained by direct deformations of the underlying Fell bundle $\A$.

This paper is organized as follows.
The necessary background on exotic crossed products and Landstad duality via generalized fixed-point algebras 
will be summarized in the preliminaries Section~\ref{sec:preliminaries}. 
In Section~\ref{sec:Landstad} we introduce our general approach for deformation of coactions and explain how 
Abadie-Exel deformations and deformations by continuous cocycles fit into this setting.
In Section~\ref{sec:twists} we recall the theory of central extensions by the circle group, and how this relates to actions on compact operators and $2$-cocycles. We then use this in Section~\ref{sec:deformation-twist} to deform coactions via twists (or by $2$-cocycles,  or by actions on compact operators). At the end of Section~\ref{sec:deformation-twist} we give the first applications of the theory, proving that our constructions behave well with respect to products of extensions, and with respect to nuclearity.
In the final Sections~\ref{sec:continuity} and~\ref{sec-K-theory} we give more elaborate applications of our theory, proving that our deformation approach also behaves well with respect to continuity of fields and $K$-theory.

\subsection*{Acknowledgement}  Most of this work has been written while the first author was visiting the University of Münster, and he would like to thank the whole group, specially the second author for the warm hospitality!

\section{Preliminaries}\label{sec:preliminaries}

\subsection{Exotic crossed products}
Let $\beta:G\car B$ be an action of the locally compact group $G$ on the \cstar{}algebra $B$, that is, a group homomorphism $\beta\colon G\to \Aut(B)$ such that for all $b\in B$ the map $t\mapsto \beta_t(b)$ is continuous from $G$ to $B$. By a {\em crossed product}
for $\beta:G\car B$ we understand any \cstar{}completion $B\rtimes_{\beta,\mu}G$
of the convolution algebra $C_c(G,B)$ 
with respect to any \cstar{}norm $\|\cdot\|_\mu$ such that for all $f\in C_c(G,B)$ we have
$$\|f\|_{\red}\leq \|f\|_\mu\leq \|f\|_\max,$$
where $\|\cdot\|_{\red}$ and $\|\cdot\|_{\max}$ denote the usual reduced and maximal crossed-product norms, which provide the usual reduced and maximal crossed products $B\rtimes_{\beta,\red}G$ and $B\rtimes_{\beta,\max}G$, respectively.
If $B\rtimes_{\beta,\mu}G$ differs from both of these, we call it an {\em exotic} crossed product. Note that the identity map on $C_c(G,B)$ induces surjective \Star{}homomorphisms
$$B\rtimes_{\beta,\max}G\onto B\rtimes_{\beta,\mu}G\onto B\rtimes_{\beta,r}G.$$

By a {\em crossed-product functor} we understand a choice of a crossed product $B\rtimes_{\beta,\mu}G$ 
for {\bf every} action $\beta\colon G\car B$, such that whenever we have a $\beta-\alpha$ 
equivariant \Star{}homomorphism $\Phi:B\to A$ for some action $\alpha:G\car A$,
then  the induced \Star{}homomorphism
$$\Phi\rtimes G: C_c(G,B)\to C_c(G, A); f\mapsto \Phi\circ f$$
extends to a \Star{}homomorphism $\Phi\rtimes_\mu G: B\rtimes_{\beta,\mu}G\to A\rtimes_{\alpha,\mu}G$.
It is well known that the maximal and reduced crossed products are crossed-product functors. 

Almost all of our constructions will require that a given crossed product functor is a {\em duality functor} in the sense that 
for every action $\beta:G\car B$ the dual coaction $\widehat{\beta}$ of $G$ on the maximal crossed product $B\rtimes_{\beta,\max}G$
factors through a coaction
$$\widehat{\beta}_\mu:B\rtimes_{\beta,\mu}G\to\M(B\rtimes_{\beta,\mu}G\otimes C^*(G)).$$
But for some of our purposes we even need stronger functoriality conditions. Recall that a $G$-equivariant correspondence between actions $(B,\beta)$  and $(A,\alpha)$ of $G$ consists of a pair $(\mathcal X,\gamma)$ such that $\mathcal X$ is a Hilbert $B$-module equipped with a $\beta$-compatible action $\gamma:G\car \mathcal X$ and a $G$-equivariant left action of $A$ on $\mathcal X$ by adjointable operators.

\begin{definition}\label{def-correspondence}
A crossed-product functor $\rtimes_\mu$  is a {\em correspondence crossed-product functor}
if for every $G$-equivariant correspondence $(\mathcal X,\gamma)$ between actions $(B,\beta)$  and $(A,\alpha)$ of $G$ the canonical $C_c(G, A)-C_c(G, B)$ bimodule $C_c(G, \X)$ 
completes to a $A\rtimes_{\alpha,\mu}G-B\rtimes_{\beta,\mu}G$ correspondence bimodule $\X\rtimes_{\gamma,\mu}G$.
\end{definition}

By \cite{BEW2}*{Theorem 4.14} we know that all correspondence functors are duality functors. Notice that a correspondence crossed-product functor in particular respects equivariant Morita equivalences: 
if $(\X,\gamma)$ is an $(A,\alpha)-(B,\beta)$ equivalence bimodule, then $\X\rtimes_{\gamma,\mu}G$ becomes an $A\rtimes_{\alpha,\mu}G-B\rtimes_{\beta,\mu}G$ equivalence bimodule.
We say $\rtimes_\mu$ is a \emph{Morita compatible} crossed-product functor when this happens. 

The maximal and the reduced crossed products are always correspondence crossed-product functors. The same holds for the smallest exact Morita compatible crossed-product functor $\rtimes_{\epsilon}$ of \cite{BGW}. 
Correspondence crossed-product functors  have been studied extensively in \cite{BEW}, where  it is  shown that 
for many non-amenable groups (like the free groups in $n$ generators with $n\geq 2$) there
 exist uncountably many distinct  correspondence functors.  
 
The above concepts can be extended to twisted actions, or more generally, to Fell bundles (see \cite{Buss-Echterhoff:Maximality}).
We essentially only use ordinary actions in this paper, but at some points we also comment how our results extend to Fell bundles and we also consider actions which are twisted by a circle-valued $2$-cocycle on the group.

\subsection{Coactions and Landstad duality}\label{subsec-Landstad}

Recall that a (nondegenerate) coaction of $G$ on a \cstar{}algebra $A$ is a \Star{}homomorphism $\delta\colon A\to \M(A\otimes C^*(G))$ 
such that $\cspn\delta(A)(1\otimes C^*(G))=A\otimes C^*(G)$ and which satisfies the identity 
$$(\delta\otimes \id_G)\circ \delta=(\id_A\otimes \delta_G)\circ \delta$$
where the {\em comultiplication}  $\delta_G: C^*(G)\to \M(C^*(G)\otimes C^*(G))$ is the integrated form
of the unitary homomorphism $g\mapsto u_g\otimes u_g$, where $u:G\to U\M(C^*(G))$ denotes the canonical inclusion. 

Given a coaction $(A,\delta)$, we define its {\em crossed product} as
$$A\rtimes_\delta \widehat{G}:=\csp\{\big((\id_A\otimes\lambda)\circ \delta(A)\big)(1\otimes M(C_0(G))\}\subseteq \M(A\otimes \K(L^2(G))$$
and write $j_A:=(\id_A\otimes\lambda)\circ \delta$ and $j_{C_0(G)}=1\otimes M$, where 
$M:C_0(G)\to \B(L^2(G))$ denotes the representation by multiplication operators. 
For a concise survey of the theory of (co)actions and their crossed products we refer to \cite{EKQR}*{Appendix A}. 

If $\rt: G\car C_0(G)$ denotes the action by right translations, then there is a {\em dual action} $\widehat{\delta}:G\car A\rtimes_\delta\widehat{G}$ such that  $\widehat\delta_g\big(j_A(a)j_{C_0(G)}(f)\big)=j_A(a)j_{C_0(G)}(\rt_g(f))$. In particular, we see that 
$j_{C_0(G)}:C_0(G)\to \M(A\rtimes_{\delta}\widehat{G})$ is a nongegenerate $\rt-\widehat\delta$ equivariant \Star{}homomorphism.
Recall from \cite{Nilsen}  that
there is always a canonical surjective \Star{}homomorphism
\begin{equation}\label{eq-bidualmap}\Psi: A\rtimes_{\delta}\widehat{G}\rtimes_{\widehat\delta,\max}G\onto A\otimes \K(L^2(G))
\end{equation}
which is given as the integrated form of 
$\big((\id_A\otimes \lambda)\circ \delta\rtimes (1\otimes M), 1\otimes \rho\big)$ where
$\rho: G\to \U(L^2(G))$ denotes the right regular representation of $G$.
A coaction $\delta$ is called {\em maximal} if $\Psi$ is faithful and it is called {\em normal} (or reduced) 
if  $\Psi$ factors through an isomorphism 
$A\rtimes_{\delta}\widehat{G}\rtimes_{\widehat\delta,r}G\cong A\otimes \K(L^2(G))$. 
In general, there is a unique  
\cstar{}norm $\|\cdot\|_\mu$ on $C_c(G, A\rtimes_{\delta}\widehat{G})$ 
such that $\Psi$ factors through an isomorphism 
\begin{equation}\label{eq-Kat}
\Psi_\mu: A\rtimes_{\delta}\widehat{G}\rtimes_{\widehat\delta,\mu}G\congto A\otimes \K(L^2(G)).
\end{equation}
We then say that $(A,\delta)$ is a {\em $\mu$-coaction}. 

Landstad duality for coaction crossed products asks under which conditions a given system $(B,\beta)$ for an action $\beta:G\car B$ 
coincides with a dual system $(A\rtimes_\delta \widehat{G},\widehat{\delta})$ for some coaction $\delta$ of $G$ on $A$. 
By the construction of crossed products by coactions explained above, 
for this to be true it is necessary that $(B,\beta)$ admits a nondegenerate $\rt-\beta$ equivariant \Star{}homomorphism $\phi:C_0(G)\to\M(B)$ 
 which under the desired isomorphism corresponds to $j_{C_0(G)}:C_0(G)\to \M(A\rtimes_{\delta}\widehat{G})$.
 Following \cite{Buss-Echterhoff:Exotic_GFPA}, we shall call such triple $(B,\beta,\phi)$ a 
 {\em weak} $G\rtimes G$-algebra. 
 \footnote{This is a special case of the weak $X\rtimes G$-algebra for a general proper $G$-space $X$ in \cite{Buss-Echterhoff:Exotic_GFPA}.}

Hence we should ask: given a weak $G\rtimes G$-algebra $(B, \beta, \phi)$ as above,  can we find 
a (unique) cosystem $(A,\delta)$ for $G$ such that 
$$(B, \beta, \phi)\cong (A\rtimes_\delta \widehat{G}, \widehat{\delta}, j_{C_0(G)})\;?$$
In this setting, existence has been shown in general by Quigg in \cite{Quigg}, but
uniqueness cannot hold without further restrictions: counterexamples are given for the dual coactions  
 $(A,\delta):=(B\rtimes_{\beta,\max}G,\widehat{\beta})$ and $(A_r,\delta_r):=(B\rtimes_{\beta, r}G,\widehat{\beta}_r)$
whenever the full and reduced crossed products do not coincide.
 To overcome this uniqueness problem, one needs to restrict to coactions that satisfy Katayama duality with respect to some  fixed (possibly exotic) crossed product $B\rtimes_{\beta,\mu} G$.
 Such a crossed product is said to be \emph{$\widehat\beta$-compatible}, if the dual coaction $\widehat{\beta}$ on $B\rtimes_{\beta,\max}G$ factors through a coaction on the quotient $B\rtimes_{\beta,\mu}G$. Of course, this always holds if the crossed product $B\rtimes_{\beta,\mu}G$ 
 comes from a duality crossed-product functor $\rtimes_\mu$.
The following has been shown  in \cite{Buss-Echterhoff:Exotic_GFPA}*{Theorem 4.3}:
 
 \begin{theorem}[exotic Landstad duality]\label{thm-Landstad}
 Let $(B, \beta, \phi)$ be a weak $G\rtimes G$-algebra and let $\|\cdot\|_\mu$ be any $\widehat{\beta}$-compatible exotic \cstar{}norm on $C_c(G,B)$.
 Then there exists a unique (up to isomorphism) cosystem $(A_\mu,\delta_\mu)$ for $G$ such that 
 $(A_\mu\rtimes_{\delta_\mu}\widehat{G}, \widehat\delta_\mu, j_{C_0(G)})$ is isomorphic to $(B,\beta, \phi)$ and such that 
 $(A_\mu,\delta_\mu)$ satisfies Katayama duality in the sense that the canonical homomorphism
$$\Psi: A_\mu\rtimes_{\delta_\mu}\widehat{G}\rtimes_{\widehat\delta_\mu,\max}G\onto A_\mu\otimes \K(L^2(G))$$
 of~\eqref{eq-bidualmap} factors through an isomorphism 
 $$A_\mu\rtimes_{\delta_\mu}\widehat{G}\rtimes_{\widehat\delta_\mu,\mu}G\cong A_\mu\otimes \K(L^2(G)).\;\footnote{We regard $\|\cdot\|_{\mu}$  as a norm on $C_c(G, A_\mu\rtimes_{\delta_\mu}\widehat{G})$ via the isomorphism $(A_\mu\rtimes_{\delta_\mu}\widehat{G}, \widehat\delta_\mu)\cong (B,\beta)$.}$$
  \end{theorem}

For later use, we need to briefly explain the basic steps for the construction of $(A_\mu,\delta_\mu)$. Let
 $(B, \beta, \phi)$ be a weak $G\rtimes G$-algebra. We shall often write 
 $f\cdot b:=\phi(f)b$ and $b\cdot f:=b\phi(f)$ for $f\in C_0(G), b\in \M(B)$ if confusions seems unlikely.
 Consider the dense \Star{}subalgebra $B_c:=\contc(G)\cdot B\cdot \contc(G)\subseteq B$ and define
 \begin{equation}\label{eq-BcG}
 B_c^{G,\beta}:=\{m\in \M(B)^{G,\beta}: f\cdot m, m\cdot f\in B_c\mbox{ for all }f\in \contc(G)\},
 \end{equation}
where $\M(B)^{G,\beta}$  denotes the fixed-point algebra for the extended
action of $G$ on $\M(B)$.
 We call $B_c^{G,\beta}$ (or just $B_c^G$ if $\beta$ is understood)
   the {\em fixed-point algebra with compact supports}.
The authors show in \cite{Buss-Echterhoff:Exotic_GFPA} that $\mathcal F_c(B):=\contc(G)\cdot B$ carries left and right actions of $B_c^{G}$ and $\contc(G,B)$, respectively, and compatible $B_c^{G}$ and $\contc(G,B)$-valued inner products
 given by the formulas
\begin{equation}\label{eq-Morita}
\begin{split}
m\cdot\xi&=m\xi\\
\xi\cdot\varphi&=\int_G\Delta(t)^{-1/2} \beta_t(\xi\varphi(t^{-1}))\dd t\\
{_{B_c^{G}}\lk \xi,\eta\rk}&=\int_G^{st} \beta_t(\xi\eta^*)\dd t\\
\lk \xi,\eta\rk_{\contc(G,B)}(t)&=\Delta(t)^{-1/2}\xi^*\beta_t(\eta),
\end{split}
\end{equation}
where the multiplication on the right hand side in the first line is given by multiplication inside $\M(B)$ 
and the strict integral in the third line is determined by the equation
$$\left(\int_G^{st} \beta_t(\xi\eta^*)\dd t\right)\zeta=\int_G \beta_t(\xi\eta^*)\zeta\dd t$$
for all $\zeta\in \F_c(B)$
(the integral on the right of this equation is then over a compactly supported continuous function).

Given any exotic crossed product \cstar{}norm $\|\cdot\|_\mu$ on the \Star{}algebra $C_c(G,B)$,
 the  $\contc(G,B)$-valued inner product can be regarded as a $B\rtimes_\mu G$-valued
inner product and the module $\F_c(B)$ then completes to  a full $B\rtimes_\mu G$-Hilbert module
$\F_\mu(B)$. Moreover, the left action of $B_c^{G}$ extends to a continuous left action
on $\F_\mu(B)$ with dense image in the compact operators $\K(\F_\mu(B))$.
Thus if  $B_\mu^{G,\beta}$ (or just $B_\mu^G$)  denotes the completion of $B_c^{G}$ with respect
to the operator norm for the module $\mathcal F_\mu(B)$, it follows that $\F_\mu(B)$
becomes a
$$B_\mu^{G}- B\rtimes_{\mu}G$$
equivalence bimodule.

 If $\|\cdot\|_\mu$ is $\widehat{\beta}$-compatible,  then
it is shown in  \cite{Buss-Echterhoff:Exotic_GFPA}*{\S 4} that there is a canonical coaction $\delta_{\F_\mu}$ on the module
$\F_\mu(B)$ which is compatible with $\widehat\beta_\mu$ and therefore induces a coaction
$\delta_{\mu}$ on the $\mu$-generalized fixed-point algebra $A_\mu:=B_\mu^{G}$.
It  is given on the dense subalgebra  $B_c^G$ by the  formula
\begin{equation}\label{eq:def-deltac}
\delta_{\mu}(m)\defeq (j_{\contz(G)}\otimes\id)(w_G)(m\otimes 1)(j_{\contz(G)}\otimes\id)(w_G)^*,\quad m\in B^G_c
\end{equation}
where $w_G\in U\M(C_0(G)\otimes C^*(G))$ is the unitary $w_G=[g\mapsto u_g]$.

The isomorphism $\big(A_\mu\rtimes_{\delta_\mu}\widehat{G}, \widehat{\delta}_\mu, j_{C_0(G)}\big)\cong (B,\beta,\phi)$
is  then given by the integrated form of the covariant pair $(\iota_\mu, \phi): (A_\mu, C_0(G))\to \M(B)$ 
 where $\iota_\mu: A_\mu\to \M(B)$ is given on the dense subalgebra
$B_c^{G}$ via its inclusion $B_c^{G}\into \M(B)$. 
For later use, we state 
 
 \begin{proposition}[{\cite{Buss-Echterhoff:Exotic_GFPA}}]\label{prop-Landstad}
The pair  $(\F_\mu(B), \delta_{\F_\mu})$ is a $\widehat{G}$-equivariant Morita equivalence
for the coactions $(A_\mu, \delta_\mu)$ and $(B\rtimes_\mu G, \widehat\beta_\mu)$. 
\end{proposition}

The following proposition follows from  the same arguments as given in the proof of \cite{Buss-Echterhoff:Exotic_GFPA}*{Lemma~7.1}.

\begin{proposition}\label{Landstad-functorial}
Suppose that $\rtimes_\mu$ is a crossed-product functor on the category of $G$-\cstar{}algebras. 
Then the assignment $(B, \beta,\phi)\mapsto B^{G,\beta}_\mu=:A_\mu$ as explained above is 
a functor from the category of weak $G\rtimes G$-algebras to the category of \cstar{}algebras.
If, moreover,  $\rtimes_\mu$ is a duality crossed-product functor, then $(B, \beta,\phi)\mapsto (A_\mu,\delta_\mu)$ 
is a functor from the category of  weak $G\rtimes G$-algebras to the category of cosystems for $G$.
\end{proposition}

\begin{remark}\label{rem-max-normal}
If $\|\cdot\|_\nu$ and  $\|\cdot\|_\mu$ are two $\widehat{\beta}$-compatible exotic \cstar{}norms with
$\|\cdot\|_\mu\geq \|\cdot\|_\nu$ on $C_c(G, B)$, then the identity maps on $B_c^{G}, \mathcal F_c(B)$, and $C_c(G,B)$, respectively, induce 
coaction equivariant quotient maps from the Morita equivalence triple 
$(A_\mu, \F_\mu(B), B\rtimes_{\mu}G)$ onto the  triple $(A_\nu, \F_\nu(B), B\rtimes_\nu G)$. In particular, since for any such  norm we assumed that
$\|\cdot\|_r\leq \|\cdot\|_\mu\leq \|\cdot\|_\max$, we obtain a chain of surjective  equivariant \Star{}homomorphisms
\begin{equation}\label{eq-max}
(A_{\max}, \delta_{\max})\onto (A_\mu,\delta_\mu)\onto (A_r,\delta_r)
\end{equation}
which all descend to the same dual system $(B,\beta)$ and such that 
$(A_\max, \delta_\max)$ satisfies Katayama duality for the maximal norm (hence it is a maximal coaction) and
$(A_r,\delta_r)$ satisfies duality for the reduced norm (hence it is a  normal  coaction).  

On the other hand, if we start with any coaction $\delta:A\to \M(A\otimes C^*(G))$, then
$(B,\beta,\phi):=(A\rtimes_\delta \widehat{G}, \widehat{\delta},j_{C_0(G)})$ is a weak 
$G\rtimes G$-algebra and there exists a unique $\widehat{\beta}=\widehat{\;\widehat\delta\;}$-compatible 
exotic \cstar{}norm $\|\cdot\|_\mu$ on 
$C_c(G, A\rtimes_\delta\widehat{G})$ such that 
$A\rtimes_{\delta}\widehat{G}\rtimes_{\widehat{\delta},\mu}G\cong A\otimes \K(L^2(G))$.
 We therefore
 recover $(A,\delta)$ as $(A_\mu,\delta_\mu)$  and the coactions $(A_{\max},\delta_{\max})$ and $(A_r,\delta_r)$ 
 of \eqref{eq-max} are then called the {\em maximalization} and {\em normalization} of $(A,\delta)=(A_\mu,\delta_\mu)$.
 For details we refer to the discussion before \cite{Buss-Echterhoff:Exotic_GFPA}*{Theorem 4.6}.
 Existence of a maximalization has first been shown in \cite{EKQ}. 
But we should point out, that this norm $\|\cdot\|_\mu$ may not be part of a crossed-product functor for $G$, an assumption we always need for our deformation procedures described below.
 \end{remark}

\section{Deformation by coactions}\label{sec:Landstad}

In this section we discuss our general approach to deformation via coactions.
Our idea is to use Landstad duality for coactions as described in the previous section.
The advantage of this approach is that it incorporates in a ``uniform'' way duality not only for maximal and reduced (or normal) coactions, but also for other ``exotic'' coactions. 

\subsection{Exotic deformation by coactions}\label{sec-deformation}
We now extend the ideas of Kasprzak \cite{Kasprzak:Rieffel} and Bhowmick, Neshveyev, and Sangha \cite{BNS}
to cover coactions which satisfy Katayama duality for a fixed duality crossed-product functor $\rtimes_\mu$ on the category of $G$-\cstar{}algebras.

For this we start with a $\mu$-coaction $\delta:A\to \M(A\otimes C^*(G))$ and let
$(j_A, j_{C_0(G)})$ denote the canonical maps from $(A,C_0(G))$ into $\M(A\rtimes_\delta\hatG)$.
We then write $(B,\beta,\phi):=(A\rtimes_\delta \widehat{G}, \widehat{\delta}, j_{C_0(G)})$ for the corresponding weak $G\rtimes G$-algebra.
In our first approach to deformation of $A$ we now vary the action $\beta$ in these data, that is, we 
define the deformation parameter space for $(B,\beta,\phi)$ as the set $\Hom_\phi(G,B)$ of all actions $\gamma:G\car B$ 
such that $\phi:C_0(G)\to \M(B)$  is $\rt - \gamma$ invariant. 
Then, for each such action $\gamma$ we obtain a {\em deformed coaction}
$(A_\mu^\gamma, \delta_\mu^{\gamma})$ as in Theorem~\ref{thm-Landstad}. 
In particular, we have
$$(A_\mu^{\gamma}\rtimes_{\delta_\mu^{\gamma}}\widehat{G}, \widehat{\delta_\mu^\gamma}, j_{C_0(G)}^{\gamma})\cong (B,\gamma, \phi)$$
for every parameter $\gamma\in \Hom_\phi(G,B)$. 

\begin{remark} Of course, one could also think about varying the parameter $\phi$ in $(B,\beta,\phi)$. But since 
for a fixed action $\gamma:G\car B$ the coactions $(A_\mu^\gamma,\delta_\mu^\gamma)$ are Morita equivalent to the dual coaction
$(B\rtimes_{\gamma,\mu}G, \widehat{\gamma}_\mu)$, this variation would not change the results -- at least up to equivariant Morita equivalence.
We do not know, however, whether a variation of the parameter $\phi$ (with fixed action $\gamma$) could possibly lead to non-isomorphic (but Morita equivalent) cosystems.
We refer to \cite{Buss-Echterhoff:Imprimitivity}*{Proposition 3.12 and Remark 3.13} for a more detailed discussion.
\end{remark}

\subsection{Abadie-Exel deformation}\label{subsec-EA}
 Following ideas of Abadie and Exel from \cite{Abadie-Exel:Deformation}, certain types of actions  in
 $\Hom_{\phi}(G, B)$ for a weak $G\rtimes G$-algebra $(B,\beta,\phi)=(A\rtimes_{\delta}\widehat{G}, \widehat{\delta}, j_{C_0(G)})$
 can be obtained as follows:
 Let $\eta:G\car B$ be any action which commutes with $\beta$ and
 such that $\eta_s(\phi(f))=\phi(f)$ for all $s\in G$ and $f\in C_0(G)$. 
 Then $\gamma_s:=\eta_s\circ \beta_s=\beta_s\circ\eta_s$ is an element of $\Hom_{\phi}(G,B)$ and we 
 can form the deformed cosystems $(A_\mu^{\gamma}, \delta_\mu^\gamma)$ as in our general approach. As before, we assume that $\delta$ is a $\mu$-coaction for a duality crossed-product functor $\rtimes_\mu$ for $G$.
 
We want to relate the actions $\eta:G\car B$ 
 with  actions $\alpha:G\car A$ which commute with the given coaction $\delta$
 in the
 sense that $\delta(\alpha_s(a))=(\alpha_s\otimes \id_G)(\delta(a))$ for all $a\in A$ and $s\in G$.
Given such action $\alpha$, the equation
 \begin{equation}\label{eq:definition-tilde-alpha}
\tilde\alpha_s(j_A(a)j_{C_0(G)}(f)):=j_A(\alpha_s(a))j_{C_0(G)}(f)\quad a\in A, f\in C_0(G)
\end{equation}
 determines an action $\tilde\alpha$ of $G$
 on $B=A\rtimes_{\delta}\hatG$ with the above properties. Indeed, the following result says that these are all actions of this form. 

\begin{lemma}\label{lem-alpha-commute}
An action $\alpha:G\car A$ that commutes with $\delta$ as above induces an action $\eta=\tilde\alpha$ of $G$ on $B=A\rtimes_\delta\dualG$ as in  \eqref{eq:definition-tilde-alpha} which commutes with $\beta=\widehat{\delta}$  and satisfies 
$\eta_s(\phi(f))=\phi(f)$ (for $\phi=j_{C_0(G)}$). 
Conversely, every action $\eta:G\car B$ with these properties is equal to $\tilde\alpha$ for some action  $\alpha:G\car A$ as above.
\end{lemma}
\begin{proof}
The covariance of $(j_A,j_{C_0(G)})$ implies the equality
$$(j_A\otimes\id)\delta(a)=(j_{C_0(G)}\otimes\id)(w_G)(j_A(a)\otimes 1)(j_{C_0(G)}\otimes \id)(w_G)^*,\quad a\in A.$$
Since $\alpha_t$ commutes with $\delta$, it is straightforward to check  that $(j_A\circ \alpha_t,j_{C_0(G)})$ is also covariant, so there is a unique \Star{}homomorphism $\eta_t\colon B\to B$ satisfying ~\eqref{eq:definition-tilde-alpha}. It follows directly from ~\eqref{eq:definition-tilde-alpha}
and strong continuity of $t\mapsto \alpha_t$ 
that $t\mapsto\eta_t$ is a strongly continuous homomorphism,  that $\eta$ fixes $\phi(f)=j_{C_0(G)}(f)$ for all $f\in C_0(G)$, and 
that it commutes with $\beta=\widehat{\delta}$.

Conversely, if $\eta$ is an action of $G$ on $B$ which commutes with $\beta$,  it follows from the functoriallity of $\rtimes_\mu$ \footnote{Recall that we  assume  $(A,\delta)$ to be a $\mu$-coaction for some duality crossed-product functor $\rtimes_\mu$.} that 
$\eta$ induces a strongly continuous action $\tilde\eta:G\car B\rtimes_\mu G$ via $\varphi\mapsto \eta_t\circ \varphi$ for  $\varphi\in C_c(G,B)$. 
If, moreover,  $\eta_t(\phi(f))=\phi(f)$ for all $f\in C_0(G)$, 
 one easily checks that $\eta_t(B^G_c)\sbe B^G_c$ for all $t\in G$ and it follows from ~\eqref{eq-Morita} that 
$\eta_t$ induces an automorphism $\eta^{\F}_t$ of the $B_c^G-C_c(G,B)$ bimodule $\mathcal F_c(B)$ by  
$$\eta^{\F}_t(\phi(f)b):=\phi(f)\eta_t(b)$$
 for $f\in C_0(G), b\in B$. 
 It hence also preserves the  norm on $A=B_\mu^G\supseteq B_c^G$ and therefore induces a strongly continuous 
  action $\alpha:G\car A$. 
Equation~\eqref{eq:def-deltac}  then implies that $\alpha$ commutes with  $\delta=\delta_\mu$.
\end{proof}

\begin{definition}\label{def-AE-deformation}
Let $(A,\delta)$ be a $\mu$-coaction with respect to a duality crossed-product functor $\rtimes_\mu$ and let $\alpha:G\car A$ be an action that commutes with $\delta$ as above. Let 
$(B,\beta, \phi)=(A\rtimes_\delta\widehat{G}, \widehat{\delta}, j_{C_0(G)})$ and let
$\gamma:=\tilde\alpha\cdot\beta\in \Hom_{\phi}(G,B)$ be the product of the actions  $\tilde\alpha$ and $\beta$ on $B$. 
Then we call $(A^\alpha, \delta^\alpha):=(A^\gamma_\mu, \delta^\gamma_\mu)$ the 
{\em Abadie-Exel deformation} of $(A,\delta)$ with respect to $\alpha$.

\end{definition}

We call this {\em Abadie-Exel deformation}, because it covers the deformation of
cross-sectional algebras of Fell bundles as studied by Abadie and Exel in \cite{Abadie-Exel:Deformation}.
Recall from \cites{Doran-Fell:Representations, Doran-Fell:Representations_2}  that a  Fell bundle $\mathcal A$ over the locally compact group $G$  is a collection of Banach spaces $\{A_s: s\in G\}$ together  with 
a set of pairings (called multiplications) 
 $A_s\times A_t\to A_{st}: (a_s, a_t)\mapsto a_sa_t$  and involutions $A_s\to A_{s^{-1}}; a_s\mapsto a_s^*$ 
which are compatible with the linear structures in the usual sense known from \cstar{}algebras including the condition 
 $\|a_sa_s^*\|=\|a_s\|^2$ for all $a_s\in A_s$.  If $G$ is not discrete, the topological structure of $\mathcal A$ is determined by the set $C_c(\mathcal A)$ of {\em continuous sections $a: s\mapsto a_s\in A_s$ with compact supports in $G$}, and  multiplication and involution on $\A$ are assumed to be continuous. Equipped with a natural convolution and involution, $C_c(\A)$ becomes a \Star{}algebra.
 As for crossed products, there are different (exotic) \cstar{}completions $C_\mu^*(\mathcal A)$ of $C_c(\mathcal A)$ which admit dual coactions 
  \begin{equation}\label{eq:mu-crossed-product-Fell-bundle}
\delta_\mu:C_\mu^*(\A)\to \M(C_\mu^*(\A)\otimes C^*(G))
\end{equation}
and which correspond to a (exotic) crossed product via duality (see \cite{Buss-Echterhoff:Maximality} for details on this).
So we can apply the above procedures to the pair $(A,\delta)=(C_\mu^*(\A), \delta_\mu)$. 

For discrete $G$ and $\|\cdot\|_\mu=\|\cdot\|_{\max}$, Abadie and Exel considered continuous actions $\alpha:G\car \A$ of $G$ by automorphisms of the Fell bundle $\A$, which then 
induce actions $\alpha:G\car C_\mu^*(\A)$ as in Definition~\ref{def-AE-deformation}. 
But different from our approach, Abadie and Exel used the action  
$\alpha:G\car \A$ to directly construct a deformed Fell bundle $\A^\alpha$ over $G$ and used the cross-sectional algebra 
$(C^*_{\mu}(\A^\alpha),\delta_\mu^\alpha)$ to define the deformation of $(A,\delta)$ corresponding to $\alpha$. 
 We shall show in the forthcoming paper \cite{BE:Fellbundles} that our deformation 
 can  be described in a similar way on the level 
of Fell bundles,  which then extends the Abadie-Exel deformation to the case of general locally compact groups and 
 (exotic) \cstar{}completions $C_\mu^*(\A)$ of $C_c(\A)$.

\subsection{Deformation by continuous cocycles}\label{subsec-cont-cocycle}

We now want to explain how we can extend our methods to also cover deformations by 
$2$-cocycles $\om:G\times G\to\T$ on $G$.  We start with the easier case of continuous cocycles, since it fits within the above 
described general approach. The more general setting of deformation by Borel cocycles will require some more refined methods, and 
will be treated in Section~\ref{sec:deformation-twist}.

Let $(B,\beta,\phi)=(A\rtimes_\delta \widehat{G},\widehat{\delta}, j_{C_0(G)})$ be a fixed $\mu$-coaction for a given  
duality crossed-product functor $\rtimes_\mu$. Let $Z_c^2(G,\T)$ denote the set of all continuous $2$-cocycles 
$\om:G\times G\to \T$ on $G$. Recall that the (normalized) cocycle conditions are given by 
\begin{equation}\label{eq-2cocycle}
\om(s,e)=\om(e,s)=1\quad\text{and}\quad \om(s,t)\om(st,r)=\om(s,tr)\om(t,r)\quad \forall s,t,r\in G.
\end{equation}

Starting with $(B,\beta,\phi)$ as above,  for each $\om\in Z_c(G,\T)$ we can construct a deformed action  $\beta^\om:G\car B$ as follows.
For $s\in G$, let $u_\om(s)\in C_b(G,\T)=U\M(C_0(G))$ be given as
\begin{equation}\label{eq-uomega}
u_\om(s)(r)=\overline{\om(r,s)}.
\end{equation}
A short computation using \eqref{eq-2cocycle} shows that  $u_\om$ is a \emph{$\om$-twisted $1$-cocycle} for $\rt:G\car C_0(G)$ 
in the sense that 
\begin{equation}\label{eq-exterior1}
u_\om(st)={\om(s,t)}u_\om(s)\rt_s(u_\om(t))\quad \forall s,t\in G.
\end{equation}
 Let $U_\om:G\to U\M(B)$ be defined by
 $U_\om(s)= \phi(u_\om(s))$.
 Since $\phi$ is $\rt-\beta$ equivariant, we immediately see that 
\begin{equation}\label{eq-exterior}
U_\om(st)={\om(s,t)}U_\om(s)\beta_s(U_\om(t))\quad \forall s,t\in G.
\end{equation}
Hence $U_\om$ is an $\om$-twisted $1$-cocycle for $\beta$ which then implies that
$$\beta^\om:G\car B; \beta^\om(s)=\Ad U_\om(s)\circ \beta(s) $$
is a new action of $G$ on $B$.
The fact that $\phi: C_0(G)\to \M(B)$ is  $\rt-\beta^\om$ equivariant follows from the equation
 $$\beta^\om_s(\phi(f))=\phi(u_\om(s))\beta_s(\phi(f))\phi(u_{\om}(s)^*)=\phi(u_{\om}(s) \rt_s(f)u_\om(s)^*)=\phi(\rt_s(f))$$
 by  $\rt-\beta$-equivariance of $\phi$ and commutativity of $\M(C_0(G))$. We now define the {\em $\om$-deformation} of $(A,\delta)$ 
 as the cosystem 
 $(A^\om_\mu, \delta^\om_\mu)$ corresponding to the weak $G\rtimes G$-algebra $(B, \beta^\om,\phi^\om)$ and $\rtimes_\mu$ as in 
 Theorem~\ref{thm-Landstad}.

 Of course, if $\rtimes_\mu=\rtimes_r$ is the reduced crossed-product functor, this gives the 
 deformation by continuous cocycles as defined in \cite{BNS}. 
Unfortunately, the above construction only works for continuous cocycles, while in general one needs to consider
 Borel cocycles. Covering Borel cocycles in the reduced case has been the major effort in \cite{BNS}. 

 \begin{remark}\label{rem-twisted} Let  $\T\into G_\om\onto G$ denote the central extension of 
 $G$ corresponding to $\om$ as discussed in Section~\ref{sec:twists} below. Let $(\beta,\iota^\om)$ be the Green-twisted action 
of the pair $(G_\om,\T)$, where we identify $\beta$ with its inflation to $G_\om$ and let $\iota^\om:\T\to U\M(B); z\mapsto z1_{\M(B)}$.
Let  $\tilde{U}_\om:G_\om\to U\M(B)$ be defined by $\tilde{U}_\om(g,z):=\bar{z}U_\om(g)$. Then 
$\tilde{U}_\om$ determines an exterior equivalence between $(\beta, \iota^\om)$ and the twisted pair 
$(\beta^\om, 1_{\T})$ as in \cite{Ech:Morita}*{p.~175},
and thus an exterior equivalence between the twisted action $(\beta,\iota^\om)$ and the ordinary $G$-action $\beta^\om$ in the sense of
\cite{Ech:Morita}. In particular, it follows that for every correspondence crossed-product functor $\rtimes_\mu$ we have
$B\rtimes_{(\beta,\iota^\om),\mu}G\cong B\rtimes_{\beta^\om,\mu}G$.
\end{remark}

\section{Group twists}\label{sec:twists}
In this section we introduce the notion of twists and their relation to actions on compact operators for locally compact groups, which will serve as the parameters in our approach to deformation by $2$-cocycles. 
Most results in this section are well known to the experts and go back to work of Mackey \cite{Mackey}, Moore \cite{mooreI},  Kleppner \cite{Klep:cont}, and the work on equivariant Brauer groups intiated by Crocker, Kumjian, Raeburn, and Williams in \cite{CKRW} (see also \cite{KMRW}). The results are generally well documented in the case of second countable groups $G$, but the documentation for non-second countable groups is sometimes rather weak, so 
we go into some detail below in order to set up notation, introduce some important constructions, and to clarify the results 
for general (possibly non-second countable) locally compact groups $G$.

By a {\em twist} $\sigma$ for $G$, we understand a central extension of locally compact groups
$$\sigma:=(\T\stackrel{\iota}{\into} G_\sigma\stackrel{{q}}{\onto} G)$$
of $G$ by the circle group $\T$. We say that two such twists $\sigma=(\T\stackrel{\iota}{\into} G_\sigma\stackrel{{q}}{\onto} G)$
and $\sigma'=(\T\stackrel{\iota'}{\into} G_{\sigma'}\stackrel{{q'}}{\onto} G)$ are {\em isomorphic} if there is an isomorphism $\varphi:G_\sigma\to G_{\sigma'}$ of topological groups 
which induces the identity maps on $G$ and $\T$. 
In what follows we shall write $[\sigma]$ for the isomorphism class of the twist $\sigma$.

\begin{remark}\label{rem-iso}
If $\sigma$ and $\sigma'$ are two twists for $G$ and if $\varphi:G_\sigma\to G_{\sigma'}$
is a continuous and bijective homomorphism which implements the identity maps on $G$ and $\T$, then 
$\varphi: G_\sigma\to G_{\sigma'}$ is already a topological isomorphism. This can be deduced from Gleason's theorem \cite{Gleason}*{Theorem 4.1}, which implies that both central extensions admit local continuous sections.
\end{remark} 

Note that the set $\Twist(G):=\{[\sigma]: \sigma \;\text{is a twist of $G$}\}$
has the structure of an abelian group with respect to the Baer multiplication defined as follows:
If $\sigma=(\T\stackrel{\iota}{\into} G_\sigma\stackrel{q}{\onto} G)$ and $\sigma'=(\T\stackrel{\iota'}{\into} G_{\sigma'}\stackrel{q'}{\onto} G)$ 
are  twists for $G$, we  define $G_\sigma*G_{\sigma'}$ as
the quotient $(G_\sigma\times_GG_{\sigma'})/\T$, where 
\begin{equation}\label{eq-productoftwists}
G_\sigma\times_G G_{\sigma'}:=\{(\tilde{g}, g')\in G_\sigma\times G_{\sigma'}: \tilde{q}(\tilde{g})=q'(g')\}
\end{equation}
denotes the fibred product of $G_\sigma$ with $G_{\sigma'}$ over $G$, and the action of $\T$ on $G_\sigma\times_GG_{\sigma'}$
is given by $z\cdot(\tig, g'):=(\bar{z}\tig, {z} g')$ (which is short for $(\iota(\bar{z})\tig, \iota'(z)g')$). It is not difficult to check that 
$$\sigma\sigma':=(\T\stackrel{\tilde\iota}{\into} G_\sigma*G_{\sigma'} \stackrel{\tilde{q}}{\onto} G)$$
with $\tilde\iota(z)=[(\tilde{e},z)]=[(z,e')]$ and $\tilde{q}([\tig, g'])=\tilde{q}(\tig)=q'(g')$
is a central extension of $G$ by $\T$. 
Here $\tilde{e}$ resp.~$e'$ denote the units of $G_\sigma$ resp.~$G_{\sigma'}$.
Then $[\sigma][\sigma']:=[\sigma\sigma']$
induces a well-defined multiplication on $\Twist(G)$ with inverse $[\sigma]^{-1}=[\bar\sigma]$ where
\begin{equation}\label{eq-inversetwist}  \bar\sigma:=(\T\stackrel{\bar{\iota}}{\into} G_\sigma \stackrel{{q}}{\onto}G),
\end{equation}
with $\bar{\iota}(z):=\iota(\bar{z})$ for all $z\in \T$.
Of course, these are standard operations in the theory of group extensions.

It follows from \cite{FG}*{Theorem~1} that every twist $\sigma=(\T\into G_\sigma\stackrel{q}{\onto} G)$ 
admits a Borel section 
$\mathfrak{s}:G\to G_\sigma$ for the quotient map $q$ which sends the unit of $G$ to the unit of $G_\sigma$. 
Then the map
\begin{equation}\label{eq-cocycle}
(g,h)\mapsto \omega(g,h):=\mathfrak{s}(g)\mathfrak{s}(h)\mathfrak{s}(gh)^{-1}
\end{equation}
defines a $\T$-valued Borel map which satisfies the cocycle identities \eqref{eq-2cocycle} and
hence an element $\om=:\om_{\sigma}\in Z^2(G,\T)$ 
and  a class $[\om_{\sigma}]\in H^2(G,\T)$, 
the Borel group cohomology with coefficients in the trivial $G$-module $\T$. It is well known (and easy to check) 
that the class $[\om_\sigma]\in H^2(G,\T)$ only depends on the class $[\sigma]\in \Twist(G)$ and not on the representative $\sigma$ 
or the Borel section $\mathfrak{s}$ (although the cocycle $\om_\sigma\in Z^2(G,\T)$ clearly depends on these choices).
We therefore obtain a well-defined map
\begin{equation}\label{eq-twist-to-cocycle}
    \Psi: \Twist(G)\to H^2(G,\T); [\sigma]\mapsto [\om_{\sigma}].
\end{equation}
It is not difficult to check that $\Psi$ is multiplicative if the product in $H^2(G,\T)$ is induced by the pointwise product of cocycles.
It follows from the work of Mackey \cite{Mackey} and Moore \cite{mooreI}, that
$\Psi:\Twist(G)\to H^2(G,\T)$ is an isomorphism of groups for second countable $G$. 
But work of Kleppner \cite{Klep:cont} shows that 
the result still holds for general locally compact groups. 
Indeed, to construct an inverse for the map 
$\Psi$ in \eqref{eq-twist-to-cocycle}, let $\om\in Z^2(G,\T)$ be fixed. 
Let $G_\om$ be the space $G\times \T$ equipped with the product
\begin{equation}\label{eq-Gom}
    (g,z)(h,w)=(gh, \om(g,h)zw).
\end{equation}
Equipped with the Borel structure from $G\times \T$, $G_\om$ becomes a Borel group.
If $\om$ is a {\em normalized} cocycle (i.e., if $\om(g,g^{-1})=1$ for all $g\in G$), then 
Kleppner shows in \cite{Klep:cont}*{p.~218} that there exists a locally compact topology on $G_\om$ 
such that the obvious sequence 
\begin{equation}\label{eq-sigmaom}
\sigma_\om:=(\T\into G_\om\onto G)
\end{equation}
is a twist for $G$ and such that
$\mathfrak{s}:G\to G_\om; g\mapsto (g, 1)$ becomes a Borel section. 
To see that the same construction works for general 
Borel cocycles $\om\in Z^2(G,\T)$ note first that every  cocycle $\om$ is equivalent to a normalized 
one, i.e., there exists a Borel map $f:G\to \T$ with $f(e)=1$ such that 
$$\om'(g,h)=f(g)f(h)f(gh)^{-1}\om(g,h)$$
is a normalized cocycle (e.g., see \cite{Klep:cont}*{p.~215}). We then obtain a Borel isomorphism
$$ \Phi :G_{\om'}\to G_{\om},\quad \Phi(g,z)=(g, f(g) z).$$
Equip $G_{\om}$ with the topology which makes this a topological isomorphism.
Then it follows from \cite{Klep:aut}*{Theorem 1} that this is the unique locally compact topology on $G_\om$
with the product Borel structure on the underlying space $G\times \T$. 
A short computation shows that $\om$ is precisely the cocycle induced from the section $\mathfrak{s}:G\to G_\om:g\mapsto (g,1)$ as in 
\eqref{eq-cocycle}.
Conversely, if $\sigma=(\T\stackrel{\iota}{\into} G_\sigma\stackrel{{q}}{\onto} G)$
is a twist for $G$, $\mathfrak{s}:G\to G_\sigma$ a Borel section and $\om\in Z^2(G,\T)$ as in 
\eqref{eq-cocycle}, then $\Phi:G_\om\to G_\sigma; (g,z)\mapsto \mathfrak{s}(g){\iota}(z)$ is an isomorphism 
of extensions. Thus $\Psi:\Twist(G)\to H^2(G,\T)$ is  an isomorphism of groups.

It is useful to recall the following result due to Mackey (for second countable $G$) and Kleppner (for the general case).
To fix notation, if $\om\in Z^2(G,\T)$  then a strongly measurable 
map $V:G\to \U(\H)$ is called a (projective) $\om$-representation, if 
\begin{equation}\label{eq-omega-rep}
V_gV_h=\om(g,h) V_{gh}\quad\forall g,h\in G.
\end{equation}
The following result is a consequence of \cite{Klep:cont}*{Theorem 1}:

\begin{proposition}\label{prop-Kleppner}
Let $\om\in Z^2(G,\T)$ be a Borel cocycle on $G$ and let $\H$ be a Hilbert space. 
There exists a one-to-one correspondence between 
\begin{enumerate}
 \item[(i)] strongly continuous unitary representations $\tilde{V}:G_\om\to \U(\H)$ satisfying 
 $\tilde{V}_{(g,z)}=z\tilde{V}_{(g,1)}$ for all $(g,z)\in G_\om$; and
    \item[(ii)] strongly measurable $\om$-representations $V:G\to \U(\H)$.
\end{enumerate}
Moreover, each strongly measurable $\om$-representation $V:G\to \U(\H)$ as in (ii) determines a continuous 
action $\alpha=\Ad V: G\car \K(\H)$.
\end{proposition}

Of course, if $\tilde{V}:G_\om\to \U(\H)$ is a  representation as in (i), then the corresponding 
$\om$-representation is given by $V_g:=\tilde{V}_{(g,1)}$ for $g\in G$. 
Conversely, if $V:G\to \U(\H)$ is an $\om$-representation, then the corresponding unitary representation 
$\tilde{V}:G_\om\to \U(\H)$ is given by $\tilde{V}_{(g,z)}=zV_g$ for each $(g,z)\in G_\om$.

\begin{remark}
For technical reasons it will be better for our purposes to work with actions $\alpha = \Ad V$, where $V:G\to \U(\H)$ 
is a strongly measurable $\bar{\om}$-representation with $\bar{\om}(g,h)=\overline{\om(g,h)}$ for all $g,h\in G$. 
An easy alteration of the above proposition shows that such representations 
are in a one-to-one correspondence with unitary representations $\tilde{V}:G_{\om}\to \U(\H)$ such that $\tilde{V}(g,z)=\bar{z}\tilde{V}(g,1)$ for all $(g,z)\in G_\om$.
\end{remark}

We now recall the Brauer group 
$\brg$ of all Morita equivalence classes of actions of $G$ on some algebra $\K(\H)$ of compact operators on a Hilbert space $\H$
(see \cites{Raeburn-Williams:Morita_equivalence,Dana:book, CKRW}). 
If $\alpha:G\car \K(\H)$ and $\beta:G\car \K(\H')$ are two actions, and if $[\alpha]$ and $[\beta]$ denote their Morita equivalence classes, then the product $[\alpha][\beta]$ in $\Br(G)$ is represented by the diagonal action 
$\alpha\otimes\beta:G\car\K(\H)\otimes \K(\H')\cong \K(\H\otimes\H')$. Of course,  the unit element is represented by the trivial action. 

In what follows below, if $\X$ is an $A-B$ equivalence bimodule, we denote by $\X^*$ the 
$B-A$ equivalence bimodule conjugate to $\X$. Recall that
$\X^*:=\{\xi^*: \xi\in \X\}$ equipped with the inner products 
\begin{equation}\label{eq-dual-inner-product}
    _B\braket{\xi^*}{\eta^*}:=\braket{\xi}{\eta}_B\quad\text{and}\quad \braket{\xi^*}{\eta^*}_A:
    ={_A}\braket{\xi}{\eta}
\end{equation}
and left and right actions of $b\in B$ and $a\in A$ given by
\begin{equation}\label{eq-action-dual}
b\xi^*:=(\xi b^*)^*\quad\text{and}\quad \xi^*a:=(a^*\xi)^*.
\end{equation}
In particular, if $\H$ is a Hilbert space, viewed as a $\K(\H)-\C$ equivalence bimodule, 
then $\H^*$ is a $\C-\K(\H)$ equivalence bimodule.

The following theorem is well known in case 
of second countable groups and separable Hilbert spaces $\H$ (e.g., see \cite{CKRW} and the introduction of \cite{KMRW}). 
But the result holds in complete generality:

\begin{theorem}\label{thm-twist-H2-Br}
Let $G$ be a locally compact group. Then there exists a canonical isomorphism 
 $\Phi:H^2(G,\T)\stackrel{\cong}{\to} \Br(G)$
 given by sending a class $[\om]\in H^2(G,\T)$ to the class $[\alpha=\Ad V]\in \Br(G)$, where 
$V:G\to\U(\H)$ is any strongly measurable $\bar\om$-representation of $G$, with $\bar\om(g,h)=\overline{\om(g,h)}$. 
As a consequence, we obtain a chain of isomorphisms
$$\Twist(G)\stackrel{\Psi}{\to} H^2(G,\T)\stackrel{\Phi}{\to} \Br(G)$$
with $\Psi:\Twist(G)\to H^2(G,\T)$ as in \eqref{eq-twist-to-cocycle}. 
\end{theorem}
\begin{proof}
We already know that $\Psi:\Twist(G)\to H^2(G,\T)$ is an isomorphism. To show that $\Phi:H^2(G,\T)\to \Br(G)$ is well defined, 
observe that if $[\om]=[\om']$ and  $V:G\to \U(\H)$ and $W:G\to \U(\H')$ are 
    strongly measurable $\bar\om$- and $\bar\om'$-representations, respectively, then the actions 
    $\Ad V$ and $\Ad W$ are $G$-equivariantly Morita equivalent. For this let $f:G\to \T$ be a Borel map such 
    that $\om=\partial( f) \om'$. Then $W':g\mapsto \overline{f(g)}W_g$ becomes an $\bar\om$-representation such that 
    $\Ad W'=\Ad W$.
    Thus we may assume without loss of generality that $\om=\om'$. 
    It is then easy to check that
    \begin{equation}
        \gamma:G\car \H'\otimes_{\C} \H^*;\quad \gamma_g(\eta\otimes \xi^*):=(W_g\eta)\otimes (V_g\xi)^*
    \end{equation}
    defines an action of $G$ which implements an $\Ad W-\Ad V$-equivariant Morita equivalence.
Moreover, if $[\om],[\om']\in H^2(G,\T)$ and if $\alpha=\Ad V$ for an $\bar\om$-representation $V:G\to \U(\H)$ and $\beta =\Ad W$ for some $\bar\om'$-representation 
$W:G\to\U(\H')$, then $V\otimes W\colon G\to \U(\H\otimes \H')$ is an $\overline{\om\om'}$-representation such that $\alpha\otimes\beta=\Ad(V\otimes W)$. 
Thus $\Phi:H^2(G,\T)\to \Br(G)$ is multiplicative. 

To show that $\Phi$ is bijective, we now construct an inverse map $\Theta:\Br(G)\to \Twist(G)$ for $\Phi\circ \Psi$.
Suppose that 
 $\alpha:G\car \K(\H)$ is an action.  We obtain a central extension 
$$\sigma_\alpha:=(\T\stackrel{\iota_\alpha}{\into} G_\alpha\stackrel{q_\alpha}{\onto} G)$$
with group $G_\alpha$ defined by 
\begin{equation}\label{eq-Galpha}
G_{\alpha}=\{(g,v)\in G\times \U(\H): \alpha_g=\Ad v\},
\end{equation}
and with the  quotient map $q_\alpha:(g,v)\mapsto g$ and inclusion map $\iota_\alpha: z\to (e, \bar{z}1_{\H})$. \footnote{It is important for later use that we choose the inclusion $z\mapsto \bar{z}1_{\H}$ and not the somewhat more natural inclusion $z\mapsto z1_{\H}$, which would give us the inverse of our extension $G_\alpha$!}
 
We need to show that, if $\alpha:G\car \K(\H_\alpha)$ 
and $\beta:G\car \K(\H_\beta)$ are Morita equivalent actions, then $\sigma_\alpha\cong \sigma_\beta$.
To see this, let $(\X,\gamma)$ be an $\alpha-\beta$ equivariant Morita equivalence between
$\K_\alpha:=\K(\H_\alpha)$ and $\K_\beta:=\K(\H_\beta)$. Let 
$$\theta:=\left(\begin{matrix}\alpha&\gamma\\ \gamma^*&\beta\end{matrix}\right):G\curvearrowright 
L(\X)=\left(\begin{matrix} \K_\alpha & \X\\ \X^* & \K_\beta\end{matrix}\right).$$
Let $\U(L(\X))$ denote the unitary group of the multiplier algebra $\M(L(\X))$.
Let $G_{\theta}:=\{(g, v)\in G\times \U(L(\X)): \theta_g=\Ad v\}$.
Let $p=\left(\begin{smallmatrix}1&0\\0&0\end{smallmatrix}\right)$ be the full projection such that 
$\K_\alpha=p L(\X)p$. We claim that 
$$\psi_\alpha: G_{\theta}\to G_{\alpha}: \psi_\alpha(g,v)=(g, pvp)$$
is an isomorphism of central extensions, and similarly for 
$$\psi_\beta: G_{\theta}\to G_{\beta}: \psi_{\beta}(g,v)=(g, (1-p)v(1-p)).$$
Indeed, if we write $v=\left(\begin{matrix} v_{11}& v_{12}\\ v_{21}& v_{22}\end{matrix}\right)\in \U(L(\X))$ such that $\Ad v=\theta_g$, then, extending $\theta_g$ to $\M(L(\X))$ and 
applying it to $p=\left(\begin{smallmatrix}1&0\\0&0\end{smallmatrix}\right)$, we see that
$$\left(\begin{matrix}1_{\H_\alpha}&0\\0&0\end{matrix}\right)=\left(\begin{matrix}\alpha_g(1_{\H_\alpha})&0\\0&0\end{matrix}\right) =\theta_g(p)=vpv^*=\left(\begin{matrix} v_{11}v_{11}^*&0\\0&0\end{matrix}\right) $$
and similarly, $1_{\H_\beta}=v_{22}v_{22}^*$. Applying the same to $\Ad v^*= \theta_{g^{-1}}$, we see that 
$pvp=v_{11}$ and $(1-p)v(1-p)=v_{22}$ are unitaries in $\U(\H_\alpha)$ and $\U(\H_\beta)$, respectively, with $\alpha_g=\Ad v_{11}$ and $\beta_g=\Ad v_{22}$.
Moreover, for each $x\in \X$ the equation 
$$\left(\begin{matrix}0&\gamma_g(x)\\0&0\end{matrix}\right)=
\theta_g\left(\begin{matrix} 0&x\\0&0\end{matrix}\right)=
v\left(\begin{matrix} 0&x\\0&0\end{matrix}\right)v^*=\left(\begin{matrix} 0&v_{11}xv_{22}^*\\0&0\end{matrix}\right) $$
shows that the matrix $\tilde v:=\left(\begin{matrix} v_{11}& 0\\0&v_{22}\end{matrix}\right)$ defines a  unitary  $\tilde{v}\in \U(L(\X))$ such that $\theta_g=\Ad\tilde{v}$. But then there must exist $z\in \T$ such that $\tilde{v}=zv$. This is only possible for $\tilde{v}=v$.
 Having this, it now easily follows that 
$\psi_\alpha$ and $\psi_\beta$ are indeed isomorphisms of central extensions.  
Thus, we  obtain a well-defined map $\Theta:\Br(G)\to \Twist(G)$ by sending a class $[\alpha]\in \Br(G)$ to the class 
$[\sigma_\alpha]\in \Twist(G)$. 

To see that it is an inverse for  $\Phi\circ \Psi$ we also
want to give a direct construction for the composition $\Phi\circ \Psi:\Twist(G)\to \Br(G)$.
For this let $\sigma=(\T\stackrel{\iota}{\into}G_\sigma\stackrel{q}{\onto} G)$.
Let $L^2(G_\sigma, \iota)$ be the Hilbert space completion of
\begin{equation}\label{Ccbariota}
C_c(G_\sigma,\iota):=\{\xi\in C_c(G_\sigma): \xi(\tilde{g}z)=\bar{z}\xi(\tilde{g})\;\forall \tig\in G_\sigma, z\in \T\}
\end{equation}
with respect to the inner product 
\begin{equation}\label{eq-innerproduct}
\braket{\xi}{\eta}:=\int_G \overline{\xi(\tig)}\eta(\tig)\,dg   \quad \text{(with $g=q(\tig)$).}
\end{equation}
Note that it follows from the definition of $C_c(G_\sigma,\iota)$ that the 
integrand in \eqref{eq-innerproduct} is constant on $\T$-cosets, and hence defines a function in $C_c(G)$. We define a strongly continuous representation 
$\tilde\rho: G_\sigma\to \U(L^2(G_\sigma,\iota))$ by
\begin{equation}\label{eq-reptilderho}
\big(\tilde\rho(\tig)\xi\big)(\tilde{s})=\Delta(g)^{\frac{1}{2}}\xi(\tilde{s}\tig).
\end{equation}
The action $\tilde\alpha:=\Ad \tilde\rho:G \car \K(L^2(G_\sigma, {\iota}))$ is constant on $\T$-cosets and therefore factors through an action $\alpha:G\car \K(L^2(G_\sigma,\iota))$.

We claim that $[\alpha]=\Phi\circ\Psi([\sigma])$. Indeed, if $\mathfrak{s}:G\to G_\sigma$ is any Borel section and if 
$\om$ is the corresponding cocycle as in \eqref{eq-cocycle}, then a short computation shows that 
$\tilde\rho\circ \mathfrak{s}$ is a strongly measurable $\bar\om$-representation with $\alpha=\Ad(\tilde\rho\circ \mathfrak{s})$. This proves the claim. 

We now  
observe that if $\alpha:G\car \K(L^2(G_\sigma,\iota))$
    is the action constructed from $\sigma$ as above, then, using Remark~\ref{rem-iso}, it is easy to check that
    $$\varphi: G_\sigma\to G_{\alpha}: \tig\mapsto (g, \tilde\rho(\tig))$$
    is a continuous isomorphism of central extensions.  
    This shows that $\Theta\circ\Phi\circ\Psi$ is the identity on $\Twist(G)$.
Conversely, if $\alpha:G\car \K(\H)$ is an action and $\sigma_\alpha:=(\T\into G_\alpha\onto G)$ is the twist as in \eqref{eq-Galpha}, then
the projection $\tilde{V}:G_\alpha\to\U(\H); (g,v)\mapsto v$
is a strongly continuous unitary representation  such that $\Ad \tilde{V}$ factors through $\alpha$ on $G$.
Moreover, we have $\tilde{V}_{(g,v)z}=\tilde{V}_{(g, \bar{z}v)}=\bar{z}\tilde{V}_{(g,V)}$ for all $z\in \T$. 

This implies that, for any Borel section $\mathfrak{s}:G\to G_{\alpha}$ with corresponding cocycle $\om\in Z^2(G,\T)$ as in \eqref{eq-cocycle}, the 
composition $V:=\tilde{V}\circ \mathfrak{s}$ is 
a strongly measurable $\bar\om$-representation which implements $\alpha$. 
Hence $\Phi\circ\Psi\circ \Theta$ is the identity on $\Br(G)$.
\end{proof}

 \begin{remark}\label{right-regular}
    Let $\mathfrak{s}:G\to G_\sigma$ be a Borel section for the extension $\sigma=(\T\into G_\sigma\onto G)$ with 
    corresponding cocycle $\om\in Z^2(G,\T)$ as in \eqref{eq-cocycle}, and let $L^2(G_\sigma,\iota)$ be as in the above proof. 
    We then
    obtain an isomorphism of Hilbert spaces 
    $L^2(G_\sigma,\iota)\stackrel{\cong}{\to} L^2(G);\xi\mapsto \xi\circ \mathfrak{s}$.
    The composition $\tilde\rho\circ \mathfrak{s}: G\to \U(L^2(G_\sigma,\iota))$ then transforms to the 
    right regular $\bar\om$-representation $\rho^{\bar\om}:G\to \U(L^2(G))$ given by the formula
$$\big(\rho^{\bar\om}(g)\xi\big)(h)=\Delta(g)^{\frac{1}{2}}\bar{\om}(h,g)\xi(hg).$$
    We leave the straightforward computations to the reader.
\end{remark}

For later use we want to point out the following direct relation  between an element
$[\sigma]\in \Twist(G)$ and the corresponding class $[\alpha]\in \Br(G)$.

\begin{corollary}\label{cor-twist-action}
    Let $\sigma=(\T\into G_\sigma\onto G)$ be a twist for $G$ and let  $\alpha:G\car \K(\H)$ be an action. Then 
     $[\alpha]=\Phi\circ\Psi([\sigma])$  with
    $\Phi\circ \Psi:\Twist(G)\to \Br(G)$ as in Theorem~\ref{thm-twist-H2-Br}
    if and only if $\alpha=\Ad \tilde{V}$ for some 
    strongly continuous unitary representation $\tildeV:G_\sigma\to \U(\H)$ such that
    $\tilde{V}_{\tig z}=\bar{z}\tilde{V}_{\tig}$ for all $\tig\in G_\sigma$ and $z\in \T$.
    \end{corollary}
\begin{proof}
First of all, as observed at the end of the proof of Theorem~\ref{thm-twist-H2-Br}, if $\tilde{V}:G_\sigma\to \U(\H)$ is 
as in the statement of the corollary, and if $\mathfrak{s}:G\to G_\sigma$ is any Borel 
section with corresponding cocycle $\om\in Z^2(G,\T)$, then $V:=\tilde{V}\circ\mathfrak{s}$ is a strongly measurable $\bar\om$-representation 
such that $\alpha=\Ad V$. Hence $[\alpha]=\Psi\circ\Phi([\sigma])$. 

Conversely, if $[\alpha]=\Psi\circ\Phi([\sigma])$, then it follows from the proof of Theorem~\ref{thm-twist-H2-Br} that 
$G_\sigma$ is isomorphic to $G_\alpha$ as group extensions,  with $G_\alpha$ as in \eqref{eq-Galpha}. 
But it is trivial to check that $\tilde{V}:G_\alpha\to\U(\H); \tilde{V}_{(g,v)}=v$ has the properties as stated in the corollary.
\end{proof}
 
\subsection*{Continuous families of twists}\label{subsec-conttwists}
We now want to proceed by introducing our notion of a continuous family of twists over $G$. These will later serve as the natural parameters in our study of continuity properties for the deformation process we define in Section~\ref{sec:deformation-twist} below.
To be short, we define a continuous family of twists over $G$ as a twist over the groupoid (i.e. the trivial group bundle) $X\times G$ in the sense of \cite{KMRW}*{Section 8}:

\begin{definition}\label{defn-continuous-family-twists}
By a {\em continuous family of twists} for a locally compact group $G$ over the locally compact space $X$ 
(or, more simply, a \emph{twist for $X\times G$}) we understand a central groupoid extension
of the form
\begin{equation}\label{eq-groupoid}
\Sigma:=(X\times \T\stackrel{\iota}\into \mathcal G\stackrel{q}{\onto} X\times G).
\end{equation}
\end{definition}

Note that a groupoid $\G_\Sigma:=\mathcal G$ as in the definition above is a group bundle over $X$ such that the fibres $G_x$ over $x\in X$
are central extensions $\sigma_x:=(\T\into G_x\onto G)$. 
Hence we obtain a class $[\sigma_x]\in \Twist(G)$ for all $x\in X$.

Another logical starting point for studying continuity of our deformation process would be to 
start with a continuous family of actions  $\alpha^x:G\car \K(\H)$ over $X$ as in:

\begin{definition}\label{def-continuousfamilyactions}
Let $G$ be a locally compact group and $X$ a locally compact space. For each $x\in X$, let $\alpha^x:G\car \K(\H)$ be an action. We then say that $x\mapsto \alpha^x$ is a {\em continuous family of actions on $\K(\H)$} if for every $k\in \K(\H)$ the map
$$X\times G\to \K(\H); (x,g)\mapsto \alpha^x_g(k)$$
is continuous.
\end{definition}
 
 The following lemma is obvious and we omit the proof.
 
\begin{lemma}\label{lem-continuousfamilyactions}
Let $G$, $X$ and $\K(\H)$ be given. Then there exists a one-to-one correspondence between 
continuous families of actions $\alpha^x:G\car \K(\H)$ over $X$ and $C_0(X)$-linear actions $\alpha:G\car C_0(X,\K(\H))$
given by 
$$\big(\alpha_g(f)\big)(x):=\alpha^x_g(f(x)).$$
\end{lemma}

The third possibility is to look for continuous families of cocycles as in

\begin{definition}\label{defn-cont-cocycles}
Let $G$ be a locally compact group. 
By a {\em continuous family} of circle-valued Borel $2$-cocycles $x\mapsto \om_x\in Z^2(G,\T)$ 
over the locally compact space $X$, we understand a Borel cocycle 
$\Om:G\times G\to C(X,\T)$ for the trivial $G$-module $C(X,\T)$ equipped with the topology of uniform convergence on compact sets, such that for all $x\in X$ and $g,h\in G$ we have 
$\om_x(g,h)=\Om(g,h)(x)$.
\end{definition}

If $G$ and $X$ are second countable, then it follows from \cite{KMRW}*{Lemma 8.2} that the first two definitions are 
equivalent up to the appropriate versions of equivalences: isomorphism of central groupoid extensions in the case of twists, and 
Morita equivalences of $C_0(X)$-linear actions of $G$ on $C_0(X,\K)$. We also have the following result, which follows from 
\cite{HORR}*{Proposition 3.1} together with \cite{CKRW}*{Theorem 5.1(3)}

\begin{proposition}\label{prop-cocycle-to-action}
    Suppose that $G$ and $X$ are second countable. Then any Borel cocycle $\Om:G\times G\to C(X,\T)$ with evaluations 
    $\om_x:=\Om(\cdot,\cdot)(x)$
    determines a 
    continuous family $x\mapsto \alpha^x$ of actions $\alpha^x:G\car \K(L^2(G))$ by defining
   $\alpha^x:=\Ad \rho^{\bar\om_x}$ for all $x\in X$, where $\rho^{\bar\om_x}:G\to \U(L^2(G))$ denotes the right regular $\bar\om_x$-representation of $G$ on $L^2(G)$. If $[\Om]=[\Om']\in H^2(G, C(X,\T))$,
   then the corresponding actions $\alpha$ and $\alpha'$ are $C_0(X)$-linearly Morita equivalent. 
\end{proposition}

\begin{remark}\label{rem-cont} 
It is an interesting question whether the above proposition still holds  if $G$ or $X$ are not second countable. 
A crucial fact used in the proof is the automatic continuity of a Borel homomorphism
between polish groups. But even in the second countable case we will usually not get all continuous families of actions
of $G$ on $\K(\H)$ over $X$ (up to Morita equivalence) from cocycles in $Z^2(G, C(X,\T))$, but only those 
which come from inner automorphisms $\alpha_g=\Ad v_g$ for some elements $v_g\in C(X, U(\H))$. 
 But not all $C_0(X)$-linear automorphisms 
of $C_0(X,\K(\H))$ need to be inner. We refer to 
\cites{CKRW, EW:locally-inner} for more details.

One can also ask whether a family of cocycles $\om_x\in Z^2(G,\T)$, $x\in X$, satisfying the condition that
for all $g,h\in G$ the function $x\mapsto \om_x(g,h)$ is continuous, will determine a Borel cocycle $\Om\in Z^2(G, C(X,\T))$
as in Definition~\ref{defn-cont-cocycles}. This is certainly true if $G$ is discrete or if the map 
$G\times G\times X\to \T; (g,h,x)\mapsto \om_x(g,h)$ is continuous, but it it is not clear to us whether it holds in general.
\end{remark}

Let us now see how  continuous families of twists and continuous families of actions are related if $G$ or $X$ 
are not assumed to be second countable. We first observe that every continuous family of actions 
determines a continuous family of twists. Below, the group $\U(\H)$ of unitaries on a Hilbert space $\H$ will be equipped with the 
strong operator topology.

\begin{lemma}\label{lem-groupbundle}
Let $X\ni x\mapsto\alpha^x$ be a continuous family of actions on $\K(\H)$. Let
$$\G_\alpha:=\{(x,g,v)\in X\times G\times \U(\H): \alpha^x_g=\Ad v\}.$$
Then the product topology on $X\times G\times \U(\H)$ induces a locally compact topology on $\G_\alpha$ and 
the canonical projection $p:\G_\alpha\to X$ gives $\G_\alpha$ the structure of a group bundle over $X$ with fibres 
$G_{\alpha^x}$ as in \eqref{eq-Galpha}. Moreover, we obtain twists 
$$\Sigma_\alpha:=(X\times \T\stackrel{\iota}{\into} \G_\alpha\stackrel{g}{\onto} X\times G)$$
over $X\times G$ as in Definition~\ref{defn-continuous-family-twists},  with quotient map $q: (x,g,v)\mapsto (x,g)$ and inclusion map
$\iota:(x,z)\mapsto (e,x, \bar{z}1_{\H})$,
which in each fibre induces the twist $\sigma_{\alpha^x}=(\T\into G_{\alpha^x}\onto G)$ as described in \eqref{eq-Galpha}.
\end{lemma}
\begin{proof} We only show that $\G_\alpha$ is locally compact since everything else is quite straightforward.
For this let $(x,g, v)$ be any fixed element of $\G_\alpha$. By Gleason's theorem \cite{Gleason}*{Theorem 4.1} we know that 
there are local continuous sections for the  central extension $\T\into \U(\H)\onto \PU(\H)$. 
Hence we can find a small neighbourhood $U$ of $v$, a neighbourhood $V$ of $q(v)\in \PU(\H)$, a continuous map
$\mathfrak s: V\to U$ and an open set $T\sbe \T$ such that $T\times V\cong U $ via $(z,[v])\mapsto z\mathfrak s([v])$.
Since the map $(x,g)\mapsto \alpha^x_g\in \Aut(\K(\H))=\PU(\H)$ is continuous, we find compact neighbourhoods 
$C$ of $x$ and $W$ of $g$
such that $\alpha_h^y\in V$ for all $(y,h)\in C\times W$. Let $\tilde{V}:=\{\alpha_h^y: (y,h)\in C\times W\}$. 
By continuity of $\alpha$, it follows that $\tilde{V}$ is compact. It follows that 
$$\G_\alpha\cap (C\times W\times U)$$
is a neighbourhood of $(x,g,v)$ in $\G_\alpha$ 
which is contained in the compact subset $K:=C\times W\times\T\mathfrak{s}(\tilde{V})$ of $X\times G\times \U(\H)$. 
Since $\G_\alpha$ is closed in $X\times G\times \U(\H)$ (which  follows from continuity of $\alpha$ and compactness of $\T$), 
the set $G_\alpha\cap K$ is a compact neighbourhood of $(x,g,v)$.
\end{proof}

The converse of this result is only clear with some restrictions: To prepare for it, let $\Sigma=(X\times \T\into \G\onto X\times G)$ be a 
twist for $X\times G$. In what follows we shall often write the elements of $\G$  as pairs $(\tig,x)$ with $x\in X$, $\tig\in G_x$, and we then denote by $(g,x)$ the image of $(\tig, x)$ under the quotient map $\G\onto X\times G$. Let
\begin{equation}\label{eq-groupoid-iota}
    C_0(\G,\iota):=\{f\in C_0(\G); f(\tig z, x)=\bar{z}f(\tig, x)\quad\forall (\tig,x)\in \G, z\in \T\},
\end{equation}
and write $C_c(\G,\iota)$ for the functions in $C_0(\G,\iota)$ with compact supports.
Define a $C_0(X)$-valued inner product on $C_c(\G,\iota)$  by
\begin{equation}\label{eq-L2Giota}
    \braket{\xi}{\eta}(x):=\int_G \overline{\xi(\tig,x)}\eta(\tig, x)\, dg
\end{equation}
With respect to this inner product, $C_c(\G,\iota)$ completes to give a Hilbert $C_0(X)$-module 
$L^2(\G,\iota)$. 
If we restrict functions in $C_c(\G, \iota)$ to the fibres $G_x$ over $x$ we obtain the 
inner product on $C_c(G_x,\iota)$ as defined in \eqref{eq-innerproduct}. Thus we see that 
$L^2(\G,\iota)$ is a continuous bundle of Hilbert spaces over $X$ with fibres $L^2(G_x,\iota)$. 
Consequently, the algebra of compact operators $\K(L^2(\G,\iota))$ is a continuous trace algebra over $X$ with fibres $\K(L^2(G_x,\iota))$. In particular, it is the \cstar{}algebra of $C_0$-sections of a continuous field of \cstar{}algebras over $X$.

In \eqref{eq-reptilderho} it is shown that the right translation action  of $G_x$ on the fibre $L^2(G_x,\iota)$ determines a unitary representation $\tilde\rho^x:G_x\to \U(L^2(G_x,\iota))$ as in \eqref{eq-reptilderho} such that the 
corresponding action $\Ad\tilde\rho^x\colon G_x\car\K(L^2(G_x,\iota))$ is constant on the central subgroup $\T\subseteq G_x$, and hence factors through a well-defined action $\alpha^x:G\car \K(L^2(G_x,\iota))$.  From this we can easily deduce that the right translation action of $\G$ on $L^2(\G,\iota)$  induces a well-defined $C_0(X)$-linear action $\alpha:G\car\K(L^2(\G,\iota))$ which is given  in each fibre $\K(L^2(G_x,\iota))$ by $\alpha^x$.

\begin{proposition}\label{prop-groupoid-to-action}
Let $\Sigma=(X\times\T\into \G\onto X\times G)$ be a twist for $X\times G$. Then the right translation action of $\G$ on $L^2(\G,\iota)$  induces a well-defined $C_0(X)$-linear action $\alpha:G\car\K(L^2(\G,\iota))$ given on each fibre by the action $\alpha^x:G\car \K(L^2(G_x,\iota))$ corresponding to the twist $\sigma_x=(\T\into G_x\onto G)$ as 
described in \eqref{eq-reptilderho}.
\end{proposition}

It is clear from the construction that if $\Sigma$ and $\Sigma'$ are isomorphic twists, i.e., $\G$ and $\G'$ are isomorphic 
groupoid extensions, then the resulting actions 
$\alpha$ and $\alpha'$ respectively, are conjugate in the sense that there exists a $C_0(X)$-linear isomorphism 
$\Phi:\K(L^2(\G,\iota))\to \K(L^2(\G',\iota))$ which intertwines $\alpha$ and $\alpha'$. In particular, the construction gives a well-defined map from the set $\Twist(X\times G)$ of isomorphism classes of twists over $X\times G$, to the set $B_0(X\times G)$
of Morita equivalence classes of $C_0(X)$-linear actions of $G$ on continuous trace algebras $A$ 
which are $C_0(X)$-linearly Morita equivalent to $C_0(X)$ (which just means that they have vanishing Dixmier-Douady invariant $\delta(A)\in H^2(X,\Z)$).

The above construction inverts the construction of $\Sigma_\alpha=(X\times\T\into \G_\alpha\onto X\times G)$
of Lemma~\ref{lem-groupbundle} up to Morita equivalence: 

\begin{lemma}\label{lem-invert}
    Let $\beta:G\car C_0(X,\K(\H))$ be a $C_0(X)$-linear action and let $\Sigma_\beta=(X\times\T\into \G_\beta\onto X\times G)$ be the 
    the corresponding twist as in Lemma~\ref{lem-groupbundle}. Then the action $\alpha:G\car\K(L^2(\G_\beta, \iota))$
    constructed from $\G_\beta$ as in Proposition~\ref{prop-groupoid-to-action} is $C_0(X)$-linearly Morita equivalent to $\beta$.
\end{lemma}
\begin{proof} We only give a sketch of the proof.
    The exterior tensor product $L^2(\G_\beta,\iota)\otimes \H^*$ is a $C_0(X)$-linear 
    $\K(L^2(\G_\alpha,\iota))-C_0(X)\otimes \K(\H)$ equivalence bimodule and carries an action $\gamma$ of the groupoid 
    $\G_\beta$ given on elementary tensors $\xi\otimes\eta*$ by 
    $$\gamma_{(x,g,v)}(\xi\otimes\eta^*)(x,h, u)=\Delta(g)^{\frac{1}{2}}\xi(x, hg, uv)\otimes (v\eta)^*$$
    at the fibre over $x\in X$. A short computation shows that the action is trivial on $X\times \T$ and therefore factors 
    through a $C_0(X)$-linear action of $G$ which then implements the desired Morita equivalence.
\end{proof}

If $G$ and $X$ are second countable, then we may pass to the Morita equivalent action 
$\alpha\otimes \id:G\car\K(L^2(\G, \iota))\otimes \K(\ell^2(\N))$ and use the well-known fact 
that $\K(L^2(\G, \iota))\otimes \K(\ell^2(\N))$ is $C_0(X)$-linearly isomorphic to $C_0(X,\K(\ell^2(\N)))$ to deduce the
equivalence of Proposition~\ref{prop-cocycle-to-action}. 
In the general case,  
the notion of continuous families of twists seems to  be the most general setting to study continuous deformations of elements 
in $\Twist(G)\cong \H^2(G,\T)\cong \Br(G)$.

We close this section by showing that each central extension $Z\stackrel{\iota_Z}{\into} H\stackrel{q_H}{\onto} G$ of $G$ by the abelian locally compact group $Z$ canonically determines a twist $\Sigma_H$ for  $\widehat{Z}\times G$, where $\widehat{Z}$ denotes the Pontryagin dual of $Z$. First observe that every character $\chi\in \widehat{Z}$ determines a twist 
\begin{equation}\label{eq-transgression}
   \sigma_\chi:=( \T\stackrel{\iota_\chi}{\into} 
G_\chi\stackrel{q_\chi}{\onto} G)
\end{equation} by defining 
$G_\chi:=(H\times \T)/Z$ with respect to the action $z(h,w)=(zh, \chi(z)w)$ for $z\in Z, (h,w)\in H\times \T$.
The inclusion and quotient maps are given by
$\iota_\chi:w\mapsto [e,w]$ and  $q_\chi:[h,w]\mapsto q_H(h)$, respectively. 
This construction can be done simultaneously 
over $\widehat{Z}$:

\begin{proposition}\label{prop-groupext}
Let $Z\stackrel{\iota_Z}{\into} H\stackrel{q_H}{\onto} G$
be as above. Let $Z$ act (freely and properly) on the product space $\widehat{Z}\times H\times \T$ by
$$z(\chi, h, w):= (\chi, zh, \chi(z)w)\quad\forall z\in Z, (\chi, h,w)\in \widehat{Z}\times H\times  \T.$$
Then there is a twist 
$$\Sigma_H:=(\widehat{Z}\stackrel{\iota}{\times} \T\into \G_H\stackrel{q}{\onto} \widehat{Z}\times G)$$
with $\G_H:=(\widehat{Z}\times H\times  \T)/Z$ and  inclusion and quotient maps are given by
$$\iota:(\chi, w)\mapsto [\chi, e, w]\quad{and}\quad q:[\chi, h, w]\mapsto (\chi, q_H(h)).$$
Here $e$ denotes the neutral element of $H$.
\end{proposition}

We omit the straightforward proof. Clearly, the fibre over $\chi\in \widehat{Z}$ of the  twist $\Sigma_H$ is just the twist 
$\sigma_\chi$ of \eqref{eq-transgression}.

\begin{remark}\label{rem-splitting}
    If there exists a Borel section $\mathfrak{s}_H:G\to H$ for the quotient map $q_H$ (e.g., if $G$ is discrete or $H$ is second countable), then we obtain a Borel cocycle
$\eta\in Z^2(G,Z)$ given by $\eta(g,h)=\mathfrak{s}_H(g)\mathfrak{s}_H(h)\mathfrak{s}_H(gh)^{-1}$. Then every character $\chi\in \widehat{Z}$
determines the cocycle $\om_\chi:=\chi\circ \eta\in Z^2(G,\T)$. One then easily checks that
$$G_{\om_\chi}\to G_\chi; (g,w)\mapsto [\mathfrak{s}_H(g), w]$$
induces  an isomorphism of central extensions of $G$ by $\T$.
\end{remark}

The following definition slightly extends a definition due to Calvin Moore in \cite{MooreII} for that it does not 
assume that $G$ is  second countable:

\begin{definition}[Moore]\label{defn-splitting}
A central extension  $Z\stackrel{\iota_Z}{\into} H\stackrel{q_H}{\onto} G$ as above is called a {\em representation group}
for $G$, if the {\em transgression map} $\tg:\widehat{Z}\to \Twist(G); \chi\mapsto [\sigma_\chi]$ is an  isomorphism of abelian groups.
We say that $G$ is {\em smooth} if there exists a representation group for $G$.
\end{definition}

Note that Moore called a second countable extension $Z\stackrel{\iota_Z}{\into} H\stackrel{q_H}{\onto} G$ a representation group, if the map $\tg:\widehat{Z}\to H^2(G,\T): \chi\mapsto [\om_\chi]$, with $\om_\chi$ as in Remark~\ref{rem-splitting} above, 
is an isomorphism of abelian groups. It follows from our discussions that both definitions are equivalent if there exists a Borel section $\mathfrak{s}_H:G\to H$, e.g., if $H$ is second countable.

In \cite{MooreII}, Moore used a representation group of $G$ (if it exists) to define a topology 
on $H^2(G,\T)$ by transporting the locally compact topology of $\widehat{Z}$ to $H^2(G,\T)$ via the transgression map.
He also shows that the topology does not depend on the choice of a particular representation group (see \cite{MooreII}*{Theorem 2.2}). Later, in \cite{MooreIV} he shows that a topology on $H^2(G,\T)$ can be defined intrinsically for all second countable $G$, 
which coincides with the above defined topology if $G$ is smooth. In general, the topology on $H^2(G,\T)$ can be non Hausdorff.
It is shown in \cite{MooreII}*{Theorem 3.1} that for every discrete group $G$ the group $H^2(G,\T)$ carries a canonical compact Hausdorff topology (induced from the topology of pointwise convergence of cocycles) and that all discrete groups admit a representation group. The proof does not use any countability conditions on $G$. 
In general, the  class of  smooth groups is quite large. Aside of discrete groups, it contains all second countable locally compact groups 
which are  compact or compactly generated abelian. Moreover, by 
\cite{MooreII}*{Proposition 2.7}  an almost connected  second countable group $G$ is smooth if and only if $H^2(G,\T)$ is Hausdorff. 
For more results we refer the reader to \cite{MooreII} and to \cite{EW:locally-inner}*{Section 4}.

\begin{example}\label{ex-RRn}
Let  $G=\R^n$ and, as a set, let $H_n = \R^{\frac{n(n+1)}2}$. Write an element of $H_n$ as
$s=(s_i,s_{j,k}),1\leq i\leq n,1\leq j<k\leq n$ and  define multiplication on $H_n$ by 
$$st = ((st)_i,(st)_{j,k})\quad\text{with}
\quad
(st)_i := s_i + t_i\; \text{and}\; (st)_{j,k} := s_{j,k} + t_{j,k} + s_jt_k.$$
Then $H_n$ is a central extension of $\R^n$ by $\R^{\frac{(n - 1)n}2}$. It is shown in \cite{EW:locally-inner}*{Example 4.7} that it is a representation group for $\R^n$.

The same construction with integer entries $s_i, s_{i,j}\in \Z$ will give us a central extension of $\Z^n$ by $\Z^{\frac{(n - 1)n}2}$ which serves as a representation group for $\Z^n$.
\end{example}

We should mention that, in general, a representation group of a smooth group is not unique (up to isomorphism of extensions).
It is shown in  \cite{EW:locally-inner}*{Proposition 4.8} that the representation group for a smooth group $G$ is unique if
every abelian extension $Z\into H\onto G_{\mathrm{ab}}$
splits, where $G_{\mathrm{ab}}$ denotes the abelianization of $G$. 
This implies, in particular, that the representation groups for $\R^n$ and $\Z^n$ are unique.

\begin{example}\label{ex-semi-simple}
Let $G$ be any connected real semi-simple Lie group.  Let $H$ be the universal covering group of $G$. It is a central extension
$$Z\into H\onto G$$
by some central discrete subgroup $Z$ of $G$. 
It is shown in \cite{MooreII}*{Proposition 3.4.} that this is the unique representation group of $G$.

In many cases, the group $Z$ will be finite, hence $\Twist(G)\cong H^2(G,\T)\cong \widehat{Z}$ will be finite (and discrete) as well. But there are some interesting exceptions. For example, the universal covering of $\PSL(2,\R)$ is a central extension of $\PSL(2,\R)$ by the integer group $\Z$, hence 
$H^2(\PSL(2,\R),\T)\cong \widehat{\Z}=\T$.
\end{example}

Suppose that $Z\into H\onto G$ is a representation group for  $G$.  Then, for every locally compact space $X$ and continuous map $\varphi:X\to\widehat{Z}$, we can construct the pullback $\varphi^*\G_H$  as 
\begin{equation}\label{eq-pullback}
    \varphi^*\G_H:=\{(x, [\chi, h, w])\in X\times\G_H: \varphi(x)=\chi\}.
\end{equation}
Together with the inclusion $\iota: (x, w)\mapsto (x, [\varphi(x), e, w])$ and the quotient map 
$q:(x, [\chi, h, w])\mapsto (x, q_H(h))$, we obtain the twist
$$\varphi^*\Sigma_H=(X\times \T\stackrel{\iota}{\into} \varphi^*\G_H\stackrel{q}{\onto} X\times G)$$
 with fibres 
$\sigma_{\varphi(x)}=(\T\into G_{\varphi(x)}\onto G)$ for all $x\in X$. If $G$ (and hence $H$) is second countable, or if $G$ is discrete,
it follows from Moore's results, as discussed above, that $\tg:\widehat{Z}\to H^2(G,\T)$ is an isomorphism of 
topological groups. We then may formulate

\begin{corollary}\label{cor-contmapH2}
Suppose that $Z\into H\onto G$ is a representation group for the second countable group $G$ and let
$\varphi:X\to H^2(G,\T)$ be any continuous map. Then
\begin{enumerate}
    \item[(i)] There exists a cocycle $\Om\in Z^2(G, C(X,\T))$ such that $\varphi(x)=[\om_x]$ for each $x\in X$,
    where $\om_x:=\Om(\cdot,\cdot)(x)$
    for all $x\in X$.
    \item[(ii)] There exists a continuous family of actions $\alpha:G\car C_0(X,\K(L^2(G))$ such that 
    $[\alpha^x]=\Phi(\varphi(x))$ for all $x\in X$, where $\Phi:H^2(G,\T)\to \brg$ is the isomorphism of 
    Theorem~\ref{thm-twist-H2-Br}.
    \end{enumerate}
\end{corollary}
\begin{proof}
    Let $\eta\in Z^2(G,Z)$ be a Borel cocycle corresponding to the extension $Z\into H\onto G$. 
   Let $\widehat{\eta}\in Z^2(G, C(\widehat{Z},\T))$ denote the cocycle given by
   $$\widehat{\eta}(g,h)(\chi)=\chi(\eta(g,h))\quad \forall g,h\in G,\chi\in \widehat{Z}.$$
   By the definition of the transgression map $\tg:\widehat{Z}\to H^2(G,\T)$ we then get 
   $[\widehat{\eta}(\cdot,\cdot)(\chi)]=[\om_\chi]=\tg(\chi)$ for all $\chi\in \widehat{Z}$.
   Thus, identifying $\widehat{Z}$ with $H^2(G,\T)$ via the transgression map, we may regard 
   $\widehat{\eta}$ as a cocycle in $Z^2(G, C(H^2(G,\T), \T))$ such that 
   $$[\widehat{\eta}(\cdot,\cdot)([\om])]=[\om]\quad \forall [\om]\in H^2(G,\T).$$
   Now, if $\varphi:X\to H^2(G,\T)$ is a continuous map, the pullback $\Om:=\varphi_*\widehat{\eta}$
   defined by $\Om(g,h)(x):=\widehat{\eta}(g,h)(\varphi(x))$, satisfies all requirements of item (i).
   Item (ii) is then a consequence of Proposition~\ref{prop-cocycle-to-action}.
   \end{proof}

Note that if $\varphi:X\to Z^2(G,\T); x\mapsto \om_x$ is any map such that for all pairs $g,h\in G$ the map 
$x\mapsto \om_x(g,h)$ is continuous, then it follows from Moore's description of the topology of $H^2(G,\T)$ 
in \cite{MooreIV}
that $z\mapsto [\om_x]\in H^2(G,\T)$ is continuous. So the above corollary will apply to this map. This gives a 
partial answer to the question raised in Remark~\ref{rem-cont}.

\section{Deformation by twists}
\label{sec:deformation-twist}

Let $\rtimes_\mu$ be a \emph{correspondence crossed-product functor} for $G$ and let $(A,\delta)$ be a $\mu$\nb-coaction. 
We  want to describe a procedure for constructing a deformed $\mu$-coaction  $(A^\sigma_\mu,\delta^\sigma_\mu)$ depending on a class 
$[\sigma]\in \Twist(G)$.
By the results in the previous section we know that $\Br(G)\cong \Twist(G)\cong H^2(G,\T)$, so 
our procedure can equally be regarded as a deformation by a class $[\alpha]\in \Br(G)$ or by a class $[\om]\in H^2(G,\T)$.

So in what follows let $\sigma=(\T\into G_\sigma\onto G)$ be a twist for $G$. By  Corollary~\ref{cor-twist-action} we can choose
 a unitary representation 
$\tilde{V}:G_\sigma\to \U(\H)$ which satisfies
$\tilde{V}_{\tig z}=\bar{z}\tilde{V}_{\tig}$ for all $\tig\in G_\sigma, z\in \T$.
Then $\Ad \tilde{V}:G_\sigma\car \K(\H)$ factors through an action $\alpha:G\car \K(\H)$ such that $[\alpha]\in \Br(G)$ corresponds to $[\sigma]$
under the isomorphism $\Twist(G)\cong \Br(G)$. 
In what follows below we shall always fix a a representation $\tilde{V}:G_\sigma\to \U(\H)$ and the  action $\alpha=\Ad\tilde{V}$ as above. 

As before, we shall often write $g:=q(\tig)$ for an element $\tig\in G_\sigma$.
Recall that
\begin{equation}\label{eq-module}
C_0(G_\sigma, \iota)=\{\xi\in C_0(G_\sigma): \xi(\tig z)=\bar{z}\xi(\tig)\; \forall \tig\in G_\sigma, z\in \T \}.
\end{equation}
This can be identified with the space of $\contz$-sections of the complex line bundle $\L_\sigma$ over $G$ associated to $\sigma$ defined by
the quotient $\L_\sigma:=\big(G_\sigma\times \C\big)/\T$,
with respect to the $\T$-action given by 
$z\cdot\big(\tig, w\big)=\big(\bar z\tig, zw\big)$ for $z\in \T$, $w\in \C$, $x\in G_\sigma$.

Then $C_0(G_\sigma, \iota)$ becomes a $C_0(G)-C_0(G)$ equivalence bimodule with respect to the 
canonical left and right actions of $C_0(G)$ by pointwise multiplication, and left and right inner products defined by
\begin{equation}\label{eq-inner}
_{C_0(G)}\braket{\xi}{\eta}(g):=\xi(\tig)\overline{\eta(\tig)}\quad \text{and}\quad
\braket{\xi}{\eta}_{C_0(G)}(g):=\overline{\xi(\tig)}\eta(\tig).
\end{equation}
Let  $\H^*=\{\xi^*:\xi\in \H\}$ denote the dual of the $\K(\H)-\C$ equivalence module $\H$ (see \eqref{eq-dual-inner-product} and \eqref{eq-action-dual} above).  Denote by $\tildeV^*$ the  action of $G_\sigma$ on $\H^*$ given by
\begin{equation}\label{eq-theta*}
\tildeV^*_{\tig}(\xi^*):=(\tilde{V}_{\tig}\xi)^*.
\end{equation}
Notice that $\tildeV^*$ then satisfies the relation $\tildeV^*_{\tig z}=z\tildeV^*_{\tig}$ for all $\tig\in G_\sigma$ and $z\in \T$.
Recall that $\rt:G\car C_0(G)$ denotes the right translation action for $G$. Similarly, in what follows, we shall denote by $\tilde\rt:G_\sigma\car C_0(G_\sigma, \iota)$ the right translation for $G_\sigma$. A short computation shows that it 
satisfies the relation
$\tilde\rt_{\tig z}=\bar{z}\tilde\rt_{\tig}$ for all $\tig\in G_\sigma, z\in \T$. 

\begin{lemma}\label{lem-Morita}
The diagonal action $\tilde\rt\otimes \tildeV^*:G_\sigma\car C_0(G_\sigma, \iota)\otimes \H^*$ factors through an action of $G$ which makes
$C_0(G_\sigma, \iota)\otimes \H^*$  an 
$(C_0(G),\rt)-(C_0(G)\otimes \K(\H),\rt\otimes\alpha)$ equivariant Morita equivalence.
\end{lemma}
\begin{proof} 
    For all $\tig\in G_\sigma$ and $z\in \T$ we have
    $$(\tilde\rt\otimes{\tildeV}^*)(\tig z)= \bar{z}\tilde\rt_{\tig}\otimes z{\tildeV}^*_{\tig}
    =(\tilde\rt\otimes {\tildeV}^*)_{\tig},$$
    which shows that the action only depends 
    on $g=q(\tig)$. It is then an easy exercise to check that this action is 
    compatible with the right translation action $\rt:G\car C_0(G)$ on the left and 
    the diagonal action $\rt\otimes \alpha$ on the right.
    \end{proof}

\begin{remark}\label{rem-twist1}
The central extension $G_\sigma$ gives rise to a Green-twisted action $(\id_\C,\iota)$ of the pair 
$(G_\sigma, \T)$ on $\C$, where $\iota:\T\to\T=\U(\C)$ is the identity. 
For each action $\beta:G\car B$ we may then tensor $\beta$ (inflated to $G_\sigma$) 
with the twisted action $(\id_\C, \iota)$ on $B=B\otimes\C$, which we denote by 
$(\beta, {\iota})$ (compare with Remark~\ref{rem-twist}).
It is then easy to check that the right translation action $\tilde\rt:G_\sigma\curvearrowright C_0(G_\sigma,\iota)$ 
implements an $\rt-(\rt,\iota)$-equivariant Morita equivalence of Green twisted actions of $(G_\sigma,\T)$, 
where we identify an action of $G$ with its inflated (trivially) twisted action of $(G_\sigma,\T)$ (see \cite{Ech:Morita} for the relevant notions of Morita equivalent twisted actions).

Similarly, one checks that the action $\tildeV^*:G_\sigma\curvearrowright \H^*$ implements an 
$(\id_\C, \iota)-\alpha$ equivariant $\C-\K(\H)$ Morita equivalence. 
The module of Lemma~\ref{lem-Morita} above is just the composition (over $\C$) of these 
two twisted Morita equivalences.
\end{remark}

\begin{proposition}\label{prop-Borelcocycle}
Let $(B,\beta, \phi)$ be any  weak $G\rtimes G$-algebra for the locally compact  group $G$ and let $(G_\sigma, \tildeV)$
and $\alpha=\Ad\tildeV:G\car\K(\H)$ be as above. Then the internal tensor product
\begin{equation}\label{eq-EalphaB}
    \E_\sigma(B):= \big(C_0(G_\sigma,\iota)\otimes_{C_0(G)} B)\otimes_\C\H^*
\end{equation}
is a full Hilbert $B\otimes \K(\H)$-module and the diagonal action 
$$\gamma_\sigma:=\tilde\rt\otimes\beta\otimes\tildeV^*:G\car \E_\sigma(B)$$
is compatible with the  action $\beta\otimes \alpha:G\car B\otimes\K(\H)$.

As a consequence, if 
$\beta_\sigma:=\Ad\gamma_\sigma$ denotes the adjoint action on $B_\sigma:=\K(\E_\sigma(B))$, then  $(\E_\sigma(B), \gamma_\sigma)$  becomes 
a $\beta_\sigma-\beta\otimes \alpha$ equivariant $B_\sigma-B\otimes \K(\H)$ Morita equivalence.

Together with the left action $\phi_\sigma: C_0(G)\to\L(\E_\sigma(B))\cong \M(B_\sigma)$, which is induced from the left action of 
$C_0(G)$ on $C_0(G_\sigma,\iota)$, the triple $(B_\sigma, \beta_\sigma,  \phi_\sigma)$
becomes a  weak $G\rtimes G$-algebra.
\end{proposition}
\begin{proof} Note first that, as in the proof of Lemma~\ref{lem-Morita}, it follows that $\gamma_\sigma$ is a well-defined action of $G$.
Since the internal tensor product of full Hilbert modules is always full, it is clear that $\E_\sigma(B)$ is a full Hilbert $B\otimes\K(\H)$-module
and therefore becomes a $\K(\E_\sigma(B))-B\otimes \K(\H)$ Morita equivalence.
We claim that the diagonal action $\tilde\rt\otimes\beta\otimes\tildeV^*$ of $G$ on $\E_\sigma(B)$ is compatible with $\beta\otimes \alpha$. 
Using the $\rt-\beta$ equivariance of $\phi$ we compute for all elementary tensors 
$\xi_i\otimes_i b_i\otimes \eta_i^*$ with $\xi_i\in C_0(G_\sigma,\iota)$, $b_i\in B$, and $\eta_i^*\in \H^*$, 
$i=1,2$, and for all $\tig\in G_\sigma$:
\begin{align*}
&\braket{\tilde\rt_{\tig}(\xi_1)\otimes\beta_g(b_1)\otimes \tildeV^*_{\tig}(\eta_1^*))}{\tilde\rt_{\tig}(\xi_2)\otimes\beta_g(b_2)\otimes  \tildeV^*_{\tig}(\eta_2^*)}_{B\otimes \K(\H)}\\
&=\braket{\tilde\rt_{\tig}(\xi_1)\otimes\beta_g(b_1)}{\tilde\rt_{\tig}(\xi_2)\otimes\beta_g(b_2)}_{B}\otimes \braket{\tildeV^*_{\tig}(\eta_1^*)}{\tildeV^*_{\tig}(\eta_2^*)}_{\K(\H)}\\
&=\beta_g(b_1^*)\phi\big(\braket{\tilde\rt_{\tig}(\xi_1)}{\tilde\rt_{\tig}(\xi_2)}_{C_0(G)}\big)\beta_g(b_2)\otimes 
\alpha_g\big(\braket{\eta_1^*}{\eta_2^*}_{\K(\H)}\big)\\
&=\beta_g(b_1^*)\phi\big(\rt_g\big(\braket{\xi_1}{\xi_2}_{C_0(G)}\big)\beta_g(b_2)\otimes 
\alpha_g\big(\braket{\eta_1^*}{\eta_2^*}_{\K(\H)}\big)\\
&=\beta_g\big(b_1^*\phi\big(\big(\braket{\xi_1}{\xi_2}_{C_0(G)}\big)b_2\big)\otimes 
\alpha_g\big(\braket{\eta_1^*}{\eta_2^*}_{\K(\H)}\big)\\
&=(\beta\otimes\alpha)_g\big(\braket{\xi_1\otimes b_1\otimes \eta_1^*}{\xi_2\otimes b_2\otimes \eta_2^*}_{B\otimes \K(\H)}\big).
\end{align*}
This proves the claim and all statements of the proposition follow  from this.
 \end{proof}

 \begin{definition}\label{def-cosystem}
 Let $(A,\delta)$ be a $\mu$-coaction for some duality crossed-product  functor $\rtimes_\mu$ and let $(B,\beta, \phi):=(A\rtimes_\delta \widehat{G}, \widehat{\delta}, j_{C_0(G)})$ be the corresponding weak $G\rtimes G$\nb-algebra. For a twist $\sigma=(\T\into G_\sigma\onto G)$ and corresponding action $\alpha=\Ad\tilde{V}:G\car \K(\H)$, let  $(B_\sigma, \beta_\sigma,  \phi_\sigma)$  be the weak  $G\rtimes G$-algebra as constructed above. 
 We then  define the {\em $\sigma$-deformation} of $(A,\delta)$ as  the cosystem 
 $(A_\mu^\sigma, \delta_\mu^\sigma)$  associated to $(B_\sigma, \beta_\sigma,  \phi_\sigma)$
and $\rtimes_\mu$ as in Theorem~\ref{thm-Landstad}. 
\end{definition} 

The following is a direct consequence of our constructions together with \cite{EKQR}*{Theorem 2.15}. Since it requires $\rtimes_\mu$ to be compatible for Morita equivalences, we shall formulate it for correspondence crossed-product functors $\rtimes_{\mu}$. Recall that these are also duality crossed-product functors.

\begin{proposition}\label{prop-coactionequivalence}
Let $(A,\delta)$ be as above and assume that $\rtimes_\mu$ is a correspondence crossed-product functor. Then the composition 
$$\F_\mu(B_\sigma)\otimes_{B_\sigma\rtimes_\mu G}\big(\E_{\sigma}(B)\rtimes_\mu G\big)$$
of the $A_\mu^\sigma-B_\sigma\rtimes_\mu G$ equivalence bimodule $\F_\mu(B_\sigma)$  of Proposition~\ref{prop-Landstad}
with the $\mu$-crossed product $\E_\sigma(B)\rtimes_\mu G$ becomes 
a $\delta_{A^\sigma_\mu}-\widehat{\beta\otimes \alpha}_\mu$ equivariant $$A^{\sigma}_\mu-(B\otimes\K(\H))\rtimes_{\beta\otimes\alpha,\mu} G$$ equivalence bimodule with respect to the composition of the coaction $\delta_{\F_\mu(B_\sigma)}$ on $\F_\mu(B_\sigma)$ with the dual coaction 
$\widehat{\gamma_\sigma}_\mu$ on $\E_\sigma(B)\rtimes_\mu G$. 
\end{proposition}

\begin{remark}\label{rem-twist} 
It is useful to observe that the diagonal action of $G_\sigma$ on the tensor factor \begin{equation}\label{eq-FalphaB}
\L(G_\sigma,B):=C_0(G_\sigma,\iota)\otimes_{C_0(G)}B
\end{equation}
of $\E_\sigma(B)$ in \eqref{eq-EalphaB} induces a twisted equivariant Morita equivalence between 
$(B_\sigma, \beta_\sigma)$ and the Green-twisted system $(B, (\beta,\iota))$ of $(G_\sigma, \T)$.
In particular, we then get $B_\sigma=\K(\L(G_\sigma,B))$ and $\beta_\sigma=\Ad(\tilde\rt\otimes \beta)$.
We used the $B_\sigma-B\otimes \K(\H)$ module $\E_\sigma(B)$ in our constructions to avoid the use 
of twisted actions and their crossed products at this place.
But this observation shows very well that our construction directly depends on the twist
$\sigma=(\T\into G_\sigma\onto G)$, and it is clear that the construction could have been done 
without any reference to the action $\alpha=\Ad\tilde{V}:G\car \K(\H)$.

Since an isomorphism of extensions $\sigma=(\T\into G_\sigma\onto G)\cong (\T\into G_{\sigma'}\onto G)=\sigma'$ directly induces an isomorphism of  
modules 
$$\L(G_\sigma,B)=C_0(G_\sigma,\iota)\otimes_{C_0(G)}B\cong C_0(G',\iota)\otimes_{C_0(G)}B=\L(G_{\sigma'},B),$$
it follows that $(B_\sigma,\beta_\sigma,\phi_\sigma)$ depends, up to isomorphism, only on the  class 
$[\sigma]\in \Twist(G)$ or, equivalently, on the class $[\alpha]\in \brg$.
\end{remark}

It follows from Proposition~\ref{prop-coactionequivalence} that the deformed cosystem $(A_\mu^\sigma, \delta_\mu^\sigma)$  is Morita equivalent to the dual cosystem $\big(B\otimes\K(\H))\rtimes_\mu G, \widehat{\beta\otimes\alpha}\big)$ if $\rtimes_\mu$ is a correspondence functor. By the above remark
we also see that $A_\mu^\sigma$ is  Morita equivalent to the Green-twisted crossed product 
$B\rtimes_{(\beta,\iota), \mu}G$ for an appropriate notion of exotic twisted crossed products. 
Now if we think in terms of classes $[\om]\in H^2(G,\T)$ via the isomorphism 
$\Psi:\Twist(G)\stackrel{\cong}{\to} H^2(G,\T)$, then $B\rtimes_{(\beta,\iota), \mu}G$  corresponds to the 
Busby-Smith type twisted crossed product
$B\rtimes_{(\beta,\om),\mu}G$ 
as used in the reduced setting in \cite{BNS}. In particular, we can now introduce

\begin{definition}\label{deformom}
For each $\om\in Z^2(G,\T)$, let $\sigma_\om=(\T\into G_\om\onto G)$ be the twist associated to $\om$ as in \eqref{eq-sigmaom}. Then, if $(A,\delta)$ is any $\mu$-coaction for some duality crossed-product functor $\rtimes_\mu$ for $G$, we define 
$$(A^\om_\mu, \delta^\om_\mu)=(A_\mu^{\sigma_\om},\delta_\mu^{\sigma_\om}). $$
We call $(A^\om_\mu, \delta^\om_\mu)$ the $\om$-deformation of $(A,\delta)$. 
\end{definition}

In the case where  $\om:G\times G\to\T$ is a continuous cocycle, 
the above construction  coincides with the one defined in \S~\ref{subsec-cont-cocycle}. 
Indeed, we can even show the following

\begin{proposition}\label{pro:condition-iso-B=Bom}
    Let $(B,\beta,\phi)$ be a weak $G\rtimes G$-algebra. Let $\om\in Z^2(G,\T)$ and 
    $\sigma_\om=(\T\into G_{\om}\onto G)$ be the extension corresponding to $\om$. Let $\mathfrak{s}:G\to G_{\om}; g\mapsto (g,1)$
    be the canonical cross-section. Let $D_{\om}(G)$ denote the $G$-invariant closed \Star{}subalgebra of $L^\infty(G)$ 
    generated by 
    $$\{f\circ \mathfrak{s}, f\in C_0(G_\om,\iota)\}\cup \{{\om}(\cdot, g): g\in G\}.$$
    If  $\phi\colon \contz(G)\to \M(B)$ extends to a 
 $G$-equivariant unital \Star{}homomorphism $\tilde\phi\colon D_{\om}(G)\to \M(B)$, then there exists an isomorphism of Hilbert $B$-modules 
    \begin{equation}\label{eq:rep-L-om}
       \Phi:  \L(G_\om,B):=\contz(G_{\om},\iota)\otimes_\phi B\congto B,\quad f\otimes b\mapsto \tilde\phi(f\circ\mathfrak{s})b,
    \end{equation}
which intertwines $\phi_{\om}$ with $\phi$, and  transforms the action $\beta_{\om}$ on $B_{\om}\cong B$  to the action
$s\mapsto \Ad U_{\om}(s)\circ \beta_s$, where $U_{\om}(s):=\tilde\phi(u_{\om}(s))$, with 
$u_{\om}(s)\in \U (D_\om(G))$ defined as in~\eqref{eq-uomega} by $u_{\om}(s)(t):=\bar\om(t,s)$. 
\end{proposition}
\begin{remark}\label{rem--iso-B-om}
Note that $D_\om(G)$ always contains $C_0(G)=\overline{C_0(G_\om,\iota)}\cdot C_0(G_\om,\iota)$. 
If $\om$ is continuous, then $D_\om(G)\subseteq C_b(G)=\M(C_0(G))$, and then $\tilde\phi: D_\om(G)\to\M(B)$
always exists. This implies that the weak $G\rtimes G$-algebra $(B_\om,\beta_\om,\phi_\om)$ of 
Proposition~\ref{prop-Borelcocycle} is isomorphic to the triple $(B,\beta^\om,\phi)$ as defined in Section~\ref{subsec-cont-cocycle}.

The proposition also applies if one can extend $\phi:C_0(G)\to \M(B)$ to a $G$\nb-equivariant unital \Star{}homomorphism 
$\tilde\phi:L^\infty(G)\to\M(B)$. As we shall see in Proposition~\ref{pro:dual-coactions-deformation}
below, this  holds true for dual coactions on crossed products. 
In a forthcoming paper we shall see that this also  holds for (dual) coactions on 
cross-sectional \cstar{}algebras of Fell bundles.

In general, the Hilbert $B$-module $\L(G_\omega,B)$ in \eqref{eq:rep-L-om} might not be isomorphic to $B$. 
For instance, if $B=\contz(G)$, then $\L(G_\omega,B)$ is the equivalence $\contz(G)$-$\contz(G)$-bimodule $\contz(G_{\omega},\iota)$, which is 
isomorphic to $\contz(G)$ as a right Hilbert $\contz(G)$-module if and only if $\omega$ is continuous.
But even if  $\L(G_\omega,B)$ is not isomorphic to $B$ as a right Hilbert $B$-module, it may still be true that $B_\om=\K(\L(G_\omega,B))\cong B$ as \cstar{}algebras. For example, this is always true for $B=C_0(G)$. Indeed, we have no example where we can definitely say that  $B$ and $B_\om$ are not isomorphic! Notice, however, that $B_\om$ and $B$ are  Morita equivalent by construction, so that they are stably isomorphic if both are $\sigma$-unital.
\end{remark}
\begin{proof}[Proof of Proposition~\ref{pro:condition-iso-B=Bom}]
In what follows below 
 products of subspaces will always mean the closed linear span of their products.
Since $\tilde\phi(D_\om(G))B\sbe B$, we have
\begin{multline*}
B\supseteq \tilde\phi (\contz(G_{\om},\iota))B\supseteq
\tilde\phi (\contz(G_{\om},\iota))\tilde\phi (\overline{\contz(G_{\om},\iota)})B\\
=\tilde\phi \big(\contz(G_{\om},\iota)\overline{\contz(G_{\om},\iota)}\big)B=\phi(\contz(G))B=B
\end{multline*}
and therefore $\tilde\phi (\contz(G_{\om},\iota))B=B$. This easily implies that the map $\Phi$  in \eqref{eq:rep-L-om} is an isomorphism of 
Hilbert $B$-modules which intertwines the left action $\phi_\om$ of $C_0(G)$ on $\L(G_\om,B)$ with $\phi$.
It therefore induces an isomorphism $(B_\om,\phi_\om)\cong (B,\phi)$. 
It remains to show that $\Phi$ transforms the action  $\beta_{\om}=\Ad(\tilde\rt\otimes \beta)$ on $B_\om$ to the action 
 $\Ad U_\om\circ\beta$ on $B$ with $U_\om=\tilde\phi\circ u_\om$. To see this
note first that the left $B$-valued inner product on elementary vectors $f\otimes a, h\otimes b\in C_0(G_\om,\iota)\otimes_{C_0(G)}B$ 
has the form
$${_B\lk f\otimes a, h\otimes b\rk}= \tilde\phi(f\circ \mathfrak{s})ab^*\tilde\phi(h\circ\mathfrak{s})$$
and the action $\Ad(\tilde\rt\otimes\beta)$ is given on such elements by the formula
$$\Ad(\tilde\rt\otimes \beta)_s(\tilde\phi(f\circ \mathfrak{s})ab^*\tilde\phi(h\circ\mathfrak{s}))=\tilde\phi(\tilde\rt_{(s,1)}(f)\circ\mathfrak{s})\beta_s(ab^*)\tilde\phi(\overline{\tilde\rt_{(s,1)}(h)\circ\mathfrak{s}}).$$
But for $f\in C_0(G_\om,\iota)$ we can compute:
\begin{align*}
\big(\tilde\rt_{(s,1)}(f)\circ \mathfrak{s}\big)(t)=&f((t,1)(s,1))=f(ts, \om(t,s))=\overline{\om(t,s)}f(ts,1)\\
&=\big({u_\om(s)}\rt_s(f\circ \mathfrak{s})\big)(t).
\end{align*}
Therefore, using $U_\om=\tilde\phi\circ u_\om$ and the $\rt-\beta$ equivariance of $\tilde\phi$,  we now conclude
\begin{align*}
\Ad(\tilde\rt\otimes \beta)_s\big(\tilde\phi(f\circ\mathfrak{s})ab^*\tilde\phi(h\circ\mathfrak{s}))\big)&=\tilde\phi({u_\om}(s)\rt_s(f))\beta_s(ab^*)\tilde\phi(\rt_s(h)u_\om(s)^*)\\
&=\Ad U_\om(s)\circ \beta_s(_B\lk f\otimes a, h\otimes b\rk).
\end{align*}
and the result follows. 
\end{proof}

We now show that Proposition~\ref{pro:condition-iso-B=Bom} applies for dual coactions and describe the expected 
result of the deformed coactions in this case. 

\begin{proposition}\label{pro:dual-coactions-deformation}
Let $(A,\delta)=(A_0\rtimes_{\alpha,\mu}G,\dual\alpha)$ be a dual coaction where $\alpha$ is some $G$-action on a \cstar{}algebra $A_0$ and $\rtimes_\mu$ is a duality crossed-product functor. Then for every $2$-cocycle $\omega$, the deformed coaction  $(A^\om_\mu,\delta^\om_\mu)$ is canonically isomorphic to $(A_0\rtimes_{\alpha_\omega,\mu}G,\dual{\alpha_\om})$, where 
$\alpha_\om=(\alpha,\iota)$ is the action $\alpha$ tensored with the twisted action $(\id, \iota)$ of 
$(G_\omega,\T)$ on $\C$ as in Remark~\ref{rem-twist}.    
\end{proposition}
\begin{proof}
First recall that $(A_0\rtimes_{\alpha,\mu}G,\dual\alpha)$ is a $\mu$-coaction. Its crossed product is
$$B=A_0\rtimes_{\alpha,\mu}G\rtimes_{\dual\alpha}\dualG\cong A_0\otimes\K(L^2(G))$$
and the bidual action $\beta=\dual{\dual\alpha}$ corresponds to $\alpha\otimes\Ad\rho$.
Moreover, the structure homomorphism $\phi\colon \contz(G)\to \M(B)$ corresponds to $1\otimes M$, where 
$M\colon \contz(G)\to\M(\K(L^2(G)))=\B(L^2(G))$ denotes the multiplication representation. 
This extends to $L^\infty(G)$ so that Proposition~\ref{pro:condition-iso-B=Bom} applies for our system
$(B,\beta,\phi)$. Then  $U_\om(s)$ corresponds to multiplication by $t\mapsto \bar\omega(t,s)$ and therefore 
the action $\beta=\alpha\otimes\Ad\rho$ gets deformed into $\beta^{\bar\om}=\alpha\otimes \Ad\rho^{\bar\om}$, with $\rho^{\bar\om}$ defined as in Remark~\ref{right-regular}. 
But it follows from \cite{Quigg:Duality_twisted}*{Theorem~3.1} or \cite{Quigg:Full}*{Theorem~3.6} that this is precisely the action on $A_0\otimes \K(L^2(G))$ which corresponds to the bidual action on 
$A_0\rtimes_{\alpha_\om,\mu}G\rtimes_{\dual\alpha_\om}\dualG$ via the Imai-Takai duality isomorphism
$$A_0\rtimes_{\alpha_\om,\mu}G\rtimes_{\dual\alpha_\om}\dualG\cong A_0\otimes\K(L^2(G)).$$
Since the dual coaction $\dual\alpha_\om$ on $A_0\rtimes_{\alpha_\om,\mu}G$ is a $\mu$-coaction (this
is a special case of the situation described in~\ref{eq:mu-crossed-product-Fell-bundle}), it follows that $(A^\om_\mu,\delta^\om_\mu)\cong (A_0\rtimes_{\alpha_\om,\mu}G,\dual\alpha_\om)$.
\end{proof}

\begin{example}
As a special class of examples, we get that all twisted exotic group algebras are deformations of the untwisted group algebras: given any duality crossed-product functor (like the maximal or the reduced) $\rtimes_\mu$ for a locally 
compact group $G$, consider its group \cstar{}algebra $C^*_\mu(G)=\C\rtimes_{\id,\mu} G$. This is a special case of Proposition~\ref{pro:dual-coactions-deformation}, so that the $\omega$-deformation 
of $(C^*_\mu(G),\delta_\mu)$ is (isomorphic to) 
$(C^*_\mu(G,\omega),\delta_\mu^\omega)=(\C\rtimes_{\id_\om,\mu}G,\dual{\id_\om})$, where we write
$\id_\om$ for the twisted action $(\id,\iota)$ of the pair $(G_\om,\T)$.
\end{example}

\subsection{Compatibility with products}
We are now going to draw some consequences out of our 
construction of deformed cosystems by twists $\sigma$ for $G$.
The first one is an analogue of \cite{BNS}*{Proposition 4.4}. 

\begin{proposition}\label{prop-cocycleproduct}
Let $(B,\beta,\phi)$ be a weak $G\rtimes G$-algebra. 
Then 
\begin{equation}\label{eq-deform-tensorproduct}
(B_{\sigma\cdot\sigma'}, \beta_{\sigma\cdot\sigma'}, \phi_{\sigma\cdot\sigma'})\cong ((B_{\sigma'})_{\sigma}, (\beta_\sigma')_{\sigma}, (\phi_{\sigma'})_{\sigma})
\end{equation}
for every pair of twists $\sigma=(\T\into G_{\sigma}\onto G), \sigma'=(\T\into G_{\sigma'}\onto G)$.  As a consequence, we  get
\begin{equation}\label{eq-iterateom}
(A_\mu^{\sigma\cdot\sigma'}, \delta_\mu^{\sigma\cdot\sigma'})\cong ((A^{\sigma'}_\mu)^{\sigma}_\mu, (\delta_\mu^{\sigma'})_\mu^{\sigma})
\end{equation}
starting with any $\mu$-coaction $(A,\delta)$ for a duality functor $\rtimes_\mu$.
\end{proposition}
\begin{proof}
Let $\sigma=(\T\into G_\sigma\onto G)$ and $\sigma'=(\T\into G_{\sigma'}\onto G)$ be given and recall
the construction of the extension product $\sigma\cdot\sigma'=(\T\into G_{\sigma\cdot \sigma'}\onto G)$ with
$G_{\sigma\cdot\sigma'}=G_{\sigma}*G_{\sigma'}=(G_\sigma\times_GG_{\sigma'})/\T$ as in \eqref{eq-productoftwists}.
Consider the map
\begin{equation}\label{iso-linebundle}
\begin{split}
\Theta: C_0(G_\sigma, \iota)\otimes_{C_0(G)} C_0(G_{\sigma'}, \iota)\mapsto C_0(G_{\sigma\cdot\sigma'}, \iota); \\
\Theta(\xi\otimes \eta)([g_1,g_2])= \xi(g_1)\eta(g_2).
\end{split}
\end{equation}
To see that $\Theta$ is a well-defined module map, we show that it preserves the left and right $C_0(G)$-valued inner products. 
For this  let $\xi,\xi'\in C_0(G_\sigma,\iota)$ and $\eta, \eta'\in C_0(G_{\sigma'},\iota)$ be given. Then, writing $g:=q([g_1,g_2])$ (which equals $q(g_1)=q(g_2)$ by construction
of $\sigma\cdot\sigma'$), we compute
\begin{align*}
&\braket{\xi\otimes \eta}{\xi'\otimes \eta'}_{C_0(G)}(g)=\braket{\eta}{\braket{\xi}{\xi'}_{C_0(G)}\eta'}_{C_0(G)}(g)\\
&= \overline{\eta{(g_2}}(\braket{\xi}{\xi'}_{C_0(G)}\eta')(g_2)=\overline{\eta{(g_2)}}\overline{\xi(g_1)}\xi'(g_1)\eta'(g_2)\\
&=\overline{\Theta(\xi\otimes \eta)([g_1,g_2])}\Theta(\xi'\otimes \eta')([g_1,g_2])=\braket{\Theta(\xi\otimes\eta)}{\Theta(\xi'\otimes\eta')}_{C_0(G)}(g).
\end{align*}
A similar computation shows that it also preserves the right inner products and 
it is clear that $\Theta$ is compatible with the left and right actions of $C_0(G)$. 
It is routine to check that $\Theta$ has dense image. Hence it is an isomorphism 
of equivalence bimodules. The group $G_{\sigma\cdot\sigma'}$ acts on 
$C_0(G_\sigma, \iota)\otimes_{C_0(G)} C_0(G_{\sigma'}, \iota)$ via the diagonal action 
$$\tilde\rt_{[g_1,g_2]}(\xi\otimes \eta)=\tilde\rt^\sigma_{g_1}(\xi)\otimes\tilde\rt^{\sigma'}_{g_2}(\eta)$$
and it is easy to check that $\Theta$ transforms this action to the right translation action
$\tilde\rt^{\sigma\cdot\sigma'}$ of $G_{\sigma\cdot\sigma'}$ on $C_0(G_{\sigma\cdot\sigma'}, \iota)$.
It follows from this that we obtain an $\tilde\rt^\sigma\otimes\tilde\rt^{\sigma'}\otimes \beta-\tilde\rt^{\sigma\cdot\sigma'}\otimes\beta$ equivariant
isomorphism $\Theta\otimes\id_B$ of right Hilbert $B$-modules
$$C_0(G_\sigma, \iota)\otimes_{C_0(G)} C_0(G_{\sigma'},\iota)\otimes_{C_0(G)}B\cong C_0(G_{\sigma\cdot\sigma'}, \iota)\otimes_{C_0(G)}B,$$
which also preserves the left action of $C_0(G)$ on these modules.
It follows that 
\begin{align*}
B_{\sigma\cdot \sigma'}&=\K(C_0(G_{\sigma\cdot\sigma'}, \iota)\otimes_{C_0(G)}B)
\cong \K(C_0(G_\sigma, \iota)\otimes_{C_0(G)} C_0(G_{\sigma'}, \iota)\otimes_{C_0(G)}B))\\
&\cong \K(C_0(G_\sigma, \iota)\otimes_{C_0(G)} \K(C_0(G_{\sigma'},\iota)\otimes_{C_0(G)}B))\\
&=\K(C_0(G_\sigma, \iota)\otimes_{C_0(G)}B_{\sigma'})
=(B_{\sigma'})_\sigma
\end{align*}
and all isomorphisms and identifications above are $G$-equivariant and preserve the left 
actions of $C_0(G)$. 
\end{proof}

\subsection{Nuclearity}
In this section we give conditions for nuclearity of deformed \cstar{}algebras.
One of the first results in this direction was obtained by Rieffel in \cite{Rief-K}*{Theorem 4.1}, where he proved that the deformed \cstar{}algebra $A^\omega$ by an action $\beta:\R^n\car A$ via a (continuous) $2$-cocycle $\omega$ on $\R^n$ 
is nuclear if and only if $A$ is nuclear. This result was extended to deformations by actions of abelian groups by Kasprzak in \cite{Kasprzak:Rieffel}*{Theorem~3.10}. In \cite{Yamashita}*{Proposition~12}, Yamashita proved a version of this result for coactions of discrete groups, 
assuming Exel's approximation property for the underlying Fell bundle. We are now
going to generalize all these results to coactions of locally compact groups below.

Below we shall use the theory of amenable actions of locally compact groups on \cstar{}algebras as introduced in \cite{BEW:amenable}. 
We shall give specific references for the results we need throughout the proof and refer to \cite{BEW:amenable} for further details.

\begin{theorem}\label{theo:nuclear}
    Let $(A,\delta)$ be a $\mu$-coaction of $G$ for some correspondence crossed-product functor $\rtimes_\mu$,
   and let $(B,\beta,\phi)=(A\rtimes_\delta\dualG,\dual\delta,j_G)$ be the corresponding weak $G\rtimes G$-algebra. Assume that $\beta$ is an amenable action. Then for every twist $[\sigma]\in \Twist(G)$, we have
   \begin{equation}\label{eq:amenable-equal}
        A^\sigma:=A^\sigma_\max=A^\sigma_\mu=A^\sigma_\red.
    \end{equation}
    and  $A^\sigma$ is nuclear if and only if $A$ is nuclear. 
\end{theorem}
\begin{proof} Let $\alpha=\Ad\tildeV\colon G\car \K(\H)$ be an action corresponding to $\sigma$ as in 
Corollary~\ref{cor-twist-action}.
We first notice that if $\beta$ is amenable, then so is $\beta\otimes\alpha:G\car B\otimes \K(\H)$ by  \cite{BEW:amenable}*{Theorem 5.16}.
Since amenability is preserved by equivariant Morita equivalence by
\cite{BEW:amenable}*{Proposition 3.20}, it follows that the action $\beta_\sigma:G\car B_\sigma$ is also amenable. 
By \cite{BEW:amenable}*{Proposition 5.10} we know 
that all crossed products by amenable actions coincide. Hence 
$$B_\sigma\rtimes_{\beta_\sigma,\max}G=B_\sigma\rtimes_{\beta_\sigma,\mu}G=B_\sigma\rtimes_{\beta_\sigma,\red}G.$$ 
By construction, $A^\sigma_\mu$ is the $\mu$-generalized fixed-point algebra $(B_{\sigma})^{(G,\beta_\sigma)}_\mu$
of the weak $G\rtimes G$-algebra $(B_\sigma,\beta_\sigma, \phi_\sigma)$, which is the completion of the generalized fixed-point algebra $(B_\sigma)_c^G$
by a norm which only depends on the norm on $B_\sigma\rtimes_{\beta_\sigma,\mu}G$. Hence, the above equality of crossed products implies
~\eqref{eq:amenable-equal}.

Now assume that $A$ is nuclear. Then so is $B=A\rtimes_\delta\dualG$.
Indeed, this follows by duality as 
$B\rtimes_{\beta,r}G=A\rtimes_\delta \dualG\rtimes_{\dual\delta,\red}G\cong A_r\otimes\K(L^2(G))$, where $(A_r,\delta_r)$ denotes the normalization of $(A,\delta)$. But $A_r$, being a quotient of $A$, is nuclear. Therefore $B\rtimes_{\beta,\red}G$ is nuclear.
This can only be true if $B$ is nuclear, by \cite{BEW:amenable}*{Theorem~7.2}. Since $\beta\otimes \alpha$ is amenable, 
the crossed product $(B\otimes\K(\H))\rtimes_{\beta\otimes\alpha}G$ is nuclear, again by \cite{BEW:amenable}*{Theorem~7.2}. From this it follows that the Morita equivalent crossed product $B_\sigma\rtimes_{\beta_\sigma}G$ is also nuclear. Since $A^\sigma$ is 
Morita equivalent to $B_\sigma\rtimes_{\beta_\sigma}G$, it is nuclear as well. 

Conversely, if $A^\sigma$ is nuclear, we can apply the above argument to the deformation 
$$A=A^{\id}\cong A^{\bar\sigma\cdot{\sigma}}\cong (A^\sigma)^{\bar\sigma}$$
of $A^\sigma$, where $\bar\sigma$ denotes the inverse of $\sigma$ as in \eqref{eq-inversetwist}. Then $A^\sigma$ nuclear implies 
$A$ nuclear.
\end{proof}

\begin{remark}
If $G$ is amenable, then all actions of $G$ are amenable, so the above result applies. In particular it generalizes the previous results by Rieffel and Kasprzak for $\R^n$ or abelian groups (\cites{Rief-K,Kasprzak:Rieffel}) mentioned above.

If $G$ is discrete and $(A,\delta)$ is a $G$-coaction with $A$ nuclear, there exists a Fell bundle $\mathcal A$ over $G$ such that 
$A_r=C_{\red}^*(\A)$, the reduced cross-sectional $C^*$-algebra of $\mathcal A$, from which it follows that 
$\A$ has Exel's approximation property (see \cite{BEW:amenable}*{Corollary 4.11}). In \cite{Yamashita} Yamashita describes a deformed Fell bundle $\mathcal A^\sigma$ (generalized to non-discrete groups in \cite{BE:Fellbundles}) such that our deformed algebra $A^\sigma$ is 
a cross-sectional $C^*$-algebra of  $\mathcal A^\sigma$. It follows then from \cite{Yamashita}*{Lemma 10} that $\mathcal A^\sigma$ also has the approximation property, from which it follows that $A^\sigma$ is nuclear as well (see \cite{Yamashita}*{Proposition 11}). 
Note  that the approximation property for $\mathcal A$ implies that   the dual action $\beta=\dual\delta$ on $B=A\rtimes_\delta\dualG$ is automatically amenable, so Theorem \ref{theo:nuclear} also applies in this situation.

We do not know whether a similar result holds true for general locally compact groups.  It would mean that for any action $\beta:G\car B$
with  $B\rtimes_rG$ nuclear, the crossed product $(B\otimes\K(\H))\rtimes_{\beta\otimes \alpha,r}G$  is nuclear as well. Note that, in general, nuclearity of $B\rtimes_{\beta,r}G$ does not imply amenability of $\beta$ -- counterexamples 
are given by crossed products $\C\rtimes_rG=C_r^*(G)$ for nonamenable (almost) connected groups $G$ (e.g., $G=\SL(2,\R)$), 
which are nuclear by a famous theorem of Connes \cite{Connes-Classification}.
A special version of this question would ask the following: if $G$ is a locally compact group 
such that $C^*_r(G)$ is nuclear, does it follow that the twisted group algebra $C_r^*(G,\om)$ is nuclear for every cocycle $\om\in Z^2(G,\T)$?
\end{remark}

\section{Continuity}\label{sec:continuity}

In this section we want to consider the question of how the deformed algebras depend 
on the parameters in a `continuous' way.  We start by introducing a notion of a bundle of \cstar{}algebras, which is slightly more general than the one introduced by Kirchberg and Wassermann in \cite{KW:continuousbundles}.
Recall first that a $C_0(X)$-algebra is a \cstar{}algebra $\A$ equipped with 
 a nondegenerate \Star{}homomorphism $\Phi:C_0(X)\to Z\M(\A)$,  the center of the multiplier of $\A$.

\begin{definition}\label{def-bundles}
A {\em bundle of \cstar{}algebras} $(\A, X, q_x:\A\to A_x)$ over the locally compact space $X$ consists of 
a $C_0(X)$-algebra $\A$ together with a collection of quotient maps $q_x:\A\to A_x; x\mapsto a_x$, for all  $x\in X$, such that
\begin{itemize}
    \item[{(*)}] For all $f\in C_0(X)$, $a\in \A$, and $x\in X$, we have $(\Phi(f)\cdot a)_x=f(x)a_x$.
\end{itemize}
We say that $(\A, X, q_x:\A\to A_x)$ is \emph{faithful}, if $\|a\|=\sup_{x\in X}\|a_x\|$ for all $a\in \A$.
A faithful {\em bundle of \cstar{}algebras} $(\A, X, q_x:\A\to A_x)$ is said to be  {\em continuous} 
 (resp. \emph{upper semicontinuous}) 
    if for all $a\in \A$ the map $x\mapsto \|a_x\|$ is continuous (resp. upper semicontinuous).
\end{definition}

Notice that a faithful bundle of \cstar{}algebras $(\A, X, q_x:\A\mapsto A_x)$ in our sense is what Kirchberg and Wassermann \cite{KW:continuousbundles} call a bundle of \cstar{}algebras. Since we require faithfulness in our definitions of (semi)continuous bundles of \cstar{}algebras, these notions coincide with the ones given in \cite{KW:continuousbundles}. 

Note that every $C_0(X)$-algebra $\A$ determines a `maximal' {\em bundle of \cstar{}algebras} $(\A, X, q^{\max}_x:\A\to A_x^{\max})$
by defining $A_x^{\max}:=\A/I_x$ with $I_x:=\Phi(C_0(X\setminus \{x\}))\A$. 
From the definition of a bundle of \cstar{}algebras
$(\A, X, q_x:\A\to A_x)$, it is clear that the ideal $I_x$ lies in the kernel of $q_x$, hence we see that $A_x$ must be a quotient of 
$A^{\max}_x$. It follows from a combination of \cite{Dana:book}*{Theorems C.25 and C.26} that $A_x^{\max}=A_x$ for every $x\in X$ if and only if $(\A,X, q_x:\A\to A_x)$ is an upper semicontinuous bundle of \cstar{}algebras. 
If we combine this with \cite{KW:continuousbundles}*{Lemma~2.3} and \cite{Dana:book}*{Proposition~C.5} we also see that  $(\A, X, q_x:\A\to A_x)$ is upper 
semicontinuous if and only if $\Prim(A_x)=\Prim(A_x^\max)$ for all $x\in X$, and this is equivalent to
\begin{equation}\label{eq-Primbundle}
\Prim(\A)=\cup_{x\in X}\Prim(A_x).
\end{equation}
Here we identify $\Prim(A/J)$ of a quotient $A/J$ of $A$  with the set of primitive ideals of $A$  containing $J$. 
If $(\A, X, q_x:\A\to A_x)$ is upper semicontinuous, the  therefore well-defined map
$$\varphi:\Prim(\A)\to X; \quad\varphi(P)=x\Leftrightarrow P\in \Prim(A_x)$$
is continuous
and it follows from \cite{Dana:book}*{Theorem~C.26} that $(\A, X, q_x:\A\to A_x)$ is a continuous bundle
if and only if the map $\varphi:\Prim(\A)\to X$ is also open.

\begin{example}\label{ex-crossed-bundle}
A typical example of a bundle of \cstar{}algebras, which plays a central role in this paper, is given as follows:
suppose that $\B$ is any $C_0(X)$-algebra with associated upper semicontinuous bundle $(\B, X, q_x:\B\to B_x)$ as explained above 
(i.e., we have $B_x=B^{\max}_x$). Let $\beta:G\car \B$ be a $C_0(X)$-linear action of the locally compact group $G$ on 
$\B$, i.e., we have $\beta_g(\Phi(f)b)=\Phi(f)\beta_g(b)$ for all $f\in C_0(X), b\in \B$. Then $\beta$ induces an action 
$\beta^x:G\car B_x$ on each fibre $B_x$ by defining
$$\beta^x_g(b_x):=(\beta_g(b))_x.$$
Now let $\rtimes_\mu$ be any (exotic) crossed-product functor for $G$. Then the crossed product 
$\B\rtimes_{\beta,\mu}G$ carries a canonical structure of a $C_0(X)$-algebra via the composition 
\begin{equation}\label{eq-bundle-crossed}
\Psi:=i_{\B}\circ \Phi: C_0(X)\stackrel{\Phi_X}{\longrightarrow} \M(\B)\stackrel{i_{\B}}{\longrightarrow} \M(\B\rtimes_{\beta,\mu}G),\end{equation}
where the second map is the unique extension of the inclusion $i_{\B}: \B\to \M(\B\rtimes_{\beta,\mu}G)$ to $\M(\B)$. Since the image of $\Phi$ lies in the center $Z\M(\B)$ and is invariant under the $G$-action by our assumptions, it is easily checked that the image of $\Psi$ lies in $Z\M(\B\rtimes_{\beta,\mu}G)$.
We then obtain a bundle of \cstar{}algebras 
$$(\B\rtimes_{\beta,\mu}G, X, q_x\rtimes_\mu G:\B\rtimes_{\beta,\mu} G\to B_x\rtimes_{\beta^x,\mu} G)$$
with fibres $B_x\rtimes_{\beta^x,\mu} G$ as in our definition. 
It is not clear to us, whether this bundle is always faithful, although it is true very often (see Proposition~\ref{prop-crossed-product-bundle} below).
\end{example}

For general crossed-product functors, we do not know to much about the continuity properties of the crossed-product bundles 
from Example~\ref{ex-crossed-bundle}. But we have the following result which is essentially due to Kirchberg and Wassermann \cite{KW:exact-groups}:

\begin{proposition}\label{prop-crossed-product-bundle}
Suppose that $\beta:G\car \B$ is a $C_0(X)$-linear action of a locally compact group on the $C_0(X)$-algebra $\B$ and let $\rtimes_\mu$ be any crossed-product functor. 
Then the following are true:
\begin{enumerate}
    \item if $\rtimes_\mu$ is an exact crossed-product functor (e.g., if $\rtimes_\mu=\rtimes_{\max}$ or if $\rtimes_\mu=\rtimes_r$ and $G$ is an exact group), then $(\B\rtimes_{\beta,\mu}G, X, q_x\rtimes_\mu G:\B\rtimes_{\beta,\mu} G\to B_x\rtimes_{\beta^x,\mu} G)$
    is upper semicontinuous.
    \item If $G$ is an exact group and $(\B,X, q_x:\B\to B_x)$ is a continuous bundle, then the bundle of reduced crossed products
     $(\B\rtimes_{\beta,r}G, X, q_x\rtimes_r G:\B\rtimes_{\beta,r} G\to B_x\rtimes_{\beta^x,r} G)$ is continuous as well.
\end{enumerate}
\end{proposition}
\begin{proof} 
The proof of item (1) is an easy exercise using the exactness of 
$$0\to I_x\rtimes_{\beta,\mu}G\to \B\rtimes_{\beta,\mu}G\to B_x\rtimes_{\beta_x,\mu}G\to 0,$$
with $I_x=\Phi(C_0(X\setminus\{x\})\B$ and the equation
$$I_x\rtimes_{\beta,\mu}G=i_{\B}(\Phi(C_0(X\setminus\{x\}))\big(\B\rtimes_{\beta,\mu}G\big).$$ 
The latter follows from
$$i_{\B}(\Phi(C_0(X\setminus\{x\}))\big(C_c(G)\odot \B)=C_c(G)\odot i_{\B}(\Phi(C_0(X\setminus\{x\}))\B=C_c(G)\odot I_x,$$
for the algebraic tensor products, which lie densely in the respective crossed products. 
The second item has been shown in \cite{KW:exact-groups}. 
\end{proof}

To obtain our main results, we also need to know that the continuity properties are preserved under Morita equivalence of bundles. Although the following lemma is more or less folklore, we include a short proof for completeness.

\begin{lemma}\label{lem-Morita-bundles}
Suppose that $(\A,X,q^{\A}_x:\A\to A_x)$ and $(\B, X, q^{\B}_x:\B\to B_x)$ are bundles of \cstar{}algebras over $X$.
Let $\X$ be an $\A-\B$  equivalence bimodule which is $C_0(X)$-linear in the sense that
$$\Phi_{\A}(f)\xi=\xi\Phi_\B(f)\quad \forall \xi\in \X, f\in C_0(X),$$
where $\Phi_\A$ and $\Phi_\B$ denote the structure maps for $\A$ and $\B$, respectively. Suppose further that 
for each $x\in X$ the module $\X$ factors through an $A_x-B_x$ equivalence bimodule $\X_x$. (We then call $\X$ 
 an
$(\A, X, q_x^\A:\A\to A_x)-(\B, X, q^{\B}_x:\B\to B_x)$ {\em equivalence bimodule}.)
Then the following hold
\begin{enumerate}
    \item $(\A,X,q^{\A}_x:\A\to A_x)$ is faithful if and only if $(\B, X, q^{\B}_x:\B\to B_x)$ is faithful.
    \item $(\A,X,q^{\A}_x:\A\to A_x)$ is continuous (resp. upper semicontinuous)
    if and only if the same holds for $(\B, X, q^{\B}_x:\B\to B_x)$.
\end{enumerate}
\end{lemma}

\begin{proof}
Let $I:=\{b\in \B: \forall x\in X: \|b_x\|=0\}$. Then $I$ is a closed ideal in $\B$ and $(\B, X, q^{\B}_x:\B\to B_x)$
is faithful if and only if $I=\{0\}$. The ideal $I$ is clearly invariant under the action of $C_0(X)$.
By the Rieffel correspondence and $C_0(X)$-linearity of $\X$, the submodule $\X I$ induces a Morita equialence 
between the ideal $J:=\overline{_{\A}\braket{\X I}{\X I}}$ of $\A$ and $I$. In particular, $J=\{0\}$ if and only $I=\{0\}$.
But for all $\xi,\eta\in \X I$, we have
\begin{equation}\label{eq-innerfaithful}
\|(_\A\braket{\xi}{\eta})_x\|=\|_{A_x}\braket{\xi_x}{\eta_x}\|=\|\braket{\xi_x}{\eta_x}_{B_x}\|=\|\big(\braket{\xi}{\eta}_\B)_x\|,
\end{equation}
where $\xi_x, \eta_x\in \X_x$ denote the images of $\xi,\eta$ under the quotient map.
Since the linear combinations of those inner products are dense in $J$ and $I$, respectively, it follows 
 that $\|a_x\|=0$ for all  $a\in J$ and all $x\in X$. It follows that 
$(\A,X,q^{\A}_x:\A\to A_x)$ is not faithful if  $(\B, X, q^{\B}_x:\B\to B_x)$ is not faithful. By symmetry of Morita equivalence, this proves item (1). 

In order to prove item (2), we first recall that $(\B, X, q^{\B}_x:\B\to B_x)$ is upper semicontinuous
if and only if $B_x=\B/I_x$ with $I_x:=C_0(X\setminus\{x\})\B$ for all $x\in X$. Write $J_x:=C_0(X\setminus\{x\})\A$. 
The $C_0(X)$-linearity of $\X$ implies that 
$Y_x:=C_0(X\setminus \{x\})\cdot\X$ is a $J_x-I_x$ equivalence bimodule, such that $\X/Y_x$ becomes an
$\A/J_x-\B/I_x$ equivalence bimodule. By our assumptions, this factors further to the $A_x-B_x$ bimodule $\X_x$ of the lemma.
Thus it follows from the Rieffel correspondence, that $A_x=\A/J_x$ if and only if $B_x=\B/I_x$ for all $x\in X$.
This proves the case of upper semicontinuous bundles. 

Finally, for continuity we use the fact that the Morita equivalence $\X$ induces a homeomorphism
between $\Prim(\A)$ and $\Prim(\B)$ by sending a primitive ideal $P\in\Prim(B)$ to the ideal 
$Q:=\overline{_\A\braket{\X P}{\X P}}$ of $\A$. By $C_0(X)$-linearity it is easy to see that this bijection 
maps $\Prim(B_x)$ onto $\Prim(A_x)$ for all $x\in X$. Thus it follows that the canonical map
$\varphi_\B:\Prim(\B)\to X$ is open if and only if the canonical map $\varphi_\A:\Prim(\A)\to X$ is open as well.
By the discussion preceding Example~\ref{ex-crossed-bundle} above, this finishes the proof.
\end{proof}

Suppose now that $\B$ is a \cstar{}algebra equipped with a nondegenerate \Star{}homomor\-phism
$\Phi:C_0(X\times G)\to \M(\B)$ such that the extension of $\Phi$ to $C_b(X\times G)$ sends all functions which are constant in the $G$-direction into the center $Z\M(\B)$ of the multiplier algebra.
In this way, $\Phi$ `restricts' to a nondegenerate \Star{}homomorphism $\Phi_X:C_0(X)\to Z\M(\B)$ 
which turns $\B$ into a $C_0(X)$-algebra and we obtain the upper semicontinuous bundle 
$(\B, X, q_x:\B\to B_x)$ with `maximal' fibres $B_x=\B/I_x$, where  $I_x=C_0(X\setminus\{x\})\B$.
Assume further that there is a $C_0(X)$-linear (i.e., a fibrewise) action $\beta:G\car \B$ 
such that $\Phi$ is $\rt-\beta$ equivariant, where here $\rt:G\car C_0(X\times G)$ denotes right translation in the $G$-variable leaving $X$ invariant. Then $\Phi$ induces $\rt-\beta_x$ equivariant structure maps 
$$\phi_x: C_0(\{x\}\times G)\to \M(B_x); \quad \phi_x(f(x,\cdot))=\Phi(f)+I_x$$
for all $x\in X$.
Thus, identifying $\{x\}\times G$ with $G$ we obtain a family of weak $G\rtimes G$-algebras $(B_x, \beta_x,\phi_x)$ indexed by $x\in X$.

\begin{definition}\label{def-semicontbundle}
A triple  $(\B, \beta, \Phi)$ as above is called an \emph{upper semicontinuous bundle of weak $G\rtimes G$-algebras}  over $X$
with fibres  $(B_x, \beta_x,\phi_x)$. If, in addition, for each $b\in \B$ the function $x\mapsto \|b_x\|$ is continuous
on $X$, then we call $(\B, \beta, \Phi)$ a \emph{continuous bundle of weak $G\rtimes G$-algebras}  over $X$.
\end{definition}

\begin{example}\label{eq-trivial}
The most obvious example is given by the trivial \cstar{}algebra bundle $\B=C_0(X,B)$ over $X$ with constant fibre $B_x=B$ for all $x\in X$, together with some fixed nondegenerate \Star{}homomorphism $\phi:C_0(G)\to \M(B)$ and a $C_0(X)$-linear action 
$\beta: G\car C_0(X,B)$ such that $\phi$ is $\rt-\beta^x$ equivariant for each fibre action $\beta^x: G\car B$.
Then  $(C_0(X,B), \beta, \Phi)$, with $\Phi:=\id_X\otimes \phi$,  is a  continuous family 
of weak $G\rtimes G$ algebras with fibres $(B, \beta^x,\phi)$. 
\end{example}

Given a fixed duality crossed-product functor $\rtimes_\mu$ on $G$, if $(\B,\beta,\Phi)$ is an upper semicontinuous 
family of weak $G\rtimes G$-algebras with fibres $(B_x, \beta_x, \phi_x)$ we obtain a family of cosystems 
$x\mapsto (A^x_\mu, \delta^x_\mu)$ via (exotic) Landstad duality.
The restriction, say $\Phi_G$, of $\Phi:C_b(X\times G)\to \M(\B)$ to $C_0(G)$,
gives $\B$ the structure of a weak $G\rtimes G$-algebra $(\B, \beta, \Phi_G)$ and, for any duality crossed-product functor $\rtimes_\mu$ on $G$, we also obtain a cosystem $(\A_\mu,\delta_\mu)$ via  Landstad duality. 
In what follows next, we want to study the question, under what conditions 
$(\A_\mu,\delta_\mu)$ becomes a (semi)continuous bundle of coactions with fibres $(A^x_\mu, \delta^x_\mu)$. 
Recall from Example~\ref{eq-bundle-crossed} that the crossed product 
$\B\rtimes_{\beta,\mu}G$ gives rise to a canonical bundle of \cstar{}algebras  
$(\B\rtimes_{\beta,\mu}G, X, q_x\rtimes_\mu G:\B\rtimes_\mu G\to B_x\rtimes_{\mu}G)$.

\begin{theorem}\label{thm-bundle1}
Let $(\B, \beta,\Phi)$ be an upper semicontinuous family of weak $G\rtimes G$-algebras $(B_x,\beta_x,\phi_x)$ over $X$ as above. 
Fix a duality crossed-product functor $\rtimes_\mu$ for $G$ and let
$(\A_\mu, \delta_\mu)$ be as above. Then there exists a unique nondegenerate \Star{}homomorphism $\Psi_{\A_\mu}:C_0(X)\to Z\M(\A_\mu)$, giving $\A_\mu$ the structure of a $C_0(X)$-algebra, such that the following hold:
\begin{enumerate}
\item The $(\A_\mu, \delta_\mu)-(\B\rtimes_{\beta,\mu}G, \widehat{\beta}_\mu)$ equivalence bimodule  $(\F_\mu(\B), \delta_{\F_\mu(\B)})$ of Proposition~\ref{prop-Landstad} factors for each $x\in X$ through the $(A^x_\mu, \delta^x_\mu)-(B_x\rtimes_{\beta_x,\mu}G, \widehat{\beta_x}_\mu)$ equivalence bimodule $(\F_\mu(B_x), \delta_{\F_\mu(B_x)})$.
\item For all $x\in X$ there are canonical $\delta_\mu-\delta^x_\mu$ equivariant quotient maps $q_x^\A:\A_\mu\to A^x_\mu$ and there is a canonical structure map $\Phi_\A:C_0(X)\to Z\M(\A_\mu)$  
giving   $(\A_\mu, X, q_x^\A:\A_\mu\to A_\mu^x)$ the structure of a bundle of \cstar{}algebras.
  \item The $\A_\mu-\B\rtimes_{\beta,\mu}G$ equivalence bimodule 
$\F_\mu(\B)$ induces a
$$(\A_\mu, X, q_x^\A:\A_\mu\to A_\mu^x)-(\B\rtimes_{\beta,\mu}G, X, q_x\rtimes_\mu G:\B\rtimes_\mu G\to B_x\rtimes_{\mu}G)$$
Morita equivalence as in Lemma~\ref{lem-Morita-bundles}. 
\end{enumerate}
\end{theorem}
\begin{proof}
The proof of this theorem follows almost directly from the construction of the module $\F_\mu(\B)$ as explained briefly in Section~\ref{sec:Landstad}. Indeed, since the image of $\Psi_G(C_0(G))$ commutes with the image of $\Psi_X(C_0(X))$ inside $\M(\B)$, 
there is a nondegenerate \Star{}homomorphism
$$\Psi_{\F_\mu(\B)}: C_0(X)\to \mathcal L(\F(\B)_\mu),\quad  \Psi_{\F(\B)_\mu}(f)b:= \Psi_X(f)b$$
for $f\in C_0(X)$, $b\in \F_c(\B)=\Psi_G(C_c(G))B$. The image  commutes with left multiplication of elements in $\B_c^{G,\beta}$, 
the fixed-point algebra with compact supports~\eqref{eq-BcG},
which is a dense subalgebra of $\A_\mu\cong \K(\F_\mu(\B))$. We therefore may regard 
$\Psi_{\F_\mu(\B)}$ as a nondegenerate \Star{}homomorphism $\Psi_{\A_\mu}:C_0(X)\to Z\M(\A_\mu)$. 
Note that on elements $m$ in the dense subalgebra $\B_c^{G,\beta}$ of $\A_\mu$ we simply have
\begin{equation}\label{C0(X)-action}
\Psi_{\A_\mu}(f)m=\Phi_X(f)m,
\end{equation}
with multiplication given inside $\M(\B)$. 

Items (1) and (2) follow immediately from the functoriality of Landstad duality as spelled out in Proposition~\ref{Landstad-functorial}, and item (3) then follows from the $C_0(X)$-linearity of $\F_\mu(\B)$.
\end{proof}

As an immediate consequence of the above theorem together with Lemma~\ref{lem-Morita-bundles}, we now get

\begin{corollary}\label{cor-continuous1}
Let $(\B, \beta,\Phi)$ and $(\A_\mu, X, q_x^\A:\A_\mu\to A_\mu^x)$ be as in Theorem~\ref{thm-bundle1}. Then $(\A_\mu, X, q_x^\A:\A_\mu\to A_\mu^x)$ is faithful (resp. continuous, resp. 
upper semicontinuous) if and only if the same holds for the crossed product bundle 
$(\B\rtimes_{\beta,\mu}G, X, q_x\rtimes_\mu G:\B\rtimes_\mu G\to B_x\rtimes_{\mu}G)$.
\end{corollary}

\begin{example}\label{ex-trivialbundle1}
Let  $(C_0(X,B), \beta, \Phi)$ be as in Example~\ref{eq-trivial} above.
Since the trivial bundle $\B=C_0(X,B)$ is always continuous, it follows from Theorem~\ref{thm-bundle1} together with Proposition~\ref{prop-crossed-product-bundle}, that 
the  bundle $(\A_\mu, X, q_x^\A:\A_\mu\to A_\mu^x)$ is always upper semicontinuous when $\rtimes_\mu$ is an exact crossed-product functor, and that the  bundle  $(\A_r, X, q_x^\A:\A_r\to A_r^x)$ corresponding to the reduced crossed product $\rtimes_r$
is a continuous bundle whenever $G$ is an exact group.
\end{example}

Suppose now that $\Sigma=(X\times\T\into \G\onto X\times G)$ is a continuous family of  twists $\sigma_x=(\T\into G_x\onto G)$
as in Definition~\ref{defn-continuous-family-twists}. As before, we shall denote the elements of $\G$ 
by pairs $(x,\tig)$ with $\tig\in G_x$.
As in \eqref{eq-groupoid-iota} let  $C_0(\G, \iota)$ be the set of functions 
$f\in C_0(\G)$ which satisfy the relation $f(x,\tig z)=\bar{z}f(x,\tig)$ for all $(x,\tig)\in \G, z\in \T$.
Then $C_0(\G, \iota)$ becomes a $C_0(X\times G)-C_0(X\times G)$ equivalence bimodule with respect to the 
canonical left and right actions of $C_0(X\times G)$ by pointwise multiplication, and left and right inner products defined by
\begin{equation}\label{eq-XxG-inner}
_{C_0(X\times G)}\braket{\xi}{\eta}(x,g):=\xi(x,\tig)\overline{\eta(x,\tig)}\quad \text{and}\quad
\braket{\xi}{\eta}_{C_0(X\times G)}(x,g):=\overline{\xi(x,\tig)}\eta(x,\tig).
\end{equation}

Now let  $L^2(\G,\iota)$ be the completion of $C_c(\G,\iota)$ as a Hilbert $C_0(X)$-module as in \eqref{eq-L2Giota}. As observed there, the right regular representation $\tilde\rho:\G\car L^2(\G,\iota)$ induces a well-defined $C_0(X)$-linear action  $\alpha:G\car \K(L^2(\G,\iota))$ 
which induces  actions $\alpha^x:=\Ad\tilde\rho^x\car \K(L^2(G_x,\iota))$ on each fibre $\K(L^2(G_x,\iota))$ as defined in \eqref{eq-reptilderho}.

Now, given a weak $G\rtimes G$-algebra $(B,\beta,\phi)$, let us denote by $\E_\Sigma(C_0(X,B))$ the right Hilbert 
$B\otimes\K(L^2(\G,\iota))$ module defined as
\begin{equation}\label{eq-EalphaC0XB}
\E_\Sigma(C_0(X,B)):=\big(C_0(\G,\iota)\otimes_{C_0(X\times G)} \contz(X,B)\big)\otimes_{C_0(X)}L^2(\G,\iota)^*.
\end{equation}
Here, the second tensor product is understood as a balanced tensor product over $\contz(X)$ in the following sense: given $\E$ a Hilbert module over a $\contz(X)$-algebra $C$ and $\F$ a Hilbert module over a $\contz(X)$-algebra $D$, we define $\E\otimes_{\contz(X)}\F$ as the quotient of the external tensor product $\E\otimes\F$, which is a Hilbert module over $C\otimes D$, by the closed subspace generated by all differences $\xi\cdot f\otimes\eta-\xi\otimes \eta\cdot f$ with $\xi\in \E$, $\eta\in \F$, $f\in \contz(X)$. Alternatively, this can be viewed as $\E\otimes_{\contz(X)}\F\cong (\E\otimes\F)_{C\otimes D}(C\otimes_{\contz(X)}D)$, where $C\otimes_{\contz(X)}D$ is the balanced tensor product of $\contz(X)$-algebras.

Again, the module in~\eqref{eq-EalphaC0XB} is `fibred' over $X$ with fibres 
$\E_{x}(B):=\E_{\sigma_x}(B)$ for all $x\in X$. 
Let $\tilde\rt$ denote the right translation action of $\G$ on $C_0(\G,\iota)$.
Precisely as in the case of a deformation by a twist $\sigma$ over $G$, one  checks that 
the diagonal action $\tilde\rt\otimes_{C_0(X\times G)}\beta\otimes_{C_0(X)} {\tilde\rho}^*$ of $\G$ factors through a well-defined  
$C_0(X)$-linear action $\gamma_\Sigma:G\car \E_\Sigma(C_0(X,B))$ which  induces the action $\gamma_x=\gamma_{\sigma_x}$ of 
Proposition~\ref{prop-Borelcocycle} on each fibre $\E_{x}(B)$. 

Let $\B_\Sigma:=\K(\E_{\Sigma}(C_0(X,B)))$ denote the algebra of compact operators equipped with the $G$-action 
$\beta_\Sigma:=\Ad\gamma_{\Sigma}$. The left action of $C_0(X\times G)$ on $C_0(\G,\iota)$ induces 
a left action, say $\Phi_\Sigma$,  of $C_0(X\times G)$ on $\E_\Sigma(C_0(X,B))$, and hence on $\B_\Sigma$. 
In particular, since the left and right 
actions of $C_0(X)$ on  $\E_{\Sigma}(B)$ commute, it follows that the image of the `restriction' $\Phi_X$ of $\Phi_\Sigma$ to  $C_0(X)$ lies in $Z\M(\B_\Sigma)$, making $\B_\Sigma$ into a $C_0(X)$-algebra. 
Since the $C_0(X)$-algebra $\K(L^2(\G,\iota))$ is a continuous bundle of compact operators $\K(L^2(G_x,\iota))$ 
(as it is a continuous-trace algebra), it follows that  $B\otimes \K(L^2(\G,\iota))$ is a continuous bundle of \cstar{}algebras 
with fibres $B\otimes\K(L^2(G_x,\iota))$.
It follows then from Lemma \ref{lem-Morita-bundles} that $(\B_\Sigma, \beta_\Sigma,\Phi_\Sigma)$ is a continuous family of weak $G\rtimes G$-algebras with fibres 
$(B_{\sigma_x}, \beta_{\sigma_x},\phi_{\sigma_x})$ for all $x\in X$.
As a direct consequence of this discussion together with Corollary~\ref{cor-continuous1} we obtain

\begin{theorem}\label{thm-bundle-actions}
Let $(B,\beta,\phi)=(A\rtimes_\delta \hatG, \widehat{\delta}, j_{C_0(G)})$ for some $\mu$-coaction $(A,\delta)$ with $\rtimes_\mu$ a duality crossed-product functor. Let 
$\Sigma=(X\times \T\into \G\onto X\times G)$ be a twist over $X\times G$,
let $(\B_\Sigma, \beta_\Sigma,\Phi_\Sigma)$ be as above
and let  $(\A_{\mu}^\Sigma, \delta_\mu^\Sigma)$ be the $\mu$-coaction associated to $(\B_\Sigma, \beta_\Sigma,\Phi_\Sigma)$ via Landstad duality.
Then the following hold:
\begin{enumerate}
\item The triple $(\B_\Sigma, \beta_\Sigma,\Phi_\Sigma)$ is a \emph{continuous family of weak $G\rtimes G$-algebras}  over $X$ with fibres 
$(B_{\sigma_x}, \beta_{\sigma_x}, \phi_{\sigma_x})$ for all $x\in X$.
 \item If $\rtimes_\mu$ is an exact crossed-product functor, then $(\A_\mu^\Sigma, X, q_x^\Sigma:\A^\alpha_\mu\to A^{\sigma_x}_\mu)$ is an upper semicontinuous bundle of \cstar{}algebras over $X$.
    \item If $G$ is an exact group, then $(\A_r^\Sigma, X, q_x^\Sigma:\A^\Sigma_r\to A^{\sigma_x}_r)$ is a continuous bundle of \cstar{}algebras over $X$.
\end{enumerate}
\end{theorem}

For later use, we also need to state the following consequence of the above discussions. Since it requires compatibility of $\rtimes_\mu$ with respect to Morita equivalences, it requires $\rtimes_\mu$ to be a {\em correspondence} crossed-product functor.

\begin{proposition}\label{prop-stabilzation}
    Let $(B,\beta,\phi)=(A\rtimes_\delta \hatG, \widehat{\delta}, j_{C_0(G)})$,  $\Sigma=(X\times \T\into \G\onto X\times G)$,
    $(\B_\Sigma, \beta_\Sigma,\Phi_\Sigma)$ and $(\A_{\mu}^\Sigma, \delta_\mu^\Sigma)$
    be as in Theorem~\ref{thm-bundle-actions},  for some correspondence crossed-product functor $\rtimes_\mu$.
    Let $\alpha:G\car \K(L^2(\G,\iota))$ be the action corresponding to $\Sigma$ 
    as in Proposition~\ref{prop-groupoid-to-action}.  
    Then  the bundle of \cstar{}algebras
$(\A_\mu^\Sigma, X, q_x^\Sigma:\A^\Sigma_\mu\to A^{\sigma^x}_\mu)$ is Morita equivalent to the 
the bundle of $\mu$-crossed products obtained from the $C_0(X)$-linear action 
$\beta\otimes \alpha$ on $B\otimes \K(L^2(\G,\iota))$ as in Example~\ref{ex-crossed-bundle}.
\end{proposition}
\begin{proof}
    This follows from composing the Morita equivalence $\F_\mu(\B)$ of Theorem~\ref{thm-bundle1}, 
    which provides a Morita equivalence between $(\A_\mu^\Sigma, X, q_x^\Sigma:\A^\Sigma_\mu\to A^{\sigma^x}_\mu)$
    and the bundle of $\mu$-crossed products $(\B\rtimes_{\beta,\mu}G, X, q_x\rtimes_\mu G:\B\rtimes_{\beta,\mu} G\to B_x\rtimes_{\beta^x,\mu} G)$,
    with the $\mu$-crossed product 
    $\E_\Sigma(C_0(X,B))\rtimes_{\gamma,\mu}G$, which provides a Morita equivalence between the bundle 
    $(\B\rtimes_{\beta,\mu}G, X, q_x\rtimes_\mu G:\B\rtimes_{\beta,\mu} G\to B_x\rtimes_{\beta^x,\mu} G)$
    and the bundle $(B\otimes \K(L^2(\G,\iota)))\rtimes_{\beta\otimes\alpha,\mu}G$ of $\mu$-crossed products over $X$ with fibres  $(B\otimes \K(L^2(G_x,\iota)))\rtimes_{\beta\otimes\alpha^x,\mu} G$.
\end{proof}

Of course, an obvious choice for a base space $X$ in Theorem~\ref{thm-bundle-actions} 
should be the space $\Twist(G)\cong H^2(G,\T)\cong  \brg$ itself, since it provides our deformation parameters $[\sigma]$ (resp. $[\om]$, resp $[\alpha]$). As discussed in 
Section~\ref{sec:twists}, in general, there is no  obvious way to equip $\Twist(G)\cong H^2(G,\T)$ with a locally compact Hausdorff topology. However, the situation becomes quite nice in the presence of a representation group $Z\into H\onto G$ for $G$ as in Definition~\ref{defn-splitting}. In that case, the group $\Twist(G)$ is isomorphic to the dual group $\widehat{Z}$ via the transgression map $\tg:\widehat{Z}\to \Twist(G):\chi\mapsto [\sigma_\chi]$ as constructed in \eqref{eq-transgression}. Using Proposition~\ref{prop-groupext}, any representation group 
$Z\into H\onto G$ for $G$ provides a canonical twist $\Sigma_H=(\widehat{Z}\times\T\into \G_H\onto \widehat{Z}\times G)$
with fibres $\sigma_\chi=(\T\into G_{\chi}\onto G)$. Thus, as a consequence of the above, 
we can apply Theorem~\ref{thm-bundle-actions}
 to obtain a bundle of deformed algebras $\chi\mapsto A^{\sigma_\chi}_{\mu}$ over the base $\widehat{Z}\cong \Twist(G)$ with the continuity properties as spelled out in Theorem~\ref{thm-bundle-actions}. 
If $G$ is discrete and amenable, this bundle has been constructed before by Raeburn in \cite{Rae}.
Moreover, if $G$ is second countable and admits a representation group
$Z\into H\onto G$, then it follows from Corollary~\ref{cor-contmapH2} that every continuous map $\varphi:X\to H^2(G,\T)$ gives rise to a 
Twist $\Sigma_\varphi$ over $X\times G$ with fibres $[\sigma_\varphi(x)]$ and hence we obtain a bundle of deformed algebras 
$x\mapsto A^{\varphi(x)}_\mu$ with the continuity properties as in Theorem~\ref{thm-bundle-actions} above.

\section{K-theoretic results}\label{sec-K-theory}

It has been shown in \cite{BNS}*{Theorem 5.1}, using some ideas developed in \cites{ELPW, ENOO}, 
that for groups $G$ satisfying the Baum-Connes conjecture with coefficients, the 
$K$\nb-theory of the reduced  deformed algebras $A_r^\om$ only depends on the homotopy class of  $\om\in Z^2(G,\T)$.
In this section we want to discuss similar results for our deformation approach.
Let us first  recall some well-known results about the $K$\nb-theory of crossed products and the Baum-Connes conjecture. As usual when dealing with $KK$-theory, in this section we shall always assume that $G$ is {\em second countable} and all \cstar{}algebras are separable (with the obvious exceptions, like multiplier algebras).

If $\beta:G\car B$ is an action of a second countable group $G$ on a separable \cstar{}algebra $B$, then the 
{\em Baum-Connes conjecture with coefficients} 
asserts that 
a certain {\em assembly map}
\begin{equation}\label{eq-assembly}
\mu_B: K_*^{\top}(G;B)\to K_*(B\rtimes_rG)
\end{equation} 
is an isomorphism. We refer to \cite{BCH} for the construction of 
the assembly map. 
Although the conjecture turned out to be false in general (see \cite{HLS}), it has been shown to be true for large classes of groups. In 
particular, it has been shown for all a-$T$-menable groups in the sense of Gromov, which cover all amenable groups but also the free groups $\mathbb F_n$ as well as the groups $\SL(2,\R), \SU(n,1), \SO(n,1)$ and many others. In fact, it has been shown by Higson and Kasparov in \cite{HK} that all those groups satisfy the {\em strong Baum-Connes conjecture}.
In short, this means that there exists a proper $G$-space $X$ and an $X\rtimes G$-algebra $\mathcal A$ (in the {\em strong sense} of Kasparov that 
the structure map $\phi:C_0(X)\to \M(\mathcal A)$ takes image in the center $Z\M(\mathcal A)$) such that 
$\mathcal A$ is $KK^G$-equivalent to $\C$. It has been shown by Kasparov and Tu that the strong Baum-Connes conjecture implies the Baum-Connes conjecture with coefficients
(e.g., see \cite{Tu}*{Theorem 2.2}) and Tu shows that all groups satisfying the strong Baum-Connes conjecture are also $K$-amenable in 
the sense of Cuntz and Julg-Valette (see \cites{Cu, JV}). 

In what follows, by a homotopy between two
 actions $\beta_0,\beta_1:G\car B$ we understand 
a fibrewise action $\beta:G\car C([0,1], B)$ which  connects  $\beta_0$ and $\beta_1$ via evaluation 
at $0$ and $1$, respectively. The following result has been shown in \cite{ENOO} based on earlier work in \cite{CEOO} and \cite{MN}. Since we need a slight modification, we formulate

\begin{proposition}\label{prop-K-theory-crossed}
Suppose that $\beta:G\car C([0,1], B)$ is a homotopy of actions connecting  $\beta_0$ and $\beta_1$ as above. Assume further that $G$ satisfies the Baum-Connes conjecture with coefficients.
\begin{enumerate}
\item  Suppose  that for every compact subgroup 
$L$ of $G$ the canonical evaluation maps  $\ev^i_*:K_*(C([0,1],B)\rtimes_{\tilde\beta} L)\to K_*(B\rtimes_{\beta^i}L)$ 
are isomorphisms for $i=0,1$. Then the evaluation maps 
$\ev^i_*:K_*(C([0,1],B)\rtimes_{\tilde\beta,r} G)\to K_*(B\rtimes_{\beta^i,r}G)$ are isomorphisms as well.
In particular
$K_*(B\rtimes_{\beta^0,r}G)\cong K_*(B\rtimes_{\beta^1,r}G)$.
\item Suppose all assumptions of (1). If,  in addition, $G$ is $K$-amenable, then $K_*(B\rtimes_{\beta^0,\mu}G)\cong K_*(B\rtimes_{\beta^1,\mu}G)$ for every correspondence crossed-product functor $\rtimes_\mu$.
\item If $G$ satisfies the strong Baum-Connes conjecture and the evaluation maps 
$\ev^i:C([0,1],B)\rtimes_{\tilde\beta}L\to B\rtimes_{\beta^i}L$ in (1)  are  $KK$-equivalences, then 
the same holds for the evaluation maps 
$\ev^i:C([0,1],B)\rtimes_{\tilde\beta,\mu} G\to B\rtimes_{\beta^i,\mu}G$ for every correspondence crossed-product functor $\rtimes_\mu$. 
In particular, we then get  $B\rtimes_{\beta^0,\mu}G\sim_{KK} B\rtimes_{\beta^1,\mu}G$.
\end{enumerate}
\end{proposition}
\begin{proof} Item  (1) follows  from item (i) in \cite{ENOO}*{Proposition 2.1} applied to the evaluation maps $\ev_i:C([0,1],B)\to B$ at $i=0,1$. 

Item (2) then follows from  \cite{BEW}*{Theorem 6.6}, which states that for every $K$-amenable group $G$ and every 
action $\beta:G\car B$ the canonical  quotient maps
\begin{equation}\label{eq-KK}
B\rtimes_{\beta,\max}G\stackrel{q_{\max,\mu}}{\onto} B\rtimes_{\beta,\mu}G\stackrel{q_{\mu,r}}{\onto} B\rtimes_{\beta,r}G
\end{equation}
are $KK$-equivalences if $\rtimes_\mu$ is a correspondence crossed-product functor for $G$.

Item (3) follows from item (iii) in  \cite{ENOO}*{Proposition 2.1} combined  with \cite{BEW}*{Theorem 6.6}.
 \end{proof}
 
\begin{remark}\label{rem-torsionfree}
Note that the requirements on the compact subgroups $L$ of $G$ in (1)--(3) of the above proposition are always satisfied if the only compact subgroup of $G$ is the trivial group, e.g., if $G$ is a discrete torsion free group, or if $G$ is a simply connected solvable Lie group. In general, the requirement on the compact subgroups is not very easy to check.
\end{remark}
 
Suppose now that $\beta:G\car C([0,1],B)$ is a homotopy of actions as above. Suppose further that
$\varphi:C_0(G)\to\M(B)$ is a nondegenerate \Star{}homomorphism making all triples $(B,\beta_t, \phi)$
into weak $G\rtimes G$-algebras. We obtain a bundle of weak $G\rtimes G$-algebras 
$(C([0,1], B), \beta, \Phi)$ over $[0,1]$ with $\Phi=\id_{C[0,1]}\otimes \varphi:C_0([0,1]\times G)\to \M(C([0,1], B)$,
connecting the weak $G\rtimes G$-algebras $(B,\beta^0,\phi)$ and $(B,\beta^1,\phi)$ as in 
Definition~\ref{def-semicontbundle}. Then a combination of Proposition~\ref{prop-K-theory-crossed} 
with Theorem~\ref{thm-bundle1} gives

\begin{corollary}\label{cor-K-theory-crossed}
Let $(C([0,1], B), \beta, \Phi)$ be as above and, for any duality crossed-product functor $\rtimes_\mu$, 
let $(\A_\mu, [0,1], q_t^\A:\A_\mu\to A_\mu^t)$ denote the corresponding bundle of deformed algebras as 
in Theorem~\ref{thm-bundle1}.  Suppose  $G$  satisfies the Baum-Connes conjecture with coefficients. 
\begin{enumerate}
\item  If $\beta:G\car C([0,1],B)$ satisfies the assumptions of item (1) in Proposition~\ref{prop-K-theory-crossed}, then 
$K_*(A_r^0)\cong K_*(A_r^1)$.
\item If $\beta:G\car C([0,1],B)$ satisfies the assumptions of item (2) in Proposition~\ref{prop-K-theory-crossed}, then $K_*(A_\mu^0)\cong K_*(A_\mu^1)$ for every correspondence crossed-product functor $\rtimes_\mu$.
\item If $\beta:G\car C([0,1],B)$ satisfies the assumptions of item (3) in Proposition~\ref{prop-K-theory-crossed}, then $A_\mu^0$ is $KK$-equivalent to $A_\mu^1$ for every correspondence crossed-product functor $\rtimes_\mu$.
\end{enumerate}
\end{corollary}

Recall that items (1) and (2) are always satisfied if $G$ has no non-trivial compact subgroups. The corollary then covers our general deformation procedure as described in Section~\ref{sec:Landstad}.
For deformations by twists $\sigma=(\T\into G_\sigma\onto G)$ with corresponding actions $\alpha:G\car \K(\H)$ 
we obtain the  following more satisfying result:

\begin{corollary}\label{cor-Ktheory-actionsK}
Suppose   that $G$ is a second countable group which 
satisfies the Baum-Connes conjecture with coefficients and let $(A,\delta)$ be a separable $\mu$-coaction for some 
correspondence crossed-product functor $\rtimes_\mu$. Let 
$(B,\beta,\phi):=(A\rtimes_\delta \hatG, \widehat{\delta}, j_{C_0(G)})$ be the corresponding weak $G\rtimes G$-algebra and let $\Sigma=([0,1]\times\T\into \G\onto [0,1]\times G)$ be a twist over $[0,1]\times G$.
Then the following are true
\begin{enumerate}
    \item  If $\rtimes_\mu=\rtimes_r$ and if $(\A_r^\Sigma, [0,1], q_t^\Sigma:\A_r^\Sigma\to A_r^{\sigma_t})$ is the corresponding bundle as in Theorem~\ref{thm-bundle-actions}, then $q^\A_{t,*}:K_*(\A_r^{\Sigma})\to K_*(A_{r}^{\sigma_t})$
    is an isomorphism for all $t\in [0,1]$. In particular, we have
    $K_*(A_r^{\sigma_0})\cong K_*(A_r^{\sigma_1})$.
    \item If, in addition, $G$ is $K$-amenable, then (1) holds with $\rtimes_r$ replaced by any correspondence crossed-product functor $\rtimes_\mu$.
    \item If $G$ satisfies the strong Baum-Connes conjecture, then all quotient maps
    $q_t^\A: \A_\mu^{\Sigma}\to A_\mu^{\sigma_t}$ are $KK$-equivalences. In particular $A_\mu^{\sigma_0}\sim_{KK} A_\mu^{\sigma_1}$.
\end{enumerate}
\end{corollary}
\begin{proof} Let $\alpha:G\car \K(L^2(\G,\iota))$ be the corresponding action. Since everything in sight is separable, 
we may assume (up to $X\rtimes G$-equivariant Morita equivalence) that $\K(L^2(\G,\iota))=C_0(X,\K(\H))$.
It follows from Proposition~\ref{prop-stabilzation}
that the bundle $(\A_\mu^\Sigma, [0,1], q_t^\Sigma:\A_\mu^\Sigma\to A_\mu^{\sigma_t})$ is Morita equivalent to the crossed-product bundle over $[0,1]$ with fibre maps
$$q_t\rtimes_\mu G: \big(B\otimes C([0,1], \K(\H))\big)\rtimes_{\beta\otimes \alpha,\mu} G\to (B\otimes \K(\H))\rtimes_{\beta\otimes \alpha_t,\mu}G.$$
Thus we only need to check the corresponding results for the crossed product bundle, for which we can apply 
Proposition~\ref{prop-K-theory-crossed} (with $B\otimes \K(\H)$ instead of $B$). 
For this let $L$ be any compact subgroup of $G$. In  the proof of \cite{ELPW}*{Theorem 1.7} it is shown that 
the restriction $\alpha|_L:L\car C([0,1],\K(\H))$ is exterior equivalent to the action 
$\id_{C[0,1]}\otimes \alpha_0|_L$, from which it follows that $C([0,1],\K(\H))\rtimes L\cong C([0,1], \K(\H)\rtimes_{\alpha_0}L)$.
But then we also get 
$$\big(B\otimes C([0,1], \K(\H))\big)\rtimes_{\beta\otimes \alpha}L\cong C\big([0,1], (B\otimes\K(\H))\rtimes_{\beta\otimes \alpha_0} L\big).$$
Hence all evaluation maps are $KK$-equivalences, and the results follow from 
Proposition~\ref{prop-K-theory-crossed}.
\end{proof}

Note that in the above corollary we can always start with a $C([0,1])$-linear action $\alpha:G\car C([0,1],\K(\H))$ and then 
use the corresponding twist $\Sigma_\alpha$ as in Lemma~\ref{lem-groupbundle}.

\begin{example}\label{ex-SL2}
For the first example we let $G=\PSL_2(\R)$, the quotient of $\SL_2(\R)$ by its center $\{\pm I_2\}$.
It follows from \cite{BEW2}*{Example 4.11} that there exist infinitely many distinct correspondence crossed-product functors for $G$. 

As mentioned in Example~\ref{ex-semi-simple} the universal covering group $H$ of $G$ provides  the unique representation group
$$\Z\into H\onto G$$
for $G$.
By Theorem~\ref{thm-bundle-actions}, for any 
$\mu$-coaction $(A,\delta)$ of $G$ and for any fixed correspondence crossed-product functor $\rtimes_\mu$ we obtain a  field 
of deformed $\mu$-coactions $(A^{\om}_\mu,\delta^{\om}_\mu)$ over 
$\widehat{\Z}\cong \T$. Since $G=\PSL(2,\R)$
has the Haagerup property (e.g., see 
\cite{CCJJV}), and since all cocycles are homotopic, it follows from 
Corollary~\ref{cor-Ktheory-actionsK} that 
all deformed algebras $A^{\om}_\mu$ are 
$KK$-equivalent to $A$, and they are also all 
$KK$-equivalent to $A_{\max}^{\omega}$ and $A_r^{\omega}$. Note that the group $G$ in this example is $K$-amenable (which explains that 
$A_{\max}^{\omega}$ and $A_r^{\omega}$ are $KK$-equivalent), but $G$ is not amenable.

In particular, if we start with $A=C_\mu^*(G):=\C\rtimes_\mu G$ with 
dual coaction $\delta$, then $A_\mu^\om=C_\mu^*(G,\om)$ is the $\mu$-twisted group algebra of $G$ by Proposition~\ref{pro:dual-coactions-deformation} above. 
Looking at the cocycle $\om_{-1}$ corresponding to $-1\in \T\cong \widehat{\Z}$ 
(which is the cocycle related to the central extension 
$1\to \{\pm I_2\}\to \SL_2(\R)\to \PSL_2(\R)\to 1$), we get the following curiosity: we observed in 
\cite{BEW:amenable}*{Example 5.26} that the maximal and reduced 
twisted group algebras $C^*_\max(G,\om_{-1})$ and $C^*_r(G,\om_{-1})$ coincide. 
By our results, $C^*_{\max}(G)$ deforms into $C^*_\max(G,\om_{-1})$ via a continuous path of cocycles
and $C^*_r(G)$ deforms into $C^*_r(G,\om_{-1})=C^*_{\max}(G, \om_{-1})$. But 
$C^*_{\max}(G)$ and $C^*_r(G)$ are not isomorphic! 
\end{example}

\begin{example}\label{ex-Kamenable}
To see other examples of groups for which the results of this paper give new information consider the 
semi-direct product $G:=\R^2\rtimes \SL_2(\R)$. 
Since $\SL_2(\R)$ has the Haagerup property
it satisfies the strong Baum-Connes conjecture by \cite{HK}. 
It follows then from \cite{JV}*{Proposition 3.3} that 
$G$ is $K$-amenable and 
\cite{CE:twistedKK}*{Theorem 7.1} implies that $G$ satisfies the Baum-Connes conjecture with coefficients.
It  follows  from  \cite{BEW2}*{Example 4.11} (via iterated crossed products) that there are infinitely many distinct correspondence crossed-product functors for $G$.

Now let $\theta\in \R$ be any real parameter. Let $\om_\theta\in Z^2(\R^2,\T)$ 
be  defined by 
$$\om_\theta\left(\left(\begin{smallmatrix}x_1\\ x_2\end{smallmatrix}\right),\left(\begin{smallmatrix}y_1\\ y_2\end{smallmatrix}\right)\right)=e^{i\pi(x_2y_1-x_1y_2)}.$$
It follows then from \cite{ELPW}*{Lemma 2.1} that there are $2$-cocycles $\tilde\om_\theta\in Z^2(G,\T)$ defined by 
$$\tilde\om_\theta\left(\big(\left(\begin{smallmatrix}x_1\\ x_2\end{smallmatrix}\right), g\big),
\big(\left(\begin{smallmatrix}y_1\\ y_2\end{smallmatrix}\right), h\big)\right)
=\om_\theta\left(\left(\begin{smallmatrix}x_1\\ x_2\end{smallmatrix}\right),g\cdot \left(\begin{smallmatrix}y_1\\ y_2\end{smallmatrix}\right)\right).$$
Since all these cocycles are homotopic to the trivial cocycle (use Remark~\ref{rem-cont}), it follows that for any given coaction 
$(A,\delta)$ of $G$ 
and for any given correspondence crossed-product functor $\rtimes_\mu$ for $G$,
the $K$-theory groups $K_*(A_\mu^{\tilde\om_\theta})$ all agree, and they agree with $K_*(A)$ if 
$\delta$ is a $\nu$-coaction for some correspondence crossed-product functor $\rtimes_\nu$ (e.g., if $\delta$ is maximal or normal). 

Restricting the cocycles $\tilde\om_\theta$ to $H=\Z^2\rtimes\SL(2,\Z)$ provides a similar example for the discrete group $H$.
\end{example}

In \cite{ENOO} the authors introduced some stronger $K$-theoretic properties for $C_0(X)$-algebras $\A$ (which in our notion 
just means upper semicontinuous  bundles of \cstar{}algebras
$(\A, X, q_x:\A\to A_x)$ over $X$), which are potentially useful to do $K$-theoretic computations of the section algebra $\A$
(e.g., see \cite{ENOO}*{Section 4}, where the authors study a $K$-theoretic analogue of the Leray-Serre spectral sequence for such bundles).
The relevant notions for us are the notions of $K$- and $KK$-fibrations, as defined in \cite{ENOO}*{Definition 1.3} (where we use 
$K$-theory in item (i) of that definition).
We leave it to the interested reader to check out \cite{ENOO}*{Definition 1.3} and to reformulate it to the more general notion of bundles used here. Note that the advantage of allowing more general 
notions lies in the fact that we do not need to assume exactness of our crossed-product functors $\rtimes_\mu$ in our statements, which guarantees the relevant crossed-product bundles to be semicontinuous by Proposition~\ref{prop-crossed-product-bundle}.

\begin{proposition}\label{prop-KandKK}
Let $\Sigma=(X\times\T\into \G\onto X\times G)$ be a twist over $X\times G$ with $X\times G$ second countable, and let $(B,\beta,\varphi)$ be a separable weak $G\rtimes G$-algebra. For any correspondence crossed-product functor 
$\rtimes_\mu$ let $(\A_\mu^{\Sigma}, X, q_x^{\Sigma}:\A_\mu^\Sigma\to A_\mu^{\sigma_x})$ be the corresponding 
bundle of `deformed' algebras. Then the following hold:
\begin{enumerate}
    \item If $G$ satisfies the Baum-Connes conjecture with coefficients, then 
    $(\A_r^{\Sigma}, X, q_x^{\Sigma}:\A_r^\Sigma\to A_r^{\sigma_x})$ is a $K$-fibration.
    \item If $G$ satisfies the strong Baum-Connes conjecture, then 
    $(\A_\mu^{\Sigma}, X, q_x^{\Sigma}:\A_\mu^\Sigma\to A_\mu^{\sigma_x})$ is a $KK$-fibration. 
\end{enumerate}
\end{proposition}
\begin{proof}
As in proof of Corollary~\ref{cor-Ktheory-actionsK} we obtain a corresponding action 
$\alpha:G\car C_0(X,\K(\H))$ such that $\Sigma\cong \Sigma_\alpha$. 
By  Proposition~\ref{prop-stabilzation}, the bundle $(\A_\mu^{\Sigma}, X, q_x^{\Sigma}:\A_\mu^\Sigma\to A_\mu^{\sigma_x})$
is $C_0(X)$-linearly Morita equivalent to the $\mu$-crossed-product bundle provided by the fibrewise action 
$\beta\otimes \alpha:G\car B\otimes C_0(X,\K(\H))\cong C_0(X, B\otimes \K(\H))$.
Since Morita equivalences are in particular $KK$-equivalences,
it suffices to show that under the given assumptions, the latter is a $K$- (resp. $KK$-)fibration.

The proof is analogous to the proof of \cite{ENOO}*{Corollary 2.4} (and also to the proof of Corollary~\ref{cor-Ktheory-actionsK} above), using that for any standard $p$-simplex 
$\Delta^p$ and any fibrewise action $\alpha:G\car C(\Delta^p,\K(\H))$ the restriction of $\alpha$ to 
any compact subgroup $L$ is exterior equivalent to an action of the form $\id_{C(\Delta^p)}\otimes \alpha_v$
for some fixed $v\in \Delta^p$. As in the proof of Corollary~\ref{cor-Ktheory-actionsK}, this easily implies 
that the quotient map $q_v: \big(B\otimes C(\Delta^p,\K(\H))\big)\rtimes_{\beta\otimes \alpha} L\to 
(B\otimes \K(\H))\rtimes_{\beta\otimes \alpha_v}L$ is a $KK$-equivalence. The result then follows from 
\cite{ENOO}*{Proposition 2.2} and the fact that 
the conditions in (2) imply that $G$ is $K$-amenable.
\end{proof}

For example, the field of deformations of Example~\ref{ex-Kamenable} is always a $K$-fibration, while the field studied in Example~\ref{ex-SL2} is a $KK$-fibration. The 
above result also applies to the continuous fields of deformed algebras over
$X=H^2(G,\T)$ (or any pullback along a continuous map $\varphi:X\to H^2(G,\T)$) 
as in Corollary~\ref{cor-contmapH2} if the group $G$ is smooth and satisfies the (strong) Baum-Connes conjecture. 



\vskip 0,5pc

\end{document}